\documentclass[12pt,twoside,reqno]{amsart}

\usepackage{amsmath, amssymb, amsthm, enumitem, esint}
\usepackage[normalem]{ulem}
\usepackage{soul}
\usepackage{xcolor}
\usepackage[hyperindex]{hyperref}
\usepackage{mathrsfs}

\usepackage{xr-hyper}   %
\newcommand{\refx}[1]{%
  \NoHyper\ref{P1-#1}\endNoHyper
}

\hypersetup{bookmarksdepth=3}

\newcommand\cross{\bgroup\markoverwith{\textcolor{red}{\rule[0.5ex]{2pt}{0.6pt}}}\ULon}
\newcommand\crosseq{\bgroup\markoverwith{\textcolor{red}{\rule[0.5ex]{4pt}{0.6pt}}}\ULon}

\newcommand{\lto}{\longrightarrow}

\newcommand{\cyl}{\textnormal{cyl}}

\DeclareMathOperator{\nullspace}{null}
\DeclareMathOperator{\Oval}{oval}
\DeclareMathOperator{\soliton}{soliton}

\newcommand{\MCF}{\mathsf{MCF}}

\newcommand{\sV}{\mathscr{V}}
\newcommand{\Qu}{\mathsf{Q}}
\newcommand{\IR}{\mathbb{R}}
\newcommand{\lb}{\linebreak[1]}

\newcommand{\IB}{\mathbb{B}}
\newcommand{\IS}{\mathbb{S}}

\newcommand{\DD}{\mathcal{D}}

\newcommand{\PP}{\mathcal{P}}
\newcommand{\UU}{\mathcal{U}}
\newcommand{\VV}{\mathcal{V}}

\newcommand{\MM}{\mathcal{M}}

\newcommand{\bx}{\mathbf{x}}
\newcommand{\by}{\mathbf{y}}
\newcommand{\bY}{\mathbf{Y}}
\newcommand{\bO}{\mathbf{0}}
\newcommand{\bq}{\mathbf{q}}
\newcommand{\bp}{\mathbf{p}}

\newcommand{\bz}{\mathbf{z}}

\newcommand{\eps}{\varepsilon}
\newcommand{\la}{\lambda}

\newcommand{\ov}[1]{\overline{#1}}
\newcommand{\td}[1]{\widetilde{#1}}

\DeclareMathOperator*{\osc}{osc}

\DeclareMathOperator{\spt}{spt}
\DeclareMathOperator{\Jac}{Jac}
\DeclareMathOperator{\rot}{rot}

\DeclareMathOperator{\loc}{loc}

\DeclareMathOperator{\id}{id}

\DeclareMathOperator{\bowl}{bowl}

\DeclareMathOperator{\supp}{supp}

\DeclareMathOperator{\spann}{span}

\DeclareMathOperator{\proj}{proj}

\DeclareMathOperator{\reg}{reg}
\DeclareMathOperator{\sing}{sing}

\newcommand{\EMPTY}[1]{}

\newtheorem{Theorem}[equation]{Theorem}
\newtheorem{Lemma}[equation]{Lemma}
\newtheorem{Corollary}[equation]{Corollary}
\newtheorem{Proposition}[equation]{Proposition}
\newtheorem{Claim}[equation]{Claim}

\theoremstyle{definition}
\newtheorem{Definition}[equation]{Definition}
\theoremstyle{remark}
\newtheorem{Remark}[equation]{Remark}

\numberwithin{equation}{section}

\title[Classification of ancient cylindrical mean curvature flows]{Classification of ancient cylindrical mean curvature flows and the Mean Convex Neighborhood Conjecture}
\author{Richard H Bamler and Yi Lai}
\address{Department of Mathematics, UC Berkeley, CA 94720, USA}
\email{rbamler@berkeley.edu}
\address{Department of Mathematics, UC Irvine, CA 92697, USA}
\email{ylai25@uci.edu}
\thanks{R.B. was supported by NSF grant DMS-2204364.
Y.L. was supported by NSF grant DMS-2506832.
This material is based upon work conducted at the Simons Laufer Mathematical Sciences Research Institute in Berkeley, California, during the Fall semester of 2024.}
\date{\today}

\begin{document}

\begin{abstract}
We resolve the Mean Convex Neighborhood Conjecture for mean curvature flows in all dimensions and for all types of cylindrical singularities. 
Specifically, we show that if the tangent flow at a singular point is a multiplicity-one cylinder, then in a neighborhood of that point the flow is mean-convex, its time-slices arise as level sets of a continuous function, and all nearby tangent flows are cylindrical.
Moreover, we establish a canonical neighborhood theorem near such points, which characterizes the flow via local models.
We also obtain a more uniform version of the Mean Convex Neighborhood Conjecture, which only requires closeness to a cylinder at some initial time and yields a quantitative version of this structural description.

Our proof relies on a complete classification of ancient, asymptotically cylindrical flows. 
We prove that any such flow is non-collapsed, convex, rotationally symmetric, and belongs to one of three canonical families: ancient ovals, the bowl soliton, or the flying wing translating solitons. 
Central to our method is a refined asymptotic analysis and a novel \emph{leading mode condition,} together with a new ``induction over thresholds'' argument.
In addition, our approach provides a full parameterization of the space of asymptotically cylindrical flows and gives a new proof of the existence of flying wing solitons.

Our method is independent of prior work and, together with our prequel paper, this work is largely self-contained.
\end{abstract}

\maketitle
\tableofcontents

\section{Introduction}
\subsection{Overview}
A standard method for studying singularities in geometric PDEs---originating in foundational work of Almgren, Federer, Fleming and Simon~\cite{Almgren_1966, Federer_book, Federer_Fleming, Simon_1996}---is to analyze solutions via tangent cones.
The guiding philosophy is that tangent cones capture the leading-order geometry near a singular point.
However, the information they provide is often coarse.
For example, tangent cones may fail to reveal precise asymptotics, the local topology near the singular point or the full set of nearby tangent cones, particularly when the tangent cone itself is singular.
The source of this limitation is often characterized by the distinction between \emph{tangent cones} and more flexible \emph{blow-up models:} tangent cones arise from rescalings about a fixed center, whereas blow-up models allow both the center and scale to vary---therefore they capture all local models of the solution at all scales.
Consequently, knowledge of tangent cones alone often does not determine all nearby blow-up behaviors.

This phenomenon is particularly relevant in mean curvature flow.
Although tangent flows (the parabolic analogues of tangent cones) play an essential role in understanding singularities, they often do not determine the full geometry of the flow near a singular point, including how singularities form and resolve precisely.
The \emph{Mean Convex Neighborhood Conjecture} of Ilmanen \cite[Problem~4]{Ilmanen_problems} and White addresses this limitation.
It predicts that if a mean curvature flow has a cylindrical tangent flow at a singular point, then the flow must be locally mean convex in a neighborhood of the singularity, and consequently, all nearby tangent flows must be cylindrical. 
To place this in context, mean convex mean curvature flows have been extensively studied and are now well understood.
Seminal work of Huisken \cite{Huisken_1984} established the behavior of fully convex and 2-convex flows, while White \cite{White_00,White_03} developed a deep general theory for globally mean convex flows.
Further progress is due to Huisken-Sinestrari \cite{HS99b,Huisken_Sinestrari_99,Huisken_Sinestrari_09} Brendle \cite{brendle2015sharp}, Brendle-Huisken \cite{Brendle_Huisken_16}, Kleiner-Haslhofer \cite{Haslhofer_Kleiner_17,Haslhofer_Kleiner_17_surgery} and Haslhofer-Hershkovits \cite{HaslhoferHershkovits_2018_LevelSetGeneralAmbient}.
These works describe the geometry of flows with \emph{globally} positive mean curvature and show that singularities are modeled on shrinking cylinders and evolve through higher-dimensional neckpinches.

In this paper we prove the \emph{Mean Convex Neighborhood Conjecture} in full generality. 
Our results therefore show that the same geometric and structural properties from the globally mean-convex theory hold \emph{locally} near every cylindrical singularity.
We also give a complete description of the singularity formation and resolution via a canonical neighborhood theorem, which is even new in the globally mean convex case.

The key ingredient in our resolution of the \emph{Mean Convex Neighborhood Conjecture} is a complete classification of ancient, asymptotically cylindrical mean curvature flows.
We show that every such ancient flow is non-collapsed, convex and rotationally symmetric and falls into one of three families: ancient ovals, the bowl soliton, or flying wing translating solitons due to Hoffman-Ilmanen-Mart\'in-White \cite{HoffmanIlmanenMartinWhite2019}---each possibly times a Euclidean factor.
This completes and unifies a long line of work around the classification of ancient flows initiated by Angenent-Daskalopoulos-Sesum \cite{ADS_2019}, Brendle-Choi~\cite{Brendle_Choi_2019}, and subsequently advanced by many others \cite{ADS_2020,Brendle_Choi_higher_dim,zhu2022so,ChoiHaslhoferHershkovits2023,DuHaslhofer2021,Choi_Haslhofer_Hershkovits_2022,Choi_Haslhofer_Hershkovits_White_22,  CHH_2023,Du_Haslhofer_2023,ChoiHaslhofer2024,ChoiHaslhoferHershkovits2024,Du_Haslhofer_2024,ChoiDuZhu2025,Du_Zhu_2025,ChoiDaskalopoulosDuHaslhoferSesum2025,ChoiHaslhoferHershkovits2025, ADS_2025,  CHH_2025, Choi_Haslhofer_2025}.
A major milestone was the full classification of ancient flows asymptotic to $1$-cylinders of the form $\IR \times \IS^{n-1}$, which resolved the Mean Convex Neighborhood Conjecture in the $1$-cylindrical setting~\cite{Choi_Haslhofer_Hershkovits_2022,Choi_Haslhofer_Hershkovits_White_22}.
Flows asymptotic to cylinders with a larger than 1-dimensional Euclidean factor (a.k.a. ``bubble sheets'') are substantially more delicate.
A more recent advance in this direction was the classification of cylindrical mean curvature flows in dimension~4  \cite{ChoiDaskalopoulosDuHaslhoferSesum2025,ChoiHaslhofer2024,ChoiHaslhoferHershkovits2025} under a non-collapsing condition, which relies on a mean-convexity assumption.

Most of these prior results---though highly ingenious---relied on delicate structural properties of the mean curvature flow equation and were therefore limited to the non-collapsed, convex, or $1$-cylindrical cases.
Since these results require mean-convexity \emph{a priori,} they cannot be used to deduce \emph{Mean Convex Neighborhood Conjecture.}

Our approach is new and independent of previous approaches and together with the prequel~\cite{Bamler_Lai_PDE_ODI} the present paper is essentially self-contained.
Because our techniques rely only on relatively coarse structural features of the equation, we are able to recover mean-convexity \emph{a posteriori,} which is what ultimately allows us to prove the \emph{Mean Convex Neighborhood Conjecture.}
Moreover, our theory yields a full description and parameterization of the model flows appearing in our classification, which is entirely based on an asymptotic mode analysis.
Specifically, we uniquely identify ancient ovals via the higher-order asymptotic expansion of the quadratic mode as $\tau \to -\infty$.
Likewise, we characterize flying wings by their exponentially decaying deviation from an ancient oval times a Euclidean factor as $\tau \to -\infty$.

The classification rests on two main ideas.
First, the PDE–ODI principle developed in our prequel~\cite{Bamler_Lai_PDE_ODI} initiates the analysis by producing high-order asymptotic expansions for ancient flows, notably without requiring any mean-convexity assumption a priori. 
However, these expansions only provide \emph{polynomial} control of the flows. To access \emph{exponentially} decaying differences---essential for distinguishing two flows---we enhance these estimates based on a new Harnack-type estimate.
The second idea establishes this Harnack-type estimate via a new notion called \emph{leading mode condition.}
It examines the difference between two ancient flows in regions where both are nearly cylindrical and determines whether this difference is dominated by a finite set of unstable modes. 
We establish the leading mode condition via a novel ``induction over thresholds'' argument, which allows us to iteratively lower a threshold above which the linear mode condition holds.
This induction step combines a propagation mechanism through cylindrical regions with a new stability estimate for the bowl soliton times a Euclidean factor.

\medskip
For further related work see 
\cite{CM12_GenericMCF_I,Colding_Ilmanen_Minicozzi,CM16_SingularSetGenericSingularities,SX22_GenericCylindricalSingularities,Bamler_Kleiner_mult1,CCS23_GenericInitialDataII,Zhu23_MeanConvexSelfShrinkers,CCMS24_GenericInitialData,CCMS24_LowEntropy,CMS25_LowEntropyII,CCMS25_RevisitingGenericR3,SWX25_PassingNondegenerate,SWX25_RegularityCylindricalSingularSets,Zhu25_CylindricalSelfShrinkers}.

\subsection{Statement of main results I: Classification of asymptotically cylindrical flows} \label{subsec_main_results_I}
To state our main results, we recall the following convenient definition.

\begin{Definition}
\label{Def_asymp_cyl} 
An \textbf{asymptotically $(n,k)$-cylindrical mean curvature flow,} for some integers $0 \leq k < n$, is an $n$-di\-men\-sion\-al, unit-regular, integral Brakke flow in $\IR^{n+1} \times (-\infty, T)$, for some $T \leq \infty$, whose tangent flow at infinity  is the multiplicity one round shrinking cylinder $\MM_{\cyl,t}^{n,k} = \IR^k \times (\sqrt{-t}\IS^{n-k})$, here $\IS^{n-k}$ denotes that round sphere with radius $\sqrt{2(n-k)}$.
In other words, we require that the parabolic rescalings $\la \MM$ converge locally smoothly to $\MM_{\cyl}^{n,k}$ as $\la \to 0$.
\end{Definition}

For convenience, we fix the axis of the limiting cylinder to be  $\IR^k\times\bO^{n+1-k}$, as any flow can be rotated into this standard form.
We remark that the condition in Definition~\ref{Def_asymp_cyl} can be replaced with an a priori weaker assumption due to the stability of cylinders \cite{Bamler_Lai_PDE_ODI,Colding_Minicozzi_12,Colding_Ilmanen_Minicozzi}.
Moreover, by \cite{Colding_Minicozzi_codim1}, our discussion extends to higher-codimension flows with minor modifications, though this is not strictly necessary: the higher-codimension case reduces to codimension~1 due to \cite{Colding_Minicozzi_codim1}.

Our main result is the following:

\begin{Theorem} \label{Thm_classification_simple}
Let $\MM$ be an asymptotically $(n,k)$-cylindrical mean curvature flow.
Then $\MM$ is smooth (possibly up to its extinction time), non-collapsed, convex and invariant under all ambient rotations that fix an axis parallel to $\IR^k \times \mathbf 0^{n-k+1}$.
Moreover, $\MM$ belongs to one of the following canonical families:
\begin{enumerate}[label=(\Alph*)]
\item \label{Thm_class_A} Round shrinking $(n,k)$-cylinders.
\item \label{Thm_class_B} Ancient ovals due to \cite{White_03, HaslhoferHershkovits2016, DuHaslhofer2021}, possibly times a Euclidean factor.
\item \label{Thm_class_C} An $(n-k+1)$-dimensional bowl soliton times $\IR^{k-1}$.
\item \label{Thm_class_D} The flying wing translating solitons due to Hoffman-Ilmanen-Mart{\'\i}n-White \cite{HoffmanIlmanenMartinWhite2019}, possibly times a Euclidean factor. 
These only occur if $k \geq 2$.
\end{enumerate}
\end{Theorem}

In addition, our methods yields a canonical characterization of these models modulo spatial and time-translations.
This picture differs somewhat from the conventional picture as it is only based on local asymptotics of leading modes.
To explain this picture, we restrict our attention to asymptotically $(n,k)$-cylindrical mean curvature flows that are non-collapsed, convex and rotationally symmetric about the axis $\IR^{k} \times \bO^{n-k+1}$ and that have uniformly bounded second fundamental form on time-intervals of the form $(-\infty,T]$ for $T < 0$.
Among these flows, we define:
\begin{itemize}
\item $\MCF^{n,k}_{\Oval}$ to be the space of such flows that go extinct at time $0$, whose extinction locus (i.e., its singular set) contains the origin and that are invariant under reflections across some collection of $n+1$ pairwise orthogonal hyperplanes, which pass through the origin and may depend on the flow.
(This set contains the models \ref{Thm_class_A} and \ref{Thm_class_B} modulo translations.)
\item $\MCF^{n,k}_{\soliton}$ to be the space of flows that are translating solitons, so whose time-slices are $\MM_t = \MM_0 + t \mathbf v$ for some $\mathbf v \in \IR^k \times \bO^{n-k+1}$, whose time-$0$-slice $(\spt \MM)_0$ contains the origin and that invariant under reflections across some collection of $n$ pairwise orthogonal hyperplanes, which pass through the origin and may depend on the flow.
(This set contains the models \ref{Thm_class_C} and \ref{Thm_class_D} modulo translations.)
\end{itemize}
Note that flows in $\MCF^{n,k}_{\Oval}$ are smooth for all negative times and go extinct at time $0$ and the flows in $\MCF^{n,k}_{\soliton}$ are defined for all times and smooth everywhere.
We now set 
\[ \MCF^{n,k}_0 := \MCF^{n,k}_{\Oval} \cup \MCF^{n,k}_{\soliton}. \]
Using this notation, Theorem~\ref{Thm_classification_simple} can be restated as:

\begin{Theorem} \label{Thm_classification_precise}
If $\MM$ is an asymptotically $(n,k)$-cylindrical mean curvature flow, then there is a $\bp \in \IR^{n+1}$ such that $\MM' := \MM + (\bp, t_0)$ is the restriction of a flow from $\MCF^{n,k}_{0}$ to the time-interval on which $\MM'$ is defined.
\end{Theorem}

We will now describe the space $\MCF^{n,k}_{0}$ with the topology induced by Brakke convergence, which is equivalent to smooth convergence wherever the limit is smooth.
Consider the canonical map from \cite[Definition~\refx{Def_Qu}]{Bamler_Lai_PDE_ODI}:
\[ \Qu : \MCF^{n,k}_{0} \lto \IR^{k \times k}_{\geq 0} \]
We recall $\Qu(\MM)$ is roughly defined as follows:
We first express the rescaled flow $\td\MM_{\tau} := e^{\tau/2} \MM_{-e^{-\tau}}$ as the normal graph of a function $u_\tau$ over the round cylinder and then study the asymptotic behavior of $u_\tau$ projected to the space of quadratic Hermite polynomials as $\tau \to -\infty$.
The behavior of this projection can be modeled by the solution to a finite-dimensional ODE up to a term of the order $O(|\tau|^{-3})$ and this solution can be parameterized by a non-negative definite matrix $\Qu(\MM)$.
In \cite[Proposition~\refx{Prop_dom_qu_asymp}]{Bamler_Lai_PDE_ODI}, we showed that $\Qu(\MM)$ even determines the asymptotic order of any finite mode $u_\tau$ up to any polynomial term of the form $O(|\tau|^{-J})$.
In \cite[Theorem~\refx{thm:oval_existence}]{Bamler_Lai_PDE_ODI}
 we also showed that each admissible value of $\Qu$ can be realized by an ancient oval in $\MCF^{n,k}_{\Oval}$.
The next result describes the restriction of $\Qu$ to the subspace $\MCF^{n,k}_{\Oval}$.

\begin{Theorem}\label{Thm_classification_Q}
The following is true:
\begin{enumerate}[label=(\alph*)]
\item \label{Thm_classification_Q_a} The map $\Qu|_{\MCF^{n,k}_{\Oval}} : \MCF^{n,k}_{\Oval} \to \IR^{k \times k}_{\geq 0}$ is bijective and a homeomorphism.
\item \label{Thm_classification_Q_b} If $\MM \in \MCF^{n,k}_{\Oval}$, then $\MM$ is invariant under reflections perpendicular to all spectral directions of $\Qu(\MM)$.
\item \label{Thm_classification_Q_c} If $\MM \in \MCF^{n,k}_{\Oval}$, and $\Qu(\MM)$ has non-trivial $l$-dimensional nullspace, then $\MM$ splits as a flow $\MM' \in \MCF^{n-l,k-l}_{\Oval}$ times an $\IR^l$-factor in the nullspace direction.
\item \label{Thm_classification_Q_d} We have $\Qu(\MM) = 0$ if and only if $\MM$ is the round shrinking $(n,k)$-cylinder $\MM^{n,k}_{\cyl}$.
\end{enumerate}
\end{Theorem}

Next, we characterize the entire space $\MCF^{n,k}_{0}$ as an extension of $\MCF^{n,k}_{\Oval}$.

\begin{Definition}
We define the map
\[ \mathbf b : \MCF^{n,k}_{0}  \lto \IR^k, \]
as follows.
If $\MM \in \MCF^{n,k}_{\Oval}$, then we set $\mathbf b(\MM) := \bO$.
If $\MM \in \MCF^{n,k}_{\soliton}$ is a translating soliton with $\mathbf v = \mathbf H(\bO,0)$, then we set $\mathbf b (\MM) := |\mathbf v|^{-2} \mathbf v$; note that $\mathbf v$ is a velocity vector of $\MM$.
\end{Definition}

The next theorem describes $\MCF^{n,k}_0$ via a homeomorphic map.

\begin{Theorem} \label{Thm_classification_Qb}
For all $\MM \in \MCF^{n,k}_0$ we have $\mathbf b(\MM) \in \nullspace \Qu(\MM)$, where ``$\nullspace$''  denotes the nullspace, and the map
\begin{equation} \label{eq_Qb_map}
  (\Qu, \mathbf b) : \MCF^{n,k}_0 \lto \{ (\Qu', \mathbf b') \in \IR^{k \times k}_{\geq 0} \times \IR^k \;\; : \;\; \mathbf b' \in \nullspace(\Qu') \} 
\end{equation}
is bijective and a homeomorphism, where for the latter space we take the subspace topology within $\IR^{k \times k} \times \IR^k$.
Moreover, the following is true:
\begin{enumerate}[label=(\alph*)]
\item \label{Thm_classification_Qb_a} We have $\mathbf b (\MM) = 0$ if and only if $\MM \in \MCF_{\Oval}^{n,k}$.
If $\Qu(\MM)$ is invertible, then $\MM \in \MCF_{\Oval}^{n,k}$.
\item \label{Thm_classification_Qb_b} We have $\Qu(\MM) = 0$ if and only if $\MM$ is the round shrinking cylinder $\MM_{\cyl}^{n,k}$ (if $\mathbf b (\MM) \neq \bO$) or homothetic to $\IR^{k-1} \times \MM^{k+1}_{bowl}$.
\item \label{Thm_classification_Qb_c} $\MM$ is invariant under translations in all directions of $\nullspace \Qu(\MM)$ that are perpendicular to $\mathbf b(\MM)$.
So $\MM$ splits of a Euclidean factor of dimension equal to the nullity of $\Qu(\MM)$ if $\mathbf b(\MM) = \bO$ or equal to the nullity of $\Qu(\MM)$ minus one if $\mathbf b(\MM) \neq \bO$.
\end{enumerate}
\end{Theorem}

A crucial aspect of our approach is that we recover $\mathbf b (\MM)$ as the coefficient of an \emph{exponentially} decaying difference between two flows that each converge to the cylinder, but only at a \emph{polynomial} rate as $\tau \to -\infty$.
More specifically, fix some asymptotically $(n,k)$-cylindrical mean curvature flow $\MM$ and choose the unique flow $\MM' \in \MCF^{n,k}_{\Oval}$ with $\Qu(\MM) = \Qu(\MM')$.
Consider the rescaled versions $\td\MM_{\tau} := e^{\tau/2} \MM_{-e^{-\tau}}$ and $\td\MM'_{\tau} := e^{\tau/2} \MM'_{-e^{-\tau}}$ of both flows.
As $\tau \to -\infty$ these rescaled flows can be expressed as graphs of functions $u_\tau$ and $u'_\tau$ over the round cylinder, defined on larger and larger domains.
We will then roughly show:

\begin{Theorem}[vague version of Theorem \ref{Thm_Vp_mode_precise}] \label{Thm_def_b_vague}
For $(\bx,\by) \in \IR^k \times \IS^{n-k}$ we have
\[ u_\tau(\bx,\by) - u'_\tau(\bx,\by) = \sum_{i=0}^k b_i e^{\tau/2} \bx_i + O(|\tau|^{-1} e^{\tau/2}). \]
Moreover, the coefficients $b_i$ are uniquely determined by $\MM$ and $\mathbf b (\MM) := \sum_{i=1}^k b_i \mathbf e_i$ must be contained in the nullspace of $\Qu(\MM)$. 
\end{Theorem}

{We also note that the picture from Theorem~\ref{Thm_classification_Qb} could be extended by considering the larger space $\MCF^{n,k} \supset \MCF^{n,k}_0$ of \emph{all} asymptotically $(n,k)$-cylindrical mean curvature flows that are rotationally symmetric about the axis $\IR^k \times \bO^{n-k}$; these include space and time-translations of elements of $\MCF^{n,k}_0$.}
This extended picture, combined with our methods, yields new approach for constructing the solitons $\MCF^{n,k}_{\soliton}$, which were originally constructed by Hoffman-Ilmanen-Mart{\'\i}n-White, as limits of elements of $\MCF^{n,k}_{\Oval}$ shifted in time and space.

\begin{Theorem} \label{Thm_limit_is_flying_wing}
For every $\MM \in \MCF^{n,k}_{\soliton}$ there are sequences $\MM_i \in \MCF^{n,k}_{\Oval}$ and $(\bp_i, t_i) \in \IR^{n+1} \times \IR$ such that $\MM_i + (\bp_i, t_i) \to \MM$ in the local smooth sense.
\end{Theorem}

For completeness, we also restate a compactness theorem the space of asymptotically cylindrical mean curvature flows; this has been known for non-collapsed flows, so due to our main result it holds for \emph{all} asymptotically cylindrical flows.
The key point is that we need to allow the dimensional constant $k$ to drop in the limit.

\begin{Theorem} \label{Thm_compactness_ancient}
Let $\MM_i$ be a sequence of asymptotically $(n,k)$-cylindrical mean curvature flows.
Then after passing to a subsequence, we have convergence in the Brakke sense (and hence in the local smooth sense away from the extinction time) $\MM_i \to \MM$, where the limit is one of the following flows:
\begin{itemize}
\item An asymptotically $(n,k')$-cylindrical mean curvature flow, for $k' \leq k$.
\item A constant, affine, multiplicity one plane.
\item An empty flow.
\end{itemize}
If $\Vert \Qu (\MM_i) \Vert$ is uniformly bounded, then we can take $k' = k$ in the first option.
\end{Theorem}
\medskip

\subsection{Statement of main results II: The Mean Convex Neighborhood Conjecture} \label{subsec_main_results_II}
Let us first state a basic version of our resolution of the Mean Convex Neighborhood Conjecture .

\begin{Theorem} \label{Thm_MCN_basic}
Suppose that $\MM$ is an $n$-dimensional, unit-regular, integral Brakke flow in $\IR^{n+1} \times I$ and suppose that the tangent flow at some point $(\bp_0, t_0)$ is a round multiplicity one $(n,k)$-cylinder.
Then there is a neighborhood $\UU \subset \IR^{n+1}$ of $\bp$ and a continuous function $u \in C^0(\UU)$ such that for $t$ close to $t_0$ we have $(\spt \MM)_t \cap \UU = u^{-1}(t)$.
Moreover, $u$ is smooth at all $\bp \in \UU$ for which $(\bp, u(\bp))$ is regular and the mean curvature vector satisfies
\[ \mathbf H (\bp, u(\bp)) = \frac{\nabla u(\bp)}{|\nabla u(\bp)|^2}, \] 
so it is positive with respect to the co-orientation induced by viewing $\MM^{\reg}_t$ as a level set of $u$.
Moreover, at all singular points near $(\bp_0, t_0)$ the tangent flows are multiplicity one round $(n,k')$-cylinders for $k' \leq k$.
\end{Theorem}

We remark that we do not need to require that $\MM$ is cyclic.

We also characterize all blow-up limits near $(\bp_0, t_0)$:

\begin{Theorem} \label{Thm_blow_ups}
Suppose that $\MM$ is an $n$-dimensional, unit-regular, integral Brakke flow in $\IR^{n+1} \times I$ and suppose that the tangent flow at some point $(\bp_0, t_0)$ is a round multiplicity one $(n,k)$-cylinder.
Let $(\bp_i, t_i) \to (\bp_0, t_0)$ and $\la_i \to \infty$ be sequences.
Then, after passing to a subsequence, we have convergence $\la_i (\MM- (\bp_i, t_i)) \to \MM_\infty$ in the Brakke sense, where $\MM_\infty$ is empty, an affine, multiplicity one plane, or isometric to some asymptotically $(n,k')$-cylindrical mean curvature flow for $k' \leq k$, which is non-collapsed, mean convex, rotationally symmetric and classified by the discussion in the previous subsection.
\end{Theorem}

Lastly, we obtain a more uniform version of the Mean Convex Neighborhood Conjecture, which only requires closeness to a cylinder at some initial time instead of cylindrical tangent flows and which characterizes a large forward parabolic neighborhood via local models.
To state this result, we first recall the notion of an $(n,k,\delta)$-neck:

\begin{Definition}
Let $\MM$ be a unit-regular, $n$-dimensional, integral Brakke flow in $\IR^{n+1} \times I$.
A point $\bp_0 \in \IR^{n+1}$ is called \textbf{center of an $(n,k,\delta)$-neck of $\MM$ at scale $r > 0$ and time $t_0$} if there is an orthogonal map $S \in O(n+1)$ such that $\MM' := r^{-1} (\MM - (\bp_0, t_0))$ is $\delta$-close to the round $(n,k)$-cylinder $M^{n,k}_{\cyl} = \IR^k \times \IS^{n-k}$ at time $0$.
This means that $\MM'$ has no singular points at time $0$ within the closed $\delta^{-1}$-ball $\ov B := \ov \IB^{n+1}_{\delta^{-1}} \subset \IR^{n+1}$ around the origin, that $(\spt \MM')_0 \cap \ov B$ and $M^{n,k}_{\cyl} \cap \ov B$ have Hausdorff distance $< \delta$ and that the regular part $\MM^{\prime, \reg}_0 \cap B$ can be expressed as the normal graph $\Gamma_{\cyl}(v)$ of a function $v$ over {a subset\footnote{times a normalization factor of $\sqrt{2(n-k)}$}} of $M^{n,k}_{\cyl}$  with $\Vert v \Vert_{C^{[\delta^{-1}]}} < \delta$.
\end{Definition}

We also need the following definition.

\begin{Definition}
Let $\MM$ be a unit-regular, $n$-dimensional, integral Brakke flow in $\IR^{n+1} \times I$ and let $\eps > 0$ and $k \in \{ 0, \ldots, n-1 \}$.
We say that a regular point $(\bp, t) \in \MM^{\reg}$ has a \textbf{strong $(\eps, k)$-canonical neighborhood} if its mean curvature vector at $(\bp, t)$ satisfies $\mathbf H(\bp, t) \neq 0$ and if the following is true for $r := |\mathbf H|^{-1} (\bp,t)$.
The rescaled flow $\MM' = r^{-1} (\MM - (\bp, t))$ is $\eps$-close to an ancient flow $\MM^*$, which is isometric (modulo a rotation) to an asymptotically $(n,k')$-cylindrical mean curvature flow for some $k' \in \{0, \ldots, k \}$, in the following sense:
\begin{itemize}
\item both $\MM'$ and $\MM^*$ are regular on $P:= \IB^{n+1}_{\eps^{-1}} \times ([-\eps^{-1}, 0] \cap r^{-2} (I-t))$,
\item $\spt \MM' \cap P$ and $\spt \MM^* \cap P$ have Hausdorff distance $< \eps$.
\item $\MM^{\reg} \cap \PP$ is the normal graph of a function $v$ over an open subset of $\MM^{*,\reg}$ with $|v| + \ldots + |\nabla^{[\eps^{-1}]} v | < \eps$.
\end{itemize}
\end{Definition}

Our result is now the following:

\begin{Theorem} \label{Thm_MCN}
For every $\eps > 0$ there is a constant $\delta (\eps) > 0$ with the following property.
Suppose that $\MM$ is a unit-regular, $n$-dimensional, integral Brakke flow in $\IR^{n+1} \times I$.
Let $(\bp_0, t_0) \in  \IR^{n+1} \times I$ and $r_0 > 0$ and assume that $\bp_0$ is a center of an $(n,k,\delta)$-neck of $\MM$ at time $t_0$ and scale $r_0$, for some $k \in \{ 0, \ldots, n-1 \}$.
Then the following is true:
\begin{enumerate}[label=(\alph*)]
\item \label{Thm_MCN_a} Let $\UU \subset B(\bp_0, \eps^{-1} r_0) \subset \IR^{n+1}$ be the closure of the component of $B(\bp_0, \eps^{-1} r_0) \setminus \MM^{\reg}_{t_0}$ containing $\bp_0$.
There is a continuous function $u : \UU \to [t_0, t_0 + \eps^{-1} r_0^2] \cap I$ such that for all $t \in [t_0, t_0 + \eps^{-1} r_0^2] \cap I$ we have
\[ (\spt \MM)_t \cap B(\bp_0, \eps^{-1} r_0) = \{ u = t \}. \]
Moreover, the measure representing $\MM$ at time $t$, restricted to $B(\bp_0, \eps^{-1} r_0)$ is equal to the $n$-dimensional Hausdorff measure restricted to $\{ u = t \}$.
\item \label{Thm_MCN_b} For any 
\[ (\bp,t) \in \spt \MM \cap \big( B(\bp_0, \eps^{-1} r_0) \times ([t_0 , t_0 + \eps^{-1} r_0^2] \cap I) \big) \]
the following is true:
\begin{itemize}
\item If $(\bp, t)$ is a singular point, then its tangent flow is a multiplicity one round shrinking cylinder isometric to $\MM^{n,k'}_{\cyl}$, for some $k' \in \{ 0, 1, \ldots, k \}$.
So its nearby blow-up models are characterized by Theorem~\ref{Thm_blow_ups}.
\item If $(\bp, t)$ is a regular point, then it has a strong $(\eps, k)$-canonical neighborhood.
\end{itemize}
\end{enumerate}
\end{Theorem}
\bigskip

Lastly, we point out a direct consequence of Theorem~\ref{Thm_MCN_basic} combined with \cite[Theorem~3.1]{Hershkovits_White_nonfattening} in the work of Hershkovits-White.
Recall that for a closed, smoothly embedded hypersurface $M_0 \subset \IR^{n+1}$ the \emph{discrepancy time} $T_{\textnormal{disc}} (M_0)$ is defined as the maximum of all times $T' > 0$ such that the level set flow, the innermost flow and outermost flows coincide on $[0,T')$.
Note that $T_{\textnormal{disc}} (M_0) \leq T_{\textnormal{fat}} (M_0)$, where the latter denotes the \emph{fattening time;} these two times are equal if $n=2$, see \cite[Theorem~1.8]{Bamler_Kleiner_mult1} and the same is conjectured to be true in higher dimensions.
Furthermore, by the foundational work of Ilmanen~\cite{Ilmanen_ell_reg}, there is a \emph{unique} matched, unit-regular, integral Brakke flow starting from $M_0$ on the time-interval $[0,T_{\textnormal{fat}} (M_0))$.
We now obtain the following corollary:

\begin{Corollary}\label{Cor_HW_consequence}
Let $M_0 \subset \mathbb{R}^{n+1}$ be a closed, smoothly embedded hypersurface and suppose that $T_{\textnormal{disc}} (M_0) < \infty$.
Let $\MM$ be the unique unit-regular, integral Brakke flow starting from $M_0$ on the time-interval $[0,T_{\textnormal{disc}} (M_0))$.
Then there is a singular point $(\bp, T_{\textnormal{disc}} (M_0))  \in \IR^{n+1} \times \IR_{\geq 0}$ at which no tangent flow is a multiplicity one cylinder.
\end{Corollary}

In other words, at the moment at which the innermost and outermost flows diverge, the flow must develop a singularity whose tangent flow is either of higher multiplicity or different from a cylinder.
\bigskip

\subsection{Structure of the paper an overview of the proof} \label{subsec_structure}
To classify ancient, asymptotically cylindrical flows, we must develop a robust method of comparing two given flows $\MM^0$ and $\MM^1$.

When both flows have dominant linear mode, this comparison was carried out in our prequel \cite{Bamler_Lai_PDE_ODI}, using the PDE-ODI principle, which showed that the flows must be homothetic to the bowl soliton times a Euclidean factor or the round shrinking cylinder.
It is helpful to recall our method: we considered the rescaled flows $\td\MM^i_\tau = e^{\tau/2} \MM^i_{-e^{-\tau}}$ and studied their convergence to the round cylinder as $\tau \to -\infty$.
In the case of dominant linear mode, this convergence is exponential and the asymptotics are governed by a finite set of leading exponential modes.
If $\MM^0$ and $\MM^1$ have the same leading modes, then our methods in \cite{Bamler_Lai_PDE_ODI} established asymptotic closeness of both flows to arbitrary exponential order in $\tau$, which can be converted to a suitable spatial decay.
This decay was fast enough to show $\MM^0 = \MM^1$ via a comparison principle.

In the case of dominant \emph{quadratic} mode, our work \cite{Bamler_Lai_PDE_ODI} introduced an asymptotic invariant $\Qu(\MM^i) \in \IR^{k \times k}_{\geq 0}$ and proved that this quantity determines the convergence to the cylinder up to arbitrary polynomial order.
More precisely, if $\Qu(\MM^0) = \Qu(\MM^1)$, then the rescaled flows $\td\MM^0$ and $\td\MM^1$ must agree to the order $O(|\tau|^{-J})$ as $\tau \to \infty$ on larger and larger neighborhoods of the origin, for any $J$.
This decay is, however, much weaker than exponential decay, and Theorem~\ref{Thm_def_b_vague} shows that exponentially small terms are, in fact, often essential for distinguishing two flows.

Unfortunately, exponentially decaying differences---arising when the individual flows converge only at a polynomial rate---lie beyond the reach of our methods from~\cite{Bamler_Lai_PDE_ODI}.
Roughly speaking this is because the PDE–ODI principle developed in our prior work rests on a pseudolocality estimate that operates at a fixed threshold, and this reliance is responsible for an error term in the resulting ODI.
While we were able to force this error term to be an arbitrary power of the leading mode, this is insufficient here: the leading mode of each flow itself is only polynomial, whereas the difference between $\td\MM^0$ and $\td\MM^1$ decays exponentially.
Thus, to control the exponentially small difference between $\td\MM^0$ and $\td\MM^1$, we must find an alternative to pseudolocality that also remains effective at small thresholds.
Specifically, we need a local \emph{Harnack-type} estimate that bounds the growth of such differences.

This is precisely what the \emph{leading mode condition} introduced in Section~\ref{sec_leading_mode} accomplishes.
It characterizes the difference between $\MM^0$ and $\MM^1$ in regions where the flows are close, so where $\MM^1$ can be written as a normal graph of a function $u$ over $\MM^0$.
Specifically, it asserts roughly that:
\begin{itemize}
\item On cylindrical regions of $\MM^0$, the function $u$ is locally well approximated by an element $U\in\sV_{\ge0}$ (the semi-stable subspace of the linearized operator at the cylinder from \cite{Bamler_Lai_PDE_ODI}), and the error of this approximation is small \emph{relative to $\|U\|$} and  modulo an absolute error term that decays exponentially in time at a rapid rate.
\item On regions of $\MM^0$ modeled on a bowl-soliton times a Euclidean factor, $u$ is controlled by its values on the adjacent cylindrical region multiplied by a spatial exponential factor, again modulo an error term that decays rapidly exponentially in time.
\item On regions not modeled on either geometry (such as the ``cap region'' of a flying wing soliton), we do not impose any bounds on $u$.
\end{itemize}
The leading mode condition depends on several auxiliary constants governing these approximations, which must be chosen carefully during our construction.

We will show that the leading mode condition holds for any two ancient cylindrical flows\footnote{For technical reasons, we assume $\MM^0$ is convex and rotationally symmetric, though these assumptions are not essential.}.
Philosophically, this follows from a stability property: the condition is designed so that---for appropriately chosen constants---a stronger version can be deduced from a weaker one.

However, the stability alone does not allow us to establish the leading mode condition from scratch, because ancient flows have no initial time at which the condition is known to hold.
To overcome this, we introduce a novel ``induction over thresholds'' argument.
We modify the leading mode condition so that the properties listed above are required only in the region where $u \gtrsim c$ for some fixed $c \ge 0$.
When $c=0$, this reduces to the original leading mode condition.
Our induction step consists of progressively lowering the threshold for $u$ at which the leading mode condition must hold.
We assume that the leading mode condition holds wherever $u\gtrsim c$ for some  $c > 0$, and we prove---using the stability property from before---that it also holds wherever $u\gtrsim \frac12 c$.
Iterating this process ultimately establishes the full leading mode condition, without an additional threshold condition, in the limit.

To initiate the induction, i.e., to show that the leading mode condition holds whenever $u \gtrsim c$ for some macroscopic $c>0$, we appeal to our prior work \cite{Bamler_Lai_PDE_ODI}.
There, it was shown that $\MM^0$ and $\MM^1$ can, in most regions, be approximated by a cylinder or a bowl times an Euclidean factor.
This approximation, however, comes with an \emph{absolute} error term and is therefore too coarse to imply the desired bounds from the list above, which characterize $u$ \emph{relative} to its nearby values.
But on regions where $u \gtrsim c$, such an absolute error can be converted into a relative one, which is precisely what is needed to begin the induction.

To achieve our induction step, we show a strengthened version of the stability property mentioned above---one that remains effective even when the leading mode condition is known only above a threshold and it is strong enough to allow that threshold to be lowered repeatedly.
It is established by propagating each of the conditions from the list above using the other conditions via a combination of maximum principle arguments, asymptotic estimates, and limiting procedures.

\medskip

Once established, the leading mode condition yields a Harnack-type inequality for the difference between the rescaled flows $\td{\MM}^0$ and $\td{\MM}^1$, which enables a localized parabolic analysis reminiscent of our PDE-ODI principle from \cite{Bamler_Lai_PDE_ODI} or the approach of Angenent-Daskalopoulos-Sesum~\cite{ADS_2019} in the rotationally symmetric case.
In Section~\ref{sec_diff}, we show that this difference is governed by a single unstable mode, which decays exponentially as $\tau \to -\infty$.
After eliminating several modes by translating $\MM^1$ in space and time, the only remaining mode is the one identified in Theorem~\ref{Thm_def_b_vague}.
If this mode vanishes, the flows agree up to a term that decays rapidly exponentially, which can be converted to suitable global quantitative control. 
This allows us to conclude $\MM^0 = \MM^1$ by a comparison principle (which is carried out in Section~\ref{sec_uniqueness}).
This yields the classification theorem stated in Subsection~\ref{subsec_main_results_I}.

\medskip

The resolution of the Mean Convex Neighborhood Conjecture stated in Subection~\ref{subsec_main_results_II} follows directly from the classification of ancient asymptotically cylindrical flows.
The argument parallels the $1$-cylindrical case in \cite{Choi_Haslhofer_Hershkovits_White_22}, except that we must allow for a hierarchical structure of singularity models: as we rescale the flow near a singular point $(\bp_0,t_0)$, regions may exhibit $(n,k)$-cylindrical structure with $k$ decreasing along the flow.
At every scale experiencing such a drop, the flow is close to an asymptotically cylindrical model, which has strictly positive mean curvature by our classification.
Hence the mean curvature cannot change sign.

\medskip

Finally, we summarize the structure of our paper.
In Section~\ref{sec_Preliminaries} we discuss preliminaries that are particularly relevant for this paper.
In Section~\ref{sec_leading_mode} we establish the leading mode condition.
While this condition is slightly involved, it is confined to this section and does not appear in the main result used in the next section.
In Section~\ref{sec_uniqueness} we carry out the comparison principle in the same section, which establishes equality of both flows under a strong asymptotic decay condition.
In Section~\ref{sec_diff} we study the evolution of the difference between two ancient rescaled flows using a localized analysis and characterize its asymptotic behavior as $\tau \to -\infty$ via the dominant mode.
Section~\ref{sec_proofs_I} contains the classification of asymptotically cylindrical flows (see Subsection~\ref{subsec_main_results_I}), and Section~\ref{sec_proofs_II} proves our results on the Mean Convex Neighborhood Conjecture (see Subsection~\ref{subsec_main_results_II}).

\subsection{Acknowledgements}
We thank Or Hershkovits for pointing out that Corollary~\ref{Cor_HW_consequence} is a consequence of combining our work with \cite{Hershkovits_White_nonfattening}.
\bigskip

\section{Preliminaries} \label{sec_Preliminaries}
We refer to the preliminaries section in \cite{Bamler_Lai_PDE_ODI} for a discussion of terminologies and basic facts.
We recall the following notions that are particularly important to this paper.
Let $\MM$ be an $n$-dimensional unit-regular, integral Brakke flow in $\IR^{n+1} \times I$.

We say that $\MM$ is \textbf{convex} if there is a $T \leq \infty$ (its \textbf{extinction time}) such that $(\spt \MM)_t = \emptyset$ for $t > T$, $\emptyset \neq (\spt \MM)_{T} = \MM^{\sing}_T$ if $T < \infty$ and if $t < T$, then $(\spt \MM)_t = \MM^{\reg}_t$ is the boundary of a closed, convex subset $C_t = \IR^{n+1}$ with non-empty interior.
The time-slice $\MM^{\reg}_t$ must have non-negative definite second fundamental form if the co-orientation is chosen accordingly.

We also recall that a submanifold $M \subset \IR^{n+1}$ or $\MM$ is called \textbf{$(n,k)$-rotationally symmetric,} if it is invariant under orthogonal maps in $O(n-k+1)$ applied to the second factor of $\IR^{n+1} = \IR^k \times \IR^{n-k+1}$.
The subspace $\IR^k \cong \IR^k \times \bO^{n-k+1}$ is called the \textbf{axis of rotation.}
When $\MM$ is asymptotically $(n,k)$-cylindrical or when the context is clear, we will often omit the prefix ``$(n,k)$''.

If $M \subset \IR^{n+1}$ is $(n,k)$-rotationally symmetric and convex, then it can be expressed in the form
\[ M = \bigcup_{\bq \in \IR^{k}} M(\bq), \]
where $M(\bq)$ is either empty or of the following form for some $\rho(\bq) \geq 0$
\[ M(\bq) = \bq \times (\rho(\bq) \IS^{n-k}) . \]
Here $\rho(\bq)$ is called the \textbf{radius of $M$ at $\bq$;} note that due to our definition of $\IS^{n-k}$, the radius of the sphere $M(\bq)$ is in fact $\sqrt{2(n-k)} \, \rho(\bq)$.
If $M(\bq) \neq \emptyset$, then we say that $\bq$ is \textbf{covered by $M$.}
If a smooth mean curvature flow $\MM$ is rotationally symmetric, then we denote the radius function of $\MM_t$ by $\rho(\cdot, t)$.
\medskip

In this paper, we will consider the \textbf{linearized mean curvature flow equation} on an evolving background.
We recall the relevant equations.
Suppose that $\MM^0$ and $\MM^1$ are smooth mean curvature flows in $\IR^{n+1} \times I$ and suppose that we can express $\MM^1$ as a normal graph of a function $u : \DD \to \IR$ over a subset $\DD \subset \MM^{0,\reg}$.
More precisely, fix a smooth unit normal vector field $(\nu_{\bp, t})_{(\bp,t) \in \MM^0}$.
Then we assume that 
\[ (\spt \MM^1) = \MM^{1,\reg} = \big\{ (\bp + u(\bp,t) \nu_{\bp,t},t) \;\; : \;\; (\bp,t) \in \DD \big\}. \]
It is well known that the mean curvature flow equation on $\MM^{1,\reg}$ can be expressed in terms of an equation on $u$ of the form
\[ \partial_{\mathbf t} u = \triangle u + |A_{\MM^0}|^2 u + Q(u, \nabla u, \nabla^2 u, A_{\MM^0}), \]
where $\partial_{\mathbf t}$ is the normal time-derivative and $\triangle$ is the intrinsic Laplacian of $\MM^0$ and $A_{\MM^0}$ is the second fundamental form.
The non-linear term vanish to second order in $u, \nabla u, \nabla^2 u$.
So the linearization of this equation is
\begin{equation} \label{eq_lin_MCF}
 \partial_{\mathbf t} u = \triangle u + |A_{\MM^0}|^2 u. 
\end{equation}
If $\MM^0$ is the round shrinking cylinder $\MM_{\cyl}^{n,k}$, with time-slices $\IR^k \times (\sqrt{-t} \IS^{n-k})$ then (since by definition the radius of $\IS^{n-k}$ is $\sqrt{2(n-k)}$) we have $|A_{\MM^0}| = \frac12 |t|^{-1}$, so
\[ \partial_{\mathbf t} u = \triangle u + \tfrac1{2|t|} u. \]
So if we pass to the rescaled flow $\td\MM_{\tau}^{\reg} = e^{\tau/2} \MM^{\reg}_{-e^{-\tau}}$, which is just the constant cylinder $\MM_{\cyl}^{n,k}$, then the function $\td u(\bp,\tau) := e^{\tau/2} u(e^{-\tau/2} \bp, -e^{-\tau})$ satisfies the familiar rescaled linearized equation from \cite{Bamler_Lai_PDE_ODI}
\[ \partial_\tau \td u = \triangle_f \td u + \tfrac12 \td u. \]

\bigskip

\section{The leading mode condition} \label{sec_leading_mode}

\subsection{Overview and main result}
This section contains the key estimate of the paper.
We will consider two asymptotically cylindrical mean curvature flows, $\MM^0$ and $\MM^1$.
In regions where the flows are close, one can describe $\MM^1$ as a perturbation of $\MM^0$ by means of a perturbation function $u$, which may be thought of informally as the ``difference'' between the flows.
Roughly speaking, we will show that this perturbation function can, in neck regions and to high accuracy, be approximated by weakly unstable modes from $\sV_{\geq 0}$.
Among other things, this implies that the separation between the flows can grow at most exponentially in space and time.

The underlying idea is that, whenever the perturbation function is sufficiently small, it behaves approximately like a solution to the linearized mean curvature flow equation.
Using this picture, we will decompose $u$ into a sum of weakly unstable modes from $\sV_{\geq 0}$ and stable modes from $\sV_{<0}$.
Since the stable modes decay more rapidly, one expects that the weakly unstable modes eventually dominate.
For ancient flows, this domination should in fact hold at all times.

Of course, this picture is far too simplistic.
In reality, the perturbation function is only defined on the open subset of $\MM^0$ where the two flows are close, which complicates our analysis.
Moreover, for general non-linear PDEs, a property of this type may often be highly desirable, but typically false.
Indeed, if $u_0$ and $u_1$ are solutions on an open domain (say an open subset of $\IR^n$), then their difference $u := u_1 - u_0$ can rarely be described by the unstable modes of the linearization of this PDE, particularly when $|u| \ll |u_0|, |u_1|$.
Even Harnack-type estimates as basic as $|\nabla u| \leq C |u|$ generally fail.
In our situation, such estimates also lie beyond the scope of the PDE-ODI principle from \cite{Bamler_Lai_PDE_ODI}, because the resulting evolution inequality for $u$ necessarily involves error terms depending on $u_0$ and $u_1$, which dominate precisely in regions where $|u| \ll |u_0|, |u_1|$.

The situation is more favorable for ancient cylindrical mean curvature flows.
Although a direct PDE analysis is unavailable in regions where $\MM^0$ and $\MM^1$ are not close, one can often describe the geometry of the flows there.
This in turn yields useful PDE-type estimates on the perturbation function near the boundary of its domain, which allow us to establish a carefully chosen \emph{leading mode condition} via an ``induction over thresholds'' argument.
Ultimately, we will prove that any two ancient cylindrical flows $\MM^0$ and $\MM^1$ (with $\MM^0$ assumed to be convex and rotationally symmetric for technical reasons) satisfy this leading mode condition.

The precise definition of the leading mode condition is given in Definition~\ref{Def_lead_mode} and it is established in Proposition~\ref{Prop_leading_mode_condition}.
This is the central result of this section and key estimate of this paper.
It provides a detailed description of the possible local differences between two ancient cylindrical flows.
Surprisingly, however, much of the content of Proposition~\ref{Prop_leading_mode_condition} can already be captured by the following proposition, which is more concise and self-contained.
Although this proposition looks deceptively simple, it is essentially equally strong as the leading mode condition and it will be more convenient for the remainder of this paper.

\medskip
The following proposition establishes an exponential growth bound on the difference between $\MM^0$ and $\MM^1$ in sufficiently cylindrical regions.
Moreover, if this difference decays rapidly in time, then it states that both flows are extremely close in regions of sufficiently large scale.
Later, we will use this closeness to show that in this case, the two flows must in fact coincide.

\begin{Proposition} \label{Prop_diff_properties}
Let $\MM^0, \MM^1$ be two asymptotically $(n,k)$-cylindrical mean curvature flows in $\IR^{n+1} \times (-\infty,T)$, where we assume $\MM^0$ to be convex and $(n,k)$-rotationally symmetric and assume that for some constant $C^* > 0$
\[ \Vert \Qu(\MM^1)\Vert \leq C^* \Vert \Qu(\MM^0)\Vert.\]
For each $i = 0,1$ consider the corresponding rescaled flow $\td\MM^i$, so $\td\MM^{i,\reg}_\tau = e^{\tau/2} \MM^{i,\reg}_{-e^{-\tau}}$ (see \cite[Subsection~\refx{subsec_resc_mod_mcf}]{Bamler_Lai_PDE_ODI} for more details).
Let $u_i \in C^\infty(\DD_i)$ for $\DD_i \subset M_{\cyl} \times (-\infty, -\log(-T))$ be the functions representing these flows as graphs over the standard round cylinder as in the statement of \cite[Proposition~\refx{Prop_PDE_ODI_MCF}]{Bamler_Lai_PDE_ODI}.
Recall that this means that for all $\tau < -\log(-T)$ and $i = 0,1$ there is a maximal radius $R'_{i,\tau} \in [0,\infty]$ such that we can write
\[ \td\MM^{i,\reg}_\tau \cap \IB^{n+1}_{R'_{i,\tau}} = \Gamma_{\cyl}(u_{i,\tau}) \]
for some $u_{i,\tau} \in C^\infty(\DD_{i,\tau})$.
Set $v := u_{1} - u_0$ on $\DD := \DD_0 \cap \DD_1$.
Then the following is true:
\begin{enumerate}[label=(\alph*)]
\item \label{Prop_diff_properties_a} There is a  constant $C(n, C^*) > 0$ with the following property.
Suppose that for some $\tau < -\log(-T)$ and $R > 1$ we have
 $|\nabla^m u_{0,\tau}|, |\nabla^m u_{1,\tau}| \leq C^{-1}$ on $\IB^k_{R} \times \IS^{n-k} \subset \DD_\tau$ for $m = 0, \ldots, 100$.
Then we have the following bound on $\IB^k_{R-C} \times \IS^{n-k}$, where $r$ denotes the radial distance function on $\IR^k$,
\begin{equation} \label{eq_Harnack_inprop}
 |v_\tau|, |\nabla v_\tau|, \ldots, |\nabla^{10} v_\tau| \leq C e^r \big( \Vert v_\tau \Vert_{L^2(\IB^k_{10} \times \IS^{n-k})} +  \Vert \Qu(\MM^0) \Vert^{10} e^{5\tau} \big). 
\end{equation}

\item \label{Prop_diff_properties_b} For large enough $A > 0$ there are constants $C(A,C^*,n), \eps'(A,C^*,n) > 0$ with the following property.
Suppose that
\[ \liminf_{\tau \to -\infty} e^{-2\tau} \Vert v \Vert_{L^2(\IB^k_{10} \times \IS^{n-k})} < \infty. \]
Then for every $r \geq C \Vert \Qu (\MM^0) \Vert$, any time $t < T$ and any point $\bq \in \IR^{k} \cong \IR^k \times \bO^{n-k+1}$ on the axis of rotation of $\MM^0$ the following is true.
Suppose that $r^{-1} (\MM^{0,\reg} - (\bq, t))$ is $\eps'$-close, at time $0$, to a submanifold $M \subset \IR^{n+1}$ which is equal to either $M_{\cyl}$ or to $S (\IR^{k-1} \times M_{\bowl})$ for some rotation $S \in O(k)$ of $\IR^k \times \bO^{n-k+1}$.
Then 
\[ (\spt \MM^i)_t \cap B(\bq, Ar) \subset \MM^{i,\reg}_t \] and the Hausdorff distance between both subsets satisfies the bound
\[ d_H \big( (\spt \MM^0)_t \cap B(\bq, Ar), (\spt \MM^1)_t \cap B(\bq, Ar) \big) \leq C\Big( \frac{\Vert \Qu(\MM^0) \Vert}{r} \Big)^{10} r. \]
\end{enumerate}
\end{Proposition}

We stress that Proposition~\ref{Prop_diff_properties} only summarizes the results of this section in a form that is convenient for subsequent sections.
Its conclusions are by no means optimal.
For example, the exponential growth rate in Assertion~\ref{Prop_diff_properties_a} could be improved to $e^{\delta r}$ and the term $\Vert \Qu (\MM^0) \Vert^{10} e^{5\tau}$ could be replaced by any term of the form $\Vert \Qu (\MM^0) \Vert^{2E} e^{E\tau}$ for $E \geq 1$.
None of these refinements, however, will be needed in the next sections.

\medskip
This section is organized as follows.
In Subsection~\ref{subsec_terminology_lead_m}, we introduce important terminology that will be used throughout the section.
In Subsection~\ref{subsec_lead_mode}, we define the leading mode condition and state the key result, which establishes this condition for any pair of flows.
The proof of the leading mode condition is carried out in Subsections~\ref{subsec:Preparatory}--\ref{subsec_leading_mode_condition_proof}; the structure of these sections will be summarized at the end of Subsection~\ref{subsec_lead_mode}.
Finally, Subsection~\ref{subsec_lead_mode_last} explains how the leading mode condition can be converted into Proposition~\ref{Prop_diff_properties}.

\bigskip

\subsection{Terminology} \label{subsec_terminology_lead_m}
In the following we will fix dimensions $1 \leq k < n$ and omit dependence on these constants for the remainder of this section.
We will often write $\IR^k$ instead of $\IR^k \times \bO^{n-k+1}$.

We first introduce the following terminology, which will be a slightly more convenient variant of the condition of being a center of an $\eps$-neck.

\begin{Definition}[$\eps$-cylindricality] \label{Def_eps_cyl}
Let $\MM$ be an $n$-dimensional ancient asymptotically cylindrical mean curvature flow on $\IR^{n+1} \times (-\infty,T)$.
We say that $\MM$ is \textbf{$\eps$-cylindrical} at a point $(\bp, t) \in \IR^{n+1} \times (-\infty,T)$ if the Gaussian area at time $t$ satisfies the bound
\[\sup_{\Delta T > 0} \Theta^{\MM}_{(\bp, t+\Delta T)} ( \Delta T) \geq \Theta_{\IR^{k} \times \IS^{n-k}} - \eps, \]
where $\Theta_{\IR^{k} \times \IS^{n-k}} $ is the entropy of the cylinder. \end{Definition}

The definition  of $\eps$-cylindricality is chosen to ensure that the characterizations in the following lemma hold.

\begin{Lemma} \label{Lem_eps_cyl}
There is an $\eps_{\cyl} > 0$ and a universal continuous function $\Psi_{\cyl} : [0,1) \to [0,1)$ with $\psi(0) = 0$ such that the following is true: 
\begin{enumerate}[label=(\alph*)]
\item \label{Lem_eps_cyl_a} Definition~\ref{Def_eps_cyl} is invariant under (parabolic) rescaling.
\item \label{Lem_eps_cyl_b} If $\MM$ is $\eps_1$-cylindrical at $(\bp,t)$, then it is also $\eps_2$-cylindrical at $(\bp,t)$ for all $\eps_2 \geq \eps_1$.
\item \label{Lem_eps_cyl_c} Consider two times $t_1 < t_2$
and assume that $\MM$ is $\eps_2$-cylindrical at $(\bp, t_2)$ for some $\eps_2 \in ( 0, \eps_{\cyl})$.
Then it is $\eps_1$-cylindrical at $(\bp, t_1)$ for some $\eps_1 < \eps_2$.
\item \label{Lem_eps_cyl_d} If $\MM$ is $\eps$-cylindrical at $(\bp,t)$, then $ \MM - (\bp,t)$ is $\Psi_{\cyl}(\eps)$-close to $M_{\cyl}$ at time $0$ and at some scale $r >0$.
\item \label{Lem_eps_cyl_e} Vice versa, if $ \MM - (\bp,t)$ is $\eps$-close to $M_{\cyl}$ at time $0$ and at some scale $r >0$, then $\MM$ is $\Psi_{\cyl}(\eps)$-cylindrical at $(\bp,t)$.
\end{enumerate}
\end{Lemma}

\begin{proof}
Assertions~\ref{Lem_eps_cyl_a} and \ref{Lem_eps_cyl_b} are clear and Assertion~\ref{Lem_eps_cyl_e} follows from a basic limit argument.

For Assertions~\ref{Lem_eps_cyl_d} fix some $\eps' > 0$.
It suffices to show that if $\MM$ is $\eps$-cylindrical at $(\bO,-1)$ for $\eps \leq \ov\eps(\eps')$, then $\MM$ must be $\eps'$-close to $M_{\cyl}$ at time $-1$.
Suppose that is not the case and choose a sequence of ancient asymptotically cylindrical flows $\MM^i$ such that $\Theta^{\MM^i}_{(\bO, -1+ r_i^2)}(r_i^2) \to \Theta_{\IR^k \times \IS^{n-k}}$ for some $r_i > 0$, but such that no $\MM^i$ is $\eps'$-close to $M_{\cyl}$ at time $-1$.
By parabolic rescaling, we may assume that $r_i = 1$.
Let $\la_i := \min \{ 1, \Vert \Qu(\MM^i) \Vert^{-1} \}$ and note that $\Vert \Qu(\la_i \MM^i) \Vert \leq 1$.
So by \cite[Proposition~\refx{Prop_Q_continuous}]{Bamler_Lai_PDE_ODI} we can pass to a subsequence such that we have convergence $\la_i \MM^i \to \MM^\infty$ in the Brakke sense to an asymptotically cylindrical flow.
The condition $\Theta^{\la_i \MM^i}_{(\bO, 0)}(\la_i^2) = \Theta^{\MM^i}_{(\bO, 0)}(1) \to \Theta_{\IR^k \times \IS^{n-k}}$ implies that $\MM^\infty = \MM_{\cyl}$, so $\la_i \Qu(\MM^i) = \Qu(\la_i \MM^i) \to \MM(\MM^\infty) = 0$.
But this implies that $\la_i=1$ for large $i$ and thus $\MM^i \to \MM_{\cyl}$, in contradiction to our assumptions.

Assertion~\ref{Lem_eps_cyl_c}, for $\eps_1 \leq \eps_2$, follows using the monotonicity of the Gaussian area.
Now suppose by contradiction that we cannot choose $\eps_1 < \eps_2$ and suppose for simplicity that $(\bp,t_2) = (\bO,0)$.
If $\eps_2$ sufficiently small, then by Assertion~\ref{Lem_eps_cyl_d}, we have $\Theta^{\MM}_{(\bO, r^2)}(r^2) \to 0$ as $r \to 0$.
Moreover, if $r \to \infty$, then $r^{-1} \MM \to \MM_{\cyl}$ in the Brakke sense, so we also have $\Theta^{\MM}_{(\bO, r^2)}(r^2) = \Theta^{r^{-1} \MM}_{(\bO, 1)}(1)\to 0$.
So the supremum in Definition~\ref{Def_eps_cyl} is attained for some $r > 0$.
It follows that $t \mapsto \Theta^{\MM}_{(\bO, r^2)} (r^2 - t)$ is constant on $[t_1, 0]$, so $\MM$ restricted to this time-interval must be a shrinker.
If $\eps_2$ is sufficiently small, then using Assertion~\ref{Lem_eps_cyl_d}, we obtain from \cite[Theorem~\refx{Thm_stability_necks}]{Bamler_Lai_PDE_ODI} that $\MM_{t_1}$ is a rescaling of $M_{\cyl}$, which implies that $\MM$ is even $0$-cylindrical at $(\bO,t_1)$.
\end{proof}

Let now $\MM$ be a convex and rotationally symmetric, asymptotically cylindrical mean curvature flows $\MM$.
We will frequently consider smooth functions $u : \DD \to \IR$, for some open domain $\DD \subset \MM^{\reg}$, whose values we regard as having the dimension of length. 
To facilitate our analysis, we will often pull back $u$ to a function defined on the standard round cylinder $M_{\cyl}$, which we call its cylindrical model.

\begin{Definition}[Cylindrical model] \label{Def_cyl_model}
Let $M \subset \IR^{n+1}$ be a convex and $(n,k)$-rotationally symmetric submanifold and $u : D \to \IR$ a  function over an open domain $D \subset M$.
Consider point $\bq \in \IR^{k} \cong \IR^k \times \bO^{n-k+1}$ that is covered by $M$ and assume that $r := \rho(\bq) > 0$.
Then the \textbf{cylindrical model} of $u$ at $\bq$ is the function $\td u_{\bq} : \td D_{\bq} \to \IR$, for $\td D_{\bq} \subset M_{\cyl}$, such that the following is true for any $(\bx,\by) \in M_\cyl$. 
We have $(\bx,\by) \in \td D_{\bq}$ if and only if $( \bq + r \bx, \rho(\bq + r \bx) \by) \in D$ and
\[ \td u_{\bq} (  \bx,\by ) := r^{-1} u \big( \bq + r \bx, \rho(\bq + r \bx) \by \big). \]
If $M = \MM_t^{\reg}$ is a time-slice of the regular part of a mean curvature flow and $u : \DD \to \IR$, $\DD \subset \MM^{\reg}$, then the cylindrical model of $u$ at $(\bq, t)$ is defined via the restriction $u|_{\DD_t}$ to the time-slice $\MM_t^{\reg}$ and it is expressed as $u_{\bq,t} : \td\DD_{\bq,t} \to \IR$.
\end{Definition}

In this section, $u : D \to \IR$ will describe a (linear or non-linear) perturbation of a cylindrical mean curvature flow. 
A key step in our proof is to show that $u$ is locally modeled by elements of $\sV_{\geq 0}$ near cylindrical regions.
In order to make this statement more precise, we use the following definition.

\begin{Definition}[Leading mode approximation] \label{Def_lead_mode_approx}
Suppose that $M \subset \IR^{n+1}$ is rotationally symmetric and $u : D \to \IR$ a smooth function over an open domain $D \subset M$.
Consider its cylindrical model $\td u_{\bq} : \td D_{\bq} \to \IR^l$ at some point $\bq \in \IR^k$.
If $\ov\IB^k_1 \times \IS^{n-k}  \subset \td D_{\bq}$, then the \textbf{leading mode approximation} of $u$ at $\bq$ is the unique element $U_{\bq} \in \sV_{\geq 0}$ that minimizes the norm 
\[ \big\Vert \td u_{\bq} - U_{\bq} \big\Vert_{L^2(\IB^k_1 \times \IS^{n-k} )}. \]
If $\ov\IB^k_1 \times \IS^{n-k}  \not\subset \td D_{\bq}$, then we say that the leading mode approximation does not exist at $\bq$.
Similarly, we define $U_{\bq,t}$ if $M = \MM^{\reg}_t$ is a time-slice of the regular part of a mean curvature flow and $u : \DD \to \IR$ for $\DD \subset \MM^{\reg}$.
\end{Definition}

Note that $U_{\bq,t}$ arises from a linear projection map in the space $L^2(\IB^k_1 \times \IS^{n-k})$, so it is uniquely defined since the restriction map $\sV_{\geq 0} \to L^2(\IB^k_1 \times \IS^{n-k} )$ is injective.
It also follows that $U_{\bq,t}$ depends linearly on $u$, it is smooth in the parameter $(\bq,t)$ and the set of such points $(\bq,t)$ for which it is defined is open.
Moreover, $U_{\bq,t}$ depends only on the values of $u$ restricted to $\MM_t^{\reg} \cap B(\bq, C \rho(\bq))$ for some dimensional constant $C$.
Recall also that $U_{\bq}$ is invariant under the rescaling $M \rightsquigarrow \la M$ and $u \rightsquigarrow (\bx \mapsto \la u(\la^{-1} \bx))$ and a similar property holds in the case in which $u$ is a function the regular part of a mean curvature flow.

Lastly, we define a useful norm, which will often use to bound $u : M \to \IR$ locally.

\begin{Definition} \label{Def_Vertu}
Suppose that $M$ and $u$ are as in Definition~\ref{Def_lead_mode_approx} and suppose that the leading mode approximation $U_\bq \in \sV_{\geq 0}$ exists for some $\bq \in \IR^k$.
For $\delta' > 0$ we set
\begin{equation} \label{eq_q_delta_norm_def}
 \Vert u \Vert_{\bq; \delta'} := \big\Vert \PP_{\sV_0}  U_\bq \big\Vert_{L^2_f} + \delta' \big\Vert \PP_{\sV_{\frac12}}  U_\bq \big\Vert_{L^2_f} 
 +  \delta^{\prime 2} \big\Vert \PP_{\sV_{1}}  U_\bq \big\Vert_{L^2_f} .
 + \delta^{\prime 3} \big\Vert \td u_{\bq} \big\Vert_{L^\infty( \IB_1^{k} \times \IS^{n-k} )}. 
\end{equation}
If the leading mode approximation does not exist at $\bq$, then we set
\[ \Vert u \Vert_{\bq; \delta'} := \infty. \]
Similarly, we define $\Vert u \Vert_{\bq,t; \delta'}$ if $M = \MM^{\reg}_t$ is a time-slice of the regular part of a mean curvature flow.
\end{Definition}

The reason for the choice of \eqref{eq_q_delta_norm_def} will become clear later in Lemma~\ref{Lem_Vgeq0_Prop_1}.
Roughly speaking, the norm $\Vert u \Vert_{\bq, \delta'}$ is chosen in such a way that it is almost constant in $\bq$ and almost monotone in time whenever $u$ describes a nearby mean curvature flow. 
We also recall that $\Vert u \Vert_{\bq,t;\delta'}$ is again invariant under the rescalings $\MM \rightsquigarrow \la \MM$ and $u \rightsquigarrow (\bx \mapsto \la u(\la^{-1} \bx, \la^{-2}t))$.
\medskip

\subsection{Flow pairs and the leading mode condition} \label{subsec_lead_mode}
The proofs of the main result of this section relies on an analysis of the difference of two cylindrical mean curvature flows $\MM^0, \MM^0$ in regions where they are close and a key observation is that this difference can be well approximated locally by its leading mode.
Consequently, such differences can only grow or decay at bounded exponential rates in space and time.
This property, is made precise by the \emph{leading mode condition} below.
To formalize our statements, we will introduce the following notion.

\begin{Definition}[Flow pair] \label{Def_flow_pair}
Consider two cylindrical mean curvature flows $\MM^0$ and $\MM^1$ defined over the same time-interval $I = (-\infty, T)$ or $(-\infty,T]$ and assume that $\MM^0$ is convex and $(n,k)$-rotationally symmetric (with axis $\IR^k \cong \IR^k \times \bO^{n-k+1}$).
Then we call $(\MM^0, \MM^1)$ a \textbf{flow pair over the time-interval $I$.}
The scale function $\rho : \IR^k \times I \to \IR_+$ will always refer to $\MM^0$.

Given a flow pair $(\MM^0, \MM^1)$, we define its \textbf{graph function} $u \in C^\infty(\DD)$, for $\DD \subset \MM^{0,\reg}$, so that $\MM^{1,\reg}$ is locally the normal graph of $u$ over $\MM^{0,\reg}$.
To be precise, for any $(\bp,t) \in \MM^{0,\reg}$ let $\nu_{\bp,t}$ be the outward unit normal vector of $\MM^{0,\reg}_t$ at $\bp$ (oriented so that the mean curvature is of the form $-H_{\bp,t}\nu_{\bp,t}$ for $H_{\bp,t} > 0$) and let $r_{\bp,t} > 0$ be the normal injectivity radius at $\bp$.
Consider the intersection of $(-0.1 r_{\bp,t}, 0.1 r_{\bp,t}) \nu_{\bp,t}$ with $(\spt \MM^1)_t$.
Then $(\bp,t) \in \DD$ if and only if this intersection consists of a single point that also belongs to the regular part of $\MM^1$.
In this case we express this intersection as $\bp + u(\bp,t) \nu_{\bp,t} $.
\end{Definition}

We emphasize that only the base flow $\MM^0$ is assumed to be smooth, convex and rotationally symmetric; no such assumption is made for $\MM^1$. 
These assumptions on $\MM^0$ are not essential, our arguments could be carried out without them, but they greatly simplify the exposition. 
Moreover, we will eventually establish that all asymptotically cylindrical flows are, in fact, convex and rotationally symmetric (possibly with respect to a different axis), so these assumptions ultimately make no difference.

We can now define the leading mode condition.

\begin{Definition}[Leading mode condition] \label{Def_lead_mode}
Let $\delta,\lb \delta',\lb C_0,\lb \eps,\lb \alpha > 0$ and  $D,\beta,C_1 \geq 0$.
We say that a flow pair $(\MM^0, \MM^1)$ with graph function $u$ over a time-interval $I$ satisfies the \textbf{$(\delta,\lb \delta',\lb C_0,\lb \eps,\lb D,\lb \alpha, \lb\beta,\lb C_1)$-leading mode condition} if the following is true for all $(\bq, t) \in \IR^{k} \times I$ at which $\MM^0$ is $\eps$-cylindrical.
Suppose that for $r := \rho(\bq, t)$
\begin{equation} \label{eq_lead_mode_conditiona}
  \Vert u  \Vert_{\bq,t;\delta'} < a  \qquad \text{for some} \quad a \in \bigg( \max \bigg\{ \beta, C_1 \Big( \frac{\Vert \Qu(\MM^0) \Vert}{r} \Big)^{10}  \bigg\} \alpha, \alpha \bigg). 
\end{equation}
Then the following two properties hold :

\begin{enumerate}[label=(\arabic*)]
\item  \label{Def_lead_mode_1} 
The cylindrical model $\td u_{\bq,t}$ at $(\bq,t)$ is defined on $\IB^{k}_{\delta^{-1}} \times \IS^{n-k}$ and it is close to its leading mode approximation $U_{\bq, t}$ in the following sense 
\begin{equation} \label{eq_close_to_leading_mode}
 \big\Vert \td u_{\bq, t} -  U_{\bq, t} \big\Vert_{C^{[\delta^{-1}]}(\IB^{k}_{\delta^{-1}} \times \IS^{n-k})} \leq  \delta a.  
\end{equation}
Moreover, we have the bounds
\begin{equation} \label{eq_dbounds_norm}
 r \big| \partial_{\bq} \Vert u \Vert_{\bq,t;\delta'} \big| \leq \delta a, \qquad  -\delta a \leq  r^2  \partial_{t} \Vert u \Vert_{\bq,t;\delta'}  \leq (1+\delta) a 
\end{equation}
and for all $(\bq', t') \in \IR^k \times I$ with $|\bq' - \bq| < r$ and $|t' -t| < r^2$ we have $\Vert u \Vert_{\bq', t'; \delta'} \leq C_0 a < \infty$.
\item  \label{Def_lead_mode_2} For any $\bq' \in \IR^k$ with $|\bq' - \bq| < Dr$ we have $\MM_{t}^{0,\reg}(\bq') \subset \DD_t$ (see Definition \ref{Def_flow_pair}) and
\[  \sup_{\MM^{0,\reg}_{t} (\bq')} |u|(\cdot, t) \leq C_0 \exp \bigg(  \frac{|\bq'-\bq|}{r} \bigg) a r. \]
\end{enumerate}
\end{Definition}
\medskip

Notice that the leading mode condition is invariant under translations (in space and time) and parabolic rescaling, since $\Vert \Qu(\MM^0) \Vert$ has the dimension of length \cite[Proposition~\refx{Prop_Q_basic_properties}]{Bamler_Lai_PDE_ODI}.
Our main result will be the following.

\begin{Proposition} \label{Prop_leading_mode_condition}
If $\delta \leq \ov\delta$, $\delta' \leq \ov\delta'(\delta)$, $C_0 \geq \underline C_0 (\delta')$, $\eps \leq \ov\eps(  C_0)$, $D \geq \underline D ( \eps)$, $\alpha \leq \ov\alpha( D)$, $C^* > 0$ and $C_1 \geq \underline{C}_1 (\alpha,  C^*)$, then any flow pair $(\MM^0, \MM^1)$ satisfying
\begin{equation} \label{eq_scale_l_scale_rot}
 \Vert \Qu (\MM^1) \Vert \leq C^* \Vert \Qu (\MM^0) \Vert
\end{equation}
satisfies the $(\delta,\lb \delta',\lb C_0,\lb \eps,\lb D,\lb \alpha, \lb 0,\lb C_1)$-leading mode condition (note that here $\beta=0$).
\end{Proposition}

We will prove Proposition~\ref{Prop_leading_mode_condition} via an induction argument on the parameter $\beta$.
Specifically, we will first establish the leading mode condition for \emph{some} $\beta > 0$, as long as the other parameters are chosen suitably:

\begin{Lemma}[Start of the induction] \label{Lem_start_induction}
If $\delta \leq \ov\delta$, $\delta' \leq \ov\delta'(\delta)$, $C_0 \geq \underline C_0 (\delta,\delta')$, $\eps \leq \ov\eps(\delta, \delta')$, $D \geq 0$, $\alpha \leq \ov\alpha(\delta, \delta', D)$, $\beta , C^* > 0$ and $C_1 \geq \underline{C}_1 (\delta, \delta', C_0, \eps, D, \alpha, \beta, C^*)$, then any flow pair $(\MM^{0}, \MM^1)$ satisfying \eqref{eq_scale_l_scale_rot} satisfies the $(\delta,\lb \delta',\lb C_0,\lb \eps,\lb D,\lb \alpha, \lb \beta,\lb C_1)$-leading mode condition.
\end{Lemma}

Then we will show that the leading mode condition remains preserved if we successively reduce $\beta$ by a factor of $2$---again under suitable conditions on the parameters.
Thus we can let $\beta \to 0$, which proves Proposition~\ref{Prop_leading_mode_condition}.

\begin{Lemma}[Induction step] \label{Lem_induction_step}
If $\delta \leq \ov\delta$, $\delta' \leq \ov\delta'(\delta)$, $C_0 \geq \underline C_0 (\delta')$, $\eps \leq \ov\eps( \delta, \delta', C_0)$, $D \geq \underline D ( \delta, \delta', C_0,\eps)$, $\alpha \leq \ov\alpha(\delta,\delta',C_0, \eps, D)$, $\beta \leq \ov\beta(\delta,\delta', C_0, \eps, D)$, $C_1 \geq \underline C_1(\delta,\delta',C_0,\eps, D)$, then the following is true.
Suppose that a flow pair $(\MM^{0}, \MM^1)$ satisfies the $(\delta,\lb \delta',\lb C_0,\lb \eps,\lb D,\lb \alpha, \lb \beta,\lb C_1)$-leading mode condition.
Then it also satisfies the $(\delta,\lb \delta',\lb C_0,\lb \eps,\lb D,\lb \alpha, \lb \frac12 \beta,\lb C_1)$-leading mode condition.
\end{Lemma}

The remainder of this section is organized as follows.
We first discuss some preparatory results in Subsection~\ref{subsec:Preparatory}.
These will almost directly imply the induction step, Lemma~\ref{Lem_induction_step}, which is proved in Subsection~\ref{subsec:induction_step}.
Next, we prove the start of the induction in Subsections~\ref{subsec_start_induction}, which is based on a discussion of the bowl soliton case in Subsection~\ref{subsec_leading_mode_bowl}.
Then we prove Proposition~\ref{Prop_leading_mode_condition} in Subsection~\ref{subsec_leading_mode_condition_proof}.
Finally, we prove the main result of this section, Proposition~\ref{Prop_diff_properties}, in Subsection~\ref{subsec_lead_mode_last}.

\bigskip

\subsection{The linearized leading mode condition}\label{subsec:Preparatory}

In this subsection we establish the main ingredients that go in the proofs of Lemmas~\ref{Lem_start_induction} and \ref{Lem_induction_step}.
These concern the limiting case in which $u : \MM^0 \to \IR$ is a solution to the \emph{linearized} mean curvature flow equation, defined over the \emph{entire} regular part $\MM^{0,\reg}$ of a convex and rotationally symmetric mean curvature flow $\MM^0$.
As we will only consider one flow in this subsection, we will drop the superscript ``$0$'' and write $\MM^0 = \MM$ in this subsection.
Moreover, since $\MM$ is assumed to be smooth, we will write $\MM = \spt \MM = \MM^{\reg}$.
We will consider the linearized leading mode condition in two cases:
\begin{itemize}
\item $\MM = \MM_{\cyl}$. In this case we will show that $u(\cdot, t)$ is in fact contained in the space $\sV_{\geq 0}$.
This allows us to verify the bounds from Property~\ref{Def_lead_mode_1} from Definition~\ref{Def_lead_mode} directly.
\item $\MM = \IR^{k-1} \times \MM_{\bowl}$.
In this case we will bound $u$ near the cap region of $\MM$ in terms of its values on the neck-like region via an exponential weight.
This will imply a bound similar to that of Property~\ref{Def_lead_mode_2} from Definition~\ref{Def_lead_mode}.
\end{itemize}
Note that in both cases we have $\Qu(\MM) = 0$.
\medskip

We will start with the following lemma, which will be used to establish Property~\ref{Def_lead_mode_1} in the proof of Lemma~\ref{Lem_induction_step}.

\begin{Lemma} \label{Lem_Vgeq0_Prop_1}
If $\delta' \leq \ov\delta' (\delta)$ and $C_0 \geq \underline{C}_0 (\delta')$, then the following is true.
Suppose that $\MM$ is the round shrinking cylinder $\MM_{\cyl}$ restricted to some time-interval $I = (-\infty,T]$, for $T<0$, and $u : \MM \to \IR$ is a solution to the linearized mean curvature flow equation \eqref{eq_lin_MCF} with the property that $\bp \mapsto u((-t)^{1/2} \bp, t)$ is contained in $\sV_{\geq 0}$ for all $t \in I$.
Then Property~\ref{Def_lead_mode_1} from Definition~\ref{Def_lead_mode} holds for all $(\bq,t) \in \IR^k \times I$ with $a = \Vert u \Vert_{\bq, t; \delta'}$ and with $\delta$ replaced by $\frac12\delta$ and $C_0$ replaced by $\frac12 C_0$.
\end{Lemma}

\begin{proof}
It is clear that $\td u_{\bq, t} = U_{\bq, t}$, so all cylindrical models agree with the leading mode approximation, which shows \eqref{eq_close_to_leading_mode}.
To see the remaining bounds, we may assume  without loss of generality, after parabolic rescaling, that $(\bq,t) = (\bO,-1)$, so $r = \rho(\bq, t) = 1$.

Fix $\mathbf e \in \IR^{k}$, $|\mathbf e| = 1$, and set $U^{(s)} := U_{s \mathbf e,-1}$ and $U:= U^{(0)}$.
It is not hard to see that $U^{(s)} (\bp) = U(\bp - s \mathbf e)$.
Therefore
\begin{align*}
 \PP_{\sV_0} U^{(s)} &= \PP_{\sV_0} U, \\ 
\PP_{\sV_{\frac12}} U^{(s)} &=  \PP_{\sV_{\frac12}} U + s F_1 (\PP_{\sV_0} U), \\
\PP_{\sV_{1}} U^{(s)} &=  \PP_{\sV_{1}} U + s F_2 (\PP_{\sV_{\frac12}} U) + s^2 F_3 (\PP_{\sV_{0}} U),  
\end{align*}
for suitable linear maps $F_1 : \sV_0 \to \sV_{\frac12}$, $F_2 : \sV_{\frac12} \to \sV_1$ and $F_3 : \sV_{0} \to \sV_1$.
Hence, writing $\Vert \cdot \Vert = \Vert \cdot \Vert_{L^2_f}$ we have for some generic dimensional constant $C > 0$
\[ \tfrac{d}{ds} |_{s=0} \Vert \PP_{\sV_0} U^{(s)} \Vert = 0, \qquad
\big| \tfrac{d}{ds} |_{s=0} \Vert \PP_{\sV_{\frac12}} U^{(s)} \Vert \big| \leq C \Vert \PP_{\sV_0} U \Vert, \qquad
\big| \tfrac{d}{ds} |_{s=0} \Vert \PP_{\sV_1} U^{(s)} \Vert \big| \leq C \Vert \PP_{\sV_{\frac12}} U \Vert. \]
Moreover, we have
\begin{equation} \label{eq_ddtLinfty}
 \big| \tfrac{d}{ds} |_{s=0} \Vert  U^{(s)} \Vert_{L^\infty( \IB_1^{k} \times \IS^{n-k} )} \big| \leq C \Vert  U \Vert 
\end{equation}
Putting this together implies that
\[ \big| \tfrac{d}{ds} |_{s=0} \Vert  U^{(s)} \Vert_{\bO,0; \delta'} \big| \leq C \delta' \Vert  U \Vert_{\bO, \delta'}. \]
So the spatial derivative bound in \eqref{eq_dbounds_norm} holds if $C \delta' \leq \delta$.

To prove the bound on the time-derivative, observe that $U^{(\tau)} := U_{\bO, e^{-\tau}}$ satisfies the rescaled linearized mean curvature flow equation $\partial_\tau U^{(\tau)} = L U^{(\tau)}$, which implies that for $U = U^{(0)}$
\[    U^{(\tau)}=\PP_{\sV_0}U+e^{\tau/2} \PP_{\sV_{\frac12}}U + e^{\tau}\PP_{\sV_1}U,  \]
so
\[ \tfrac{d}{d\tau} |_{\tau=0} \Vert \PP_{\sV_0} U^{(\tau)} \Vert = 0, \qquad
\big| \tfrac{d}{d\tau} |_{\tau=0} \Vert \PP_{\sV_{\frac12}} U^{(\tau)} \Vert \big| = \tfrac12 \Vert \PP_{\sV_{\frac12}} U \Vert, \qquad
\big| \tfrac{d}{d\tau} |_{\tau=0} \Vert \PP_{\sV_1} U^{(\tau)} \Vert \big| =  \Vert \PP_{\sV_{1}} U \Vert. \]
It follows that
\[ 0 \leq \tfrac{d}{d\tau}|_{\tau=0} \big(  \Vert \PP_{\sV_0} U^{(\tau)} \Vert + \delta' \Vert \PP_{\sV_{\frac12}} U^{(\tau)} \Vert + \delta^{\prime 2} \Vert \PP_{\sV_1} U^{(\tau)} \Vert \big) \leq   \Vert \PP_{\sV_0} U \Vert + \delta' \Vert \PP_{\sV_{\frac12}} U \Vert + \delta^{\prime 2} \Vert \PP_{\sV_1} U \Vert , \]
and
\[  \big| \tfrac{d}{d\tau} |_{\tau=0} \Vert  U^{(\tau)} \Vert_{L^\infty( \IB_1^{k} \times \IS^{n-k} )} \big| \leq C \Vert  U \Vert \le C\delta'^{-2}\Vert  U \Vert_{\bO, \delta'}.
 \]
So the bound on the time-derivative in \eqref{eq_dbounds_norm} holds again if $C\delta' \leq \delta$.
The last statement of Property~\ref{Def_lead_mode_1} holds for $C_0 \geq \underline C_0(\delta')$.
\end{proof}

The next lemma shows that the assumption from Lemma~\ref{Lem_Vgeq0_Prop_1} automatically holds if we assume the linearized leading mode condition.

\begin{Lemma} \label{Lem_cyl_u}
If $\delta < \frac1{n-k}$, $\delta' \leq \ov\delta'(\delta)$ and $C_0 \geq \underline{C}_0(\delta')$ and $ C_0, \beta, C_1 > 0$, then the following is true.
Suppose that $\MM$ is the round shrinking cylinder $\MM_{\cyl}$ restricted to some time-interval $I = (-\infty,T]$, for $T<0$, and $u : \MM \to \IR$ is a solution to the linearized mean curvature flow equation \eqref{eq_lin_MCF}.

Assume that for any $t \in I$ both properties from Definition~\ref{Def_lead_mode} hold at $(\bq,t)=(\bO, t)$ for the constants $\delta, \delta',C_0$ and $D = \infty$ and $r = \rho(\bO, t)$ whenever
\[ \Vert u \Vert_{\bO,t;\delta'} < a, \qquad \text{and} \qquad a \in (\beta , \infty). \] 
Then $\bp \mapsto u((-t)^{1/2} \bp, t)$ is contained in $\sV_{\geq 0}$ for all $t \in I$.
\end{Lemma}

Note that it is a subtle but crucial point that we do not assume the properties of Definition~\ref{Def_lead_mode} to hold at points $\bq \neq \bO$. 
This omission matters, because in the proof of Lemma~\ref{Lem_improve_Prop1} the round shrinking cylinder arises from a limit procedure, where $\eps$ in the leading mode condition tends to zero.
While the limiting cylindrical flow is obviously cylindrical at every $(\bq,t) \in \IR^k \times I$, for the approximating sequence the $\eps$-cylindricality condition is only guaranteed to hold at the origin.

\begin{proof}
Recall that we have $\rho(\cdot, t) = (-t)^{1/2}$ for $\MM_{\cyl}$.
The bounds in \eqref{eq_dbounds_norm} imply that
\[ 
 - \delta (-t)^{-1} \Vert u \Vert_{\bO,t;\delta'} \leq \partial_{t} \Vert u \Vert_{\bO,t;\delta'} \leq (1 + \delta) (-t)^{-1} \Vert u \Vert_{\bO,t;\delta'}  \qquad \text{if} \quad \Vert u \Vert_{\bO,t;\delta'} > \beta. \]
Integrating this bound backwards in time implies that for some constant $C' > 0$, which may depend on $u$, 
\[  \Vert u  \Vert_{ \bO, t; \delta'} 
\leq C'(-t)^{\delta}+\beta. \]
Combining this with Property~\ref{Def_lead_mode_2} from Definition~\ref{Def_lead_mode} yields that for any $(\bq',t) \in \IR^k \in I$ (recall that $D=\infty$)
\begin{equation} \label{eq_expboundu}
 (-t)^{-1/2}  \sup_{\MM_t(\bq')} |u| (\cdot, t) \leq  \big( C' (-t)^{\delta}+ \beta\big) \cdot C_0 \exp \bigg( \frac{|\bq'|}{\sqrt{-t}} \bigg), 
\end{equation}
which implies the following bound for the weighted $L^2_f$-norm on the rescaled linearized mean curvature flow $\td u (\bp, \tau) := e^{\tau/2} u (e^{\tau/2} \bp, -e^{-\tau})$, for some uniform $C'' > 0$
\[ \Vert \td u(\cdot, \tau) \Vert_{L^2_f} \leq C'' ( e^{-\delta \tau}  +\beta). \]
Consider the splitting $L^2_f (M_{\cyl}) = \sV_{\geq 0} \oplus \sV_{<0}$ and write $\td u = \td u^{\geq 0} + \td u^{< 0}$.
Then, since the largest negative eigenvalue of $L$ is $-\frac1{n-k}$ (see \cite[Lemma~\refx{Lem_mode_dec}]{Bamler_Lai_PDE_ODI}), we obtain that for any $\tau_1 < \tau_2 < - \log(-T)$
\[ \Vert \td u^{< 0} (\cdot, \tau_2) \Vert_{L^2_f} 
\leq e^{-\frac1{n-k} (\tau_1- \tau_2)} \Vert \td u^{< 0} (\cdot, \tau_1) \Vert_{L^2_f} 
\leq e^{-\frac1{n-k} (\tau_1- \tau_2)} \cdot C'' ( e^{-\delta \tau_2} + \beta) . \]
As we've chosen $\delta < \frac1{n-k}$, the right-hand side goes to $0$ as $\tau_2 \to -\infty$.
Therefore, $\td u^{<0} \equiv 0$ and $\td u (\cdot, \tau) \in \sV_{\geq 0}$ for all $\tau < - \log(-T)$, which finishes the proof.
\end{proof}
\medskip

Next, we develop the key tool for establishing Property~\ref{Def_lead_mode_2} in the proof of Lemma~\ref{Lem_induction_step}.
For this purpose, we study solutions to the linearized mean curvature flow on $\IR^{k-1} \times \MM_{\bowl}$.
We begin with the following technical lemma.

\begin{Lemma} \label{Lem_v_on_bowl}
Consider the $(n-k+1)$-dimensional bowl soliton $\MM_{\bowl}$ with $\MM_{\bowl,t} = t \mathbf e_1 + M_{\bowl}$.
There is a $\gamma > 0$ and a compact subset $K \subset M_{\bowl}$ such that for any sufficiently small $\la > 0$ there is a smooth, positive super-solution $v_\la : \MM_{\bowl} \to \IR_+$ to the linearized mean curvature flow equation, i.e., 
\begin{equation} \label{eq_v_supersol}
 \partial_{\mathbf t} v_\la \geq \triangle v_\la +  |A|^2 v_\la, 
\end{equation}
that is rotationally invariant (so it only depends on the $x_1$-coordinate and time) and that satisfies   identity
\begin{equation} \label{eq_v_alltimes}
v_{\la}(\bp, t) = e^{- \gamma \la t} v_{\la}( \bp - t \mathbf e_1, 0).
\end{equation}
Moreover, if $H > 0$ denotes the scalar mean curvature function on $\MM_{\bowl}$, then we have 
\begin{equation} \label{eq_vla_exp_asymp}
 v_\la(\cdot, 0) = e^{- \la x_1} H \qquad \text{on} \quad M_{\bowl} \setminus  K 
\end{equation}
and $\sup_K \frac{v_\la}{H} (\cdot, 0) \leq 2 \sup_{\partial K} \frac{v_\la}{H} (\cdot, 0)$.
\end{Lemma}

\begin{proof}
Let us fix $\la > 0$ and drop the index in ``$v_\la$''.
In the following, we will define $v(\cdot, 0)$ at time $0$ and then use \eqref{eq_v_alltimes} to extend $v$ to all times.
Note that in terms of $v = v(\cdot, 0) : M_{\cyl} \to \IR$ the bound \eqref{eq_v_supersol} becomes
\begin{equation} \label{eq_supersol_static}
 \triangle v + \nabla_{\mathbf e_1^\Vert} v + |A|^2 v \leq  -\gamma \la v  , 
\end{equation}
where $\mathbf e_1^\Vert$ is the projection of $\mathbf e_1$ onto the tangent space of $M_{\bowl}$.

Recall the following two evolution identities for the scalar mean curvature $H$ and the restriction $x := x_1 |_{\MM_{\bowl}}$ of the first coordinate function
\[ \partial_{\mathbf t} H = \triangle H +  |A|^2 H, \qquad
\partial_{\mathbf t} x = \triangle x, \]
which can also be rewritten as static equations (note that $x (\bp, t) = x (\bp - t \mathbf e_1, 0)+ t$):
\[  \triangle H + \nabla_{\mathbf e_1^\Vert} H +  |A|^2 H = 0, \qquad
\triangle x +  \nabla_{\mathbf e_1^\Vert} x = 1, \]

We will now construct $v = v(\cdot, 0)$ via the following Ansatz:
\[ v  :=  \td v H , \qquad \td v :=
 F \circ x, \]
for some smooth function $F : \IR \to \IR_+$, which we will determine at the end of the proof.
Since
\begin{multline*}
 ( \triangle +  \nabla_{\mathbf e_1^\Vert}  + |A|^2) v
= ( \triangle +  \nabla_{\mathbf e_1^\Vert}) \td v \cdot  H + \td v \cdot ( \triangle +  \nabla_{\mathbf e_1^\Vert} +  |A|^2) H 
+2 \nabla \td v \cdot \nabla H  \\
=   (\triangle + \nabla_{\mathbf e_1^\Vert} ) \td v \cdot  H + 2  \nabla \td v \cdot \nabla H 
\end{multline*}
and
\begin{equation*}
 (\triangle +  \nabla_{\mathbf e_1^\Vert} ) \td v 
= (\triangle +  \nabla_{\mathbf e_1^\Vert} )F(x)
= F'(x) \cdot (\triangle +  \nabla_{\mathbf e_1^\Vert} ) x + F''(x) |\nabla x|^2 
= F'(x)  + F''(x) |\nabla x|^2,
\end{equation*}
we get that \eqref{eq_supersol_static} is equivalent to
\begin{equation} \label{eq_supersol_F}
 F'(x) H  + F''(x) |\nabla x|^2 H + 2 F'(x) \nabla_{\nabla x} H \leq - \gamma \la F(x) H. 
\end{equation}
Suppose for a moment that $F(s) = e^{- \la s}$.
Using $|\nabla x| \leq 1$ and the asymptotic bound $H \sim x^{-1/2}$ (see also Lemma~\ref{Lem_warping}), we obtain that the left-hand side of \eqref{eq_supersol_F} is bounded above by
\begin{equation} \label{eq_expressionFH}
 e^{- \la x} \big( - \la  +  \la^2  +  C\la x^{1/2} |\nabla H| \big) H  
\end{equation}
for some universal constant $C > 0$.
Since $s^{-1/2} (\MM - s \mathbf e_1 )$ smoothly converges to the round cylinder as $s \to \infty$, we get that $x^{1/2} |\nabla H| \to 0$.
So if $\la$ is sufficiently small and $x \geq X$ for some some constant $X > 0$, which is independent of $\la$, then \eqref{eq_expressionFH} is bounded above by $-\frac{1}2  \la e^{- \la x} H 
\leq - \gamma  \la F(x)H$ for some constant $\gamma > 0$, which is again independent of $\la$.
So \eqref{eq_supersol_F} holds on $\{ x \geq X \}$ for this choice of $\gamma$.

It remains to choose $F$ on $[0,X]$ so that \eqref{eq_supersol_F} continues to hold on $\{ x < X \}$, possibly after replacing $\gamma$ with a smaller universal constant.
To do this, let $c , A > 0$ be constants whose values we will determine later and set
\[ F(s) := \begin{cases} e^{-\la X} + c\la e^{-\la X} (1-(s/X)^A) & \text{if $s < X$} \\ e^{-\la s} & \text{if $s \geq X$} \end{cases} \]
It is clear that $F$ is continuous and if $c \leq \ov c(A)$, then we can ensure $F$ doesn't change more than by a factor of $2$ on $[0,K]$ and that $\frac{d}{ds^-} F (X) \geq \frac{d}{ds^+} F (X)$.
Therefore, if \eqref{eq_supersol_F} holds on the set $\{ x \neq X \}$, then we can replace $F$ with a suitable smoothing to ensure that \eqref{eq_supersol_F} holds everywhere, possibly after a slight reduction of $\gamma$.

It remains to verify \eqref{eq_supersol_F} on $\{ x < X \}$.
On $\{ x < X \}$ the left-hand side of \eqref{eq_supersol_F} equals
\[ c \la e^{-\la X} \cdot A \big( -X^{-1}  (x/X)^{A-1} H -  X^{-2} (A-1) (x/X)^{A-2} |\nabla x|^2 H - 2 c X^{-1}  (x/X)^{A-1} \nabla_{\nabla x} H \big).  \]
Since the third term in the parenthesis vanishes for $x = 0$, we can find a constant $c' > 0$ such that on $\{ x < c' \}$ it is dominated by the first term and hence the entire expression in the parentheses is negative.
Moreover, by choosing $A$ sufficiently large, we can ensure that the second term dominates on $\{ c' \leq x < X \}$, making the entire term in the paranthesis negative over $\{ x \leq X \}$
With $A$ and $c$ fixed to satisfy the required bounds, the left-hand side of  \eqref{eq_supersol_F} is bounded from above by $-c'' \la e^{-\la X}$ on $\{ x < X \}$, for some universal $c'' > 0$.
Since $H$ is uniformly bounded from below on $\{ x < X \}$, it follows that \eqref{eq_supersol_F} holds for some universal choice of $\gamma > 0$, which finishes the proof. 
\end{proof}
\medskip

We will use Lemma~\ref{Lem_v_on_bowl} to prove the following result, which will be the key step in the proof of Lemma~\ref{Lem_improve_Prop2}.

\begin{Lemma} \label{Lem_cap_control}
If $\delta \leq \ov\delta$, $\delta' > 0$, $C_0 \geq \underline C_0(\delta')$, $\eps \leq \ov\eps$, $D \geq \underline D(\eps)$, then the following is true for any choices of constants $C'_0, C_1 > 0$.

Consider a mean curvature flow $\MM$ obtained from $\MM_{\cyl}$ or $\IR^{k-1} \times \MM_{\bowl}$ by applying translation, a time-shift and/or a parabolic rescaling, and then restriction to a time-interval of the form $I = (-\infty,T]$.
Let $u :  \MM \to  \IR$ a smooth solution to the linearized mean curvature flow equation \eqref{eq_lin_MCF} on $\MM$.
Then Statement~\ref{st_A} below implies Statement~\ref{st_B}.
\begin{enumerate}[label=(\Alph*)]
\item \label{st_A}
For all $(\bq, t) \in \IR^k \times I$ at which $\MM$ is $\eps'$-cylindrical for some $\eps' < \eps$ (here $\eps'$ may depend on $(\bq,t)$) the following is true:
If
\begin{equation} \label{eq_strict_version_ua}
 \Vert u \Vert_{\bq, t; \delta'} < a \qquad \text{for some} \quad a \in \big(  \max \{ 1,  C_1 \rho^{-10}(\bq,t) \}, \infty \big), 
\end{equation}
then both Properties of Definition~\ref{Def_lead_mode} hold for the constants $a,\delta, \delta', D$ and $C_0$ replaced with $C'_0$.
\item \label{st_B}
For all $(\bq, t) \in \IR^k \times I$ at which $\MM$ is $\eps$-cylindrical the following is true:
If 
\begin{equation} \label{eq_ula_ass}
 \Vert u \Vert_{\bq, t; \delta'} \leq a \qquad \text{for some} \quad a \in \big[  \max \{ \tfrac12,  C_1 \rho^{-10}(\bq,t) \}, \infty \big), 
\end{equation}
then Property~\ref{Def_lead_mode_2} of Definition~\ref{Def_lead_mode} holds for the constants $a$, $C_0$ and for $D$ replaced with $\infty$.
\end{enumerate}
\end{Lemma}

The key insight of this lemma is that the arbitrary constant $C'_0$ can be replaced by a constant $C_0$ that depends only on $\delta'$.
The change from $D$ to $\infty$, the modified inequalities, and the condition involving $\eps'$ are purely technical: they guarantee that if $\MM$ and $u$ arise as limits of a sequence of flow pairs $(\MM_j^0, \MM_j^1)$ together with suitable rescalings of their graph functions, then both the assumption and the conclusion of the lemma can be passed consistently between the sequence and its limit.

\begin{proof}
We first settle the cylindrical case.

\begin{Claim}
The lemma is true if $\MM$ is a round shrinking cylinder.
\end{Claim}

\begin{proof}
Assume without loss of generality that $\MM = \MM_{\cyl}|_{(-\infty,T)}$ and $T< 0$.
Then as in the proof of Lemma~\ref{Lem_cyl_u}, we obtain a bound of the form $\Vert u \Vert_{\bO, t;\delta'} \leq C'((-t)^\delta+1)$.
Integrating the spatial derivative of \eqref{eq_dbounds_norm} at each time, implies a bound of the form \eqref{eq_expboundu} if $\delta \leq \ov\delta$.
So the proof of Lemma~\ref{Lem_cyl_u} implies that $\bp \mapsto u ((-1)^{1/2} \bp, t)$ is contained in $\sV_{\geq 0}$ and hence grows at most at a quadratic polynomial rate.
Its norm within $\sV_{\geq 0}$ is bounded by $C(\delta') \Vert u \Vert_{\bq, t;\delta'}$.
So Property~\ref{Def_lead_mode_2} of Definition~\ref{Def_lead_mode} is true as long as $C_0 \geq C_0(\delta')$.
\end{proof}

Let us now consider the bowl soliton case.
After application of a time-shift, it suffices to verify the assertion of the lemma at time $0$.
Since assumptions and assertions of the lemma are scaling invariant---as long as we adjust $C_1$---we may assume without loss of generality that $\MM = \IR^{k-1} \times \MM_{\bowl} |_{(-\infty,0]}$; here we assume that $\MM_{\bowl}$ is moving at speed $1$ in the positive direction.
After possibly applying another translation and a time-shift, it is enough to verify the assertion for some point of the form $\bq_0 = (\bO^{k-1}, q_0) \in \IR^k = \IR^{k-1} \times \IR$ at which $\MM$ is $\eps$-cylindrical.
So our goal will be to show the following statement:
\begin{enumerate}[label=(\Alph*$'$), start=2]
\item \label{st_Bp}
Suppose that, for $\bq_0 = (\bO^{k-1}, q_0)$, the point $(\bq_0,0)$ is $\eps$-cylindrical and set $r_0 := \rho(\bq_0, 0)$.
Fix $a_0$ such that
\begin{equation} \label{eq_ula_assp}
 \Vert u \Vert_{\bq, 0; \delta'} \leq a_0 \qquad \text{and} \quad a_0 \in \big[  \max \{ \tfrac12,  C_1 r_0^{-10} \}, \infty \big). 
\end{equation}
Then Property~\ref{Def_lead_mode_2} of Definition~\ref{Def_lead_mode} holds for the constants $a_0$, $C_0$ and for $D$ replaced with $\infty$.
\end{enumerate}

We will fix $a_0$ henceforth.
Since $a_0 \geq C_1 r_0^{-10}$ and $a_0 \geq \frac12$, Statement~\ref{st_A} implies the following weaker statement (note that the condition \eqref{eq_stronger_a} implies \eqref{eq_strict_version_ua}):
\begin{enumerate}[label=(\Alph*$'$)]
\item \label{st_Ap}
For all $(\bq, t) \in \IR^k \times \IR_{\leq 0}$ at which $\MM$ is $\eps'$-cylindrical for some $\eps' < \eps$ the following is true for $r : = \rho(\bq, t)$:
If
\begin{equation} \label{eq_stronger_a}
  \Vert u  \Vert_{\bq,t;\delta'} < a \qquad \text{for some} \quad a  > 2\max \bigg\{ 1,  \Big( \frac{r}{r_0} \Big)^{-10}  \bigg\} a_0 
\end{equation}
then both Properties of Definition~\ref{Def_lead_mode} hold for the constants $a,\delta, \delta', D$ and $C_0$ replaced with $C'_0$.
\end{enumerate}

Next, by the symmetries of $\MM$, for any $t \leq 0$ the set of points $\bq \in \IR^k$ such that $\MM$ is $\eps$-cylindrical at $(\bq,t)$ is of the form
\[  \{ x_k \geq q_{\eps} (t) = q_{\eps} + t \}\subset \IR^{k} \times \bO^{n-k+1}, \]
where $x_k$ is the $k$-th coordinate function and $q_\eps$ is strictly decreasing for small $\eps$ with $q_{\eps} \to \infty$  as $\eps \to 0$ (see Lemma~\ref{Lem_eps_cyl}).
So instead of requiring a bound of the form $\eps \leq \ov\eps$, we may in the following impose a bound of the form $q_\eps \geq \underline{q}$.
Since we have assumed that $\MM$ is $\eps$-cylindrical at $(\bq_0, 0)$, we have  $
 q_0 \geq q_{\eps} .$
By continuity, it even suffices to consider the case
\[  q_0 > q_{\eps} . \]
The condition in Statement~\ref{st_Ap} that a $\MM$ is $\eps'$-cylindrical at $(\bp, t)$ for some $\eps' < \eps$ is equivalent to 
\[ x_k(\bp) > q_\eps(t). \]
So due to our assumption, $\MM$ is even $\eps'$-cylindrical at $(\bq_0,0)$ for some $\eps' < \eps$.

Let $r_1 := \rho((\bO^{k-1}, q_\eps), 0)$ be the scale corresponding to the threshold $q_\eps$.
We now assume that $D$ is chosen such that $D r_1 > q_{\eps}$; since both $r_1$ and $q_\eps$ depend only on $\eps$ this condition follows from a bound of the form $D  \geq \underline{D}(\eps)$.
This choice ensures that Statement~\ref{st_Ap} can be used to bound $u$ on cap regions.
More precisely, consider the \emph{cap region} of $\MM_t$ defined by  
$$\MM_t^{x_k \leq q_\eps} := \bigcup_{\bq \in \IR^{k},x_k(\bq) \leq q_\eps(t)} \MM_t (\bq) = \IR^{k-1} \times \MM^{x_k \leq q_\eps}_{\bowl,t}$$
and define the \emph{cylindrical region} $\MM_t^{x_k > q_\eps} = \IR^{k-1} \times \MM^{x_k > q_\eps}_{\bowl,t}$ likewise.
Then Property~\ref{Def_lead_mode_2} from Definition~\ref{Def_lead_mode}, invoked in Statement~\ref{st_Ap}, provides a bound on $|u|(\cdot, t)$ on the cap cross-section $\bq^* \times \MM_{\bowl,t}^{x_k \leq q_\eps}$, for $\bq^* \in \IR^{k-1}$, in terms of the size of $u$ near its boundary.
This bound is summarized by the following claim.

\begin{Claim} \label{Cl_basic_extend_to_cap}
There is a constant $C''_0 > 0$, which may depend on $C'_0, \eps$, but not on time or space, such that if $D \geq \underline{D}(\eps)$, then the following is true for any $t \leq 0$ and $\bq = (\bq^*, q_\eps(t) + 1) \in \IR^k$.
If
\[  \Vert u \Vert_{\bq,t;\delta'} < a \qquad \text{for some} \quad  a  > 2 \max \bigg\{ 1,  \Big( \frac{\rho(\bq, t)}{r_0} \Big)^{-10}  \bigg\} a_0 , \]
then
\[ \sup_{\bq^* \times \MM_{\bowl,t}^{x_k \leq q_\eps}} |u|(\cdot, t) \leq C''_0 a. \]
\end{Claim}

We will now extend this bound bound onto the larger cap region
$$\MM_t^{x_k \leq q_0} := \bigcup_{\bq \in \IR^{k},x_k(\bq) \leq q_0 + t} \MM_t (\bq) = \IR^{k-1} \times \MM^{x_k \leq q_0}_{\bowl,t}.$$

\begin{Claim}
If $\delta \leq \ov\delta$ and $q_\eps \geq \underline q$, then for any $t \leq 0$ and any point $\bp \in  \partial \MM^{x_k \leq q_0}_t$ in the boundary of the cap region we have the bound
\begin{equation} \label{eq_bound_on_partial_cap}
 |u|(\bp, t) \leq C(\delta') \exp \bigg( -2\delta \frac{t}{r_0} + \delta \frac{|\proj_{\IR^{k-1}} (\bp)|}{r_0}  \bigg)  a_0  r_0 . 
\end{equation}
Moreover, for any $\bp \in   \MM^{\rho \leq \rho_0}_t$ within the cap region we have the bound
\begin{equation} \label{eq_bound_on_cap}
 |u|(\bp, t) \leq  C(\delta', q_0, C''_0) \exp \bigg( -2\delta \frac{t}{r_0} + \delta \frac{|\proj_{\IR^{k-1}} (\bp)|}{r_0}   \bigg)   a_0 r_0 . 
\end{equation}
\end{Claim}

\begin{proof}
Let $\bq \in \IR^k$ be a point with $x_k(\bq) = x_k(\bq_0) = q_0$.
Consider the unit speed straight line segment $s \mapsto \bq(s) \in \IR^{k-1} \times \{ q_0 \}$.
Since $\MM$ is $\eps'$-cylindrical at all $(\bq(s),0)$ for some $\eps' < \eps$, we obtain from Statement~\ref{st_Ap} that Property~\ref{Def_lead_mode_1} from Definition~\ref{Def_lead_mode} holds at $(\bq(s),0)$ for all $a$ satisfying \eqref{eq_stronger_a}, that is $\Vert u \Vert_{\bq(s),0;\delta'} \leq a$ and $a > 2a_0$.
Integrating the spatial derivative bound in \eqref{eq_dbounds_norm} therefore implies that
\[ \Vert u \Vert_{\bq, 0; \delta'} \leq 2
\exp \bigg(  \delta \frac{|\proj_{\IR^{k-1}} (\bq)|}{r_0}   \bigg)   a_0 \qquad \text{for all} \quad \bq \in \IR^{k-1} \times \{ q_0 \}. \]

Next, fix again $\bq = (\bq^*, q_0) \in \IR^{k-1} \times \{ q_0 \}$ and consider spacetime paths of the form $t \mapsto (\bq(t),t)$, where $\bq(t) := (\bq^*, q_0+t)$.
As in the previous paragraph, we can apply \eqref{eq_dbounds_norm} whenever $\Vert u \Vert_{\bq(t),t;\delta'} \leq a$ and $a > 2a_0$, which implies
\[ \frac{d}{dt}  \Vert u \Vert_{\bq(t),t;\delta'} \geq - r_0^{-1} \delta a - r_0^{-2} a = - \Big( \frac{\delta}{r_0} + \frac{\delta}{r_0^2} \Big) a. \]
Integrating this implies that for $t \leq 0$
\[ \Vert u \Vert_{\bq(t),t; \delta'}  \leq 2\exp \bigg(- \Big( \frac{\delta}{r_0} + \frac{\delta}{r_0^2} \Big) t + \delta \frac{|\proj_{\IR^{k-1}} (\bq)|}{r_0}   \bigg)  a_0. \]
By Definition~\ref{Def_Vertu} this implies a pointwise bound on $|u|$ over $\partial \MM^{x_k \leq q_0}_t$.
So \eqref{eq_bound_on_partial_cap} follows as long as $r_0 \geq 1$, which can be ensured by a bound of the form $q_0 \geq q_\eps \geq \underline q$.

To see \eqref{eq_bound_on_cap}, we integrate \eqref{eq_dbounds_norm} along segments within $\{ q_\eps \leq x_k \leq q_0 \}$ and use Claim~\ref{Cl_basic_extend_to_cap}.
\end{proof}

Next, we observe that the norm of $u$ is a sub-solution to the linearized mean curvature flow equation
\begin{equation} \label{eq_usubsolution_bowl}
\partial_{\mathbf t} |u| \leq \triangle |u| +  |A|^2 |u|. 
\end{equation}
Let $\la > 0$ be a constant, whose value we will determine in a moment, and consider the super-solution $v_\la : \MM_{\bowl} \to \IR_+$ to the linearized mean curvature flow equation on the bowl soliton from Lemma~\ref{Lem_v_on_bowl}.
Define the function $\ov v_\la : \MM = \IR^{k-1} \times \MM_{\bowl} \to \IR_+$ by
\[ \ov v_{\la} (\bp, t) :=  \exp \bigg( \frac{(k-1) \delta^2}{r_0^2}\, t \bigg) \cdot \prod_{i=1}^{k-1} \cosh \bigg( \frac{\delta x_i (\bp)}{r_0 }   \bigg) \cdot  v_\la \big(\proj_{\MM_{\bowl,t}}(\bp) , t\big). \]
Since $\ov v_\la$ is a product of a solution to the heat equation on $\IR^{k-1}$ and a super-solution to the linearized mean curvature flow equation on $\MM_{\bowl}$, it must be a super-solution to the linearized mean curvature flow equation on $\MM = \IR^{k-1} \times \MM_{\bowl}$, so
\begin{equation} \label{eq_super_barv}
\partial_{\mathbf t} \ov v_\la \geq \triangle \ov v_\la +  |A|^2 \ov v_\la.
\end{equation}

It is well known that we have $\frac{C^{-1}}{r_0} \leq H \leq \frac{C}{r_0}$ on $\partial\MM^{x_k \leq q_0}_t$ for some generic dimensional constant $C > 0$ (see also Lemma~\ref{Lem_warping} below).
So if $q_0 \geq q_\eps \geq \underline q$, we can use \eqref{eq_v_alltimes} and \eqref{eq_vla_exp_asymp} to bound $\ov v_\la (\bp, t)$ for any $\bp \in \partial\MM^{x_k \leq q_0}_t$ as follows:
\begin{equation} \label{eq_vbarla_lower}
 \ov v_\la (\bp, t) \geq \frac1{2^k} \exp \bigg( -\gamma\la t + \frac{(k-1)\delta^2}{r_0^2}\, t + \delta \frac{|\proj_{\IR^{k-1}} (\bp)|}{r_0}   \bigg) e^{-\la q_{0}} \cdot  \frac{C^{-1}}{r_0}. 
\end{equation}
Here we have used the bound $\cosh (x) \geq \frac12 e^{|x|}$ and $\gamma > 0$ is the universal constant from Lemma~\ref{Lem_v_on_bowl}.
Let us now assume that
\[ -\gamma\la  + \frac{(k-1)\delta^2}{r_0^2} \leq - \frac{3 \delta}{r_0}, \]
which can be ensured if we choose
\begin{equation} \label{eq_lambda_choice}
 \delta \leq \ov\delta, \qquad \la := \frac{4 \delta}{\gamma r_0}. 
\end{equation}
Then \eqref{eq_vbarla_lower} implies that we have the following bound on the \emph{boundary} of the cap region
\begin{equation} \label{eq_vbarla_lower_better}
 \ov v_\la (\bp, t) \geq \frac{1}{2^kC} \, \frac{1}{r_0} e^{-\la q_{0}} \exp \bigg( - 3 \delta \frac{t}{r_0}  + \delta \frac{|\proj_{\IR^{k-1}} (\bp)|}{r_0}   \bigg)  \qquad \text{if} \quad \bp \in \partial \MM^{x_k \leq q_0}_t. 
\end{equation}
Due to the positivity of $v_\la$, we moreover obtain a lower bound on $\ov v_\la$ on the \emph{entire}  cap region. 
That is for some $c''(q_0,\la) >0$, which may depend on $q_0$ and $\la$, but is independent of space and time,
\begin{equation} \label{eq_vbarla_lower_entire}
 \ov v_\la (\bp, t) \geq c''(q_0,\la) \exp \bigg( - 3 \delta \frac{t}{r_0}  + \delta \frac{|\proj_{\IR^{k-1}} (\bp)|}{r_0}   \bigg)  \qquad \text{if} \quad \bp \in  \MM^{x_k \leq q_0}_t. 
\end{equation}

Combining \eqref{eq_vbarla_lower_better} and \eqref{eq_vbarla_lower_entire} with \eqref{eq_bound_on_partial_cap} and \eqref{eq_bound_on_cap} implies that we have a bound of the form
\begin{equation} \label{eq_u_less_v}
 |u| \leq C(\delta') r_0 e^{\la q_{0}} \, \ov v_\la  a_0 r_0  
\end{equation}
on the parabolic boundary $\partial \MM^{x_k \leq q_0} \cup \MM^{x_k \leq q_0}_t$ for $t \ll 0$.
An application of the maximum principle to the difference of both sides, combined with \eqref{eq_usubsolution_bowl} and \eqref{eq_super_barv}, therefore implies that this bound holds on all of $\MM^{x_k \leq q_0}_t$.
We will now bound the right-hand side of \eqref{eq_u_less_v} from above.
By Lemma~\ref{Lem_v_on_bowl}, our choice of $\lambda$ in \eqref{eq_lambda_choice}, the fact that $\cosh x \leq e^{|x|}$ and again the asymptotic bound $H \sim \sqrt{x_k}$ (see Lemma~\ref{Lem_warping}), we have for any $\bp \in \MM^{x_k \leq q_0}_0$
\begin{multline*}
 e^{\la q_0} \ov v_\la (\bp,0) \leq 2 \exp \bigg( \frac{\delta}{r_0} \sum_{i=1}^{k-1} |x_i(\bp)| + \la (q_{0} - x_{k}( \bp) ) \bigg) H(\bp, 0)\\
 \leq \frac{C}{\sqrt{x_{k}(\bp)+1}} \exp \bigg( \frac{(k-1) \delta + 4\gamma^{-1} \delta}{r_0} |\bq_0 - \proj_{\IR^{k}} (\bp)  | \bigg) .  
\end{multline*}
So from \eqref{eq_u_less_v} we obtain that for any $\bq' \in \IR^{k}$ with $x_{k} (\bq') \leq x_{k}(\bq_0) = q_0$, assuming $\delta \leq \ov\delta$,
\begin{align}
 \sup_{\MM_0(\bq')} |u|(\cdot, 0)
&\leq \frac{C(\delta') r_0}{\sqrt{x_{k}(\bq')+1}} \exp\bigg( \frac{|\proj_{\IR^{k}}( \bq_0 - \bq' )|}{2r_0} \bigg) a_0 r_0 \notag \\
&\leq C(\delta') \frac{\sqrt{x_{k}(\bq_0)+1}}{\sqrt{x_{k}(\bq')+1}} \exp\bigg( - \frac{|x_{k}(\bq_0) - x_{k}(\bq')|}{2r_0} + \frac{|\proj_{\IR^{k}}(\bq_0 - \bq')|}{r_0}  \bigg) a_0 r_0. \label{eq_supu_bound_complicated}
\end{align}
Set $y := x_{k}(\bq')+1$ and $z :=  x_{k}(\bq_0)+1$, so $1 \leq y \leq z$ and use again the fact that $r_0 \leq C \sqrt{z}$.
Then
\[ \sqrt{\frac{z}{y}} \exp \Big({-\frac{z-y}{C \sqrt{z}}}\Big) 
\leq \sqrt{\frac{z}{y}} \exp \Big({-\frac{\sqrt{z}}{C}}\Big) 
\leq \sqrt{\frac{z}{y}} \exp \Big({-  \frac1C \sqrt{\frac{z}{y}} }\Big), \]
which is uniformly bounded from above.
Combining this with \eqref{eq_supu_bound_complicated}, implies the desired bound from Property~\ref{Def_lead_mode_2} of Definition~\ref{Def_lead_mode} for $C_0 \geq \underline C_0(\delta')$, whenever $\bq' \in \IR^{n-k}$ with $x_{k} (\bq') \leq x_{k}(\bq_0)$.

Lastly, consider the case $x_{k} (\bq') \geq x_{k}(\bq_0) = q_0$.
Fix $\bq'$ and let $s \mapsto \bq(s)$ be the unit speed line segment between $\bq_0$ and $\bq'$.
Then $x_k (\bq(s)) > q_0$, so $\MM$ is $\eps'$-cylindrical at $(\bq(s),0)$ and we can apply Statement~\ref{st_Ap} for $r = \rho(\bq(s),0) \geq \rho(\bq_0,0) = r_0$.
So \eqref{eq_dbounds_norm} from Property~\ref{Def_lead_mode_1} of Definition~\ref{Def_lead_mode} implies
\[ \frac{d}{ds} \Vert u \Vert_{\bq(s),0;\delta'} \leq \delta r_0^{-1} \Vert u \Vert_{\bq(s),0;\delta'} \qquad \text{if} \quad \Vert u \Vert_{\bq(s),0;\delta'} \geq 2 a_0. \]
Integrating this bound implies that for $\delta \leq \frac12$
\[ \Vert u \Vert_{\bq',0;\delta'} \leq 2 \exp\Big(\frac{|\bq'-\bq_0|}{2r_0}\Big) a_0. \]
So by Definition~\ref{Def_Vertu}
\begin{equation*}
 \sup_{\MM_{0}(\bq')} |u|(\cdot,0) \leq C(\delta')  \exp\Big(\frac{|\bq'-\bq_0|}{2r_0}\Big) \frac{\rho(\bq',0)}{r_0}   a_0 r_0
 \end{equation*}
 If $\rho(\bq', 0) \leq 10 \rho(\bq_0,0)$, then this shows again Property~\ref{Def_lead_mode_2} of Definition~\ref{Def_lead_mode} for $C_0 \geq \underline C_0(\delta')$.
 If $\rho(\bq', 0) > 10 \rho(\bq_0,0)$, then we have $|\bq' - \bq_0| \geq \rho(\bq',0)$, so
 \[ \sup_{\MM_{0}(\bq')} |u|(\cdot,0) \leq C(\delta')  \exp\Big(\frac{|\bq'-\bq_0|}{2r_0} + \frac{\rho(\bq',0)}{2r_0} \Big)   a_0 r_0
 \leq C(\delta')  \exp\Big(\frac{|\bq'-\bq_0|}{r_0}  \Big)   a_0 r_0, \]
 which also shows Property~\ref{Def_lead_mode_2} of Definition~\ref{Def_lead_mode}  for $C_0 \geq \underline C_0(\delta')$.
So Property~\ref{Def_lead_mode_2} of Definition~\ref{Def_lead_mode} holds as long as $q_0 \geq q_\eps \geq \underline q$ and $C_0 \geq \underline C_0(\delta')$.
This finishes the proof.
\end{proof}
\medskip

\subsection{Proof of the induction step, Lemma~\ref{Lem_induction_step}}\label{subsec:induction_step}
We will carry out the induction step in the following two lemmas.
In the first lemma we show that, under appropriate choice of the constants, we can generalize Property~\ref{Def_lead_mode_1} to the case where $\beta$ is replaced with $\frac12 \beta$.

\begin{Lemma} \label{Lem_improve_Prop1}
If $\delta \leq \ov\delta$, $\delta' \leq \ov\delta' (\delta)$, $C_0 \geq \underline C_0(\delta')$, 
$\eps \leq \ov\eps (\delta, \delta', C_0)$,
 $D \geq \underline D(\delta, \delta', C_0)$,  $\alpha \leq \ov\alpha(\delta,\delta', C_0)$ and $\beta \leq \ov\beta(\delta,\delta', C_0)$ and $C_1 > 0$, then the following is true.
Suppose that a flow pair $(\MM^{0}, \MM^1)$ satisfies the $(\delta,\lb \delta',\lb C_0,\lb \eps,\lb D,\lb \alpha, \lb\beta,\lb C_1)$-leading mode condition.
Then Property~\ref{Def_lead_mode_1} in Definition~\ref{Def_lead_mode} even holds under the weaker condition that $\MM^0$ is $\eps$-cylindrical at $(\bq, t)$ and
\begin{equation} \label{eq_relaxed_cond_Prop_1}
  \Vert u  \Vert_{\bq,t;\delta'} < a \qquad \text{for some} \quad a \in \bigg( \max \bigg\{ \frac12 \beta, C_1 \Big( \frac{\Vert \Qu(\MM^{0} \Vert}{r} \Big)^{10}  \bigg\} \alpha, \alpha \bigg) 
\end{equation}
\end{Lemma}

\begin{proof}
Suppose that $\delta \leq \ov\delta$ and $\delta' \leq \ov\delta'(\delta)$ and $C_0 \geq \underline C_0(\delta')$ are chosen so that Lemmas~\ref{Lem_Vgeq0_Prop_1} and \ref{Lem_cyl_u} can be applied.
Fix $\delta, \delta', C_0$, suppose that the lemma was false and choose sequences $\eps_j, \alpha_j, \beta_j \to 0$, $D_j \to \infty$ and $C_{1,j} > 0$.
Consider a sequence of counterexamples $(\MM_{j}^0, \MM_j^1)$ that satisfy the $(\delta,\lb \delta',\lb C_0,\lb \eps_j,\lb D_j,\lb \alpha_j, \lb\beta_j,\lb C_{1,j})$-leading mode condition, but for which Property~\ref{Def_lead_mode_1} fails for some point $(\bq_j, t_j)$ and for some number $a_j$ that satisfies the relaxed condition \eqref{eq_relaxed_cond_Prop_1}.
Let $u_j : \DD_j \to \IR$ be the corresponding graph functions.
After parabolic rescaling and application of a time-shift, we may assume without loss of generality that $\bq_j$ is the origin, $t_j =0$ and $r_j = \rho^{\MM_{j}^0}(\bO, 0) = 1$.
Then $\MM_j^0$ is $\eps_j$-cylindrical at $(\bO,0)$ and since $\eps_j \to 0$ we have smooth convergence $\MM_j^0 |_{(-\infty,0]} \to \MM_\infty$ to a round shrinking cylinder (see Lemma~\ref{Lem_eps_cyl}).

Since the $(\delta,\lb \delta',\lb C_0,\lb \eps_j,\lb D_j,\lb \alpha_j, \lb\beta_j,\lb C_{1,j})$-leading mode condition holds, but $(\mathbf0, 0)$ and $a_j$ violate Property~\ref{Def_lead_mode_1}, the factor $\frac12$ in \eqref{eq_relaxed_cond_Prop_1} in front of $\beta_j$ is essential, so we must have
\begin{equation} \label{eq_ai_bounds}
a_j\leq\beta_j \alpha_j \qquad \text{and} \qquad a_j \in \Big(\max\Big\{\frac12 \beta_j,C_{1,j} \Vert \Qu (\MM_{j}^0)\Vert^{10} \Big\}\alpha_j,  \alpha_j\Big).
\end{equation}
Let us now consider an arbitrary time $t \leq 0$.
By the backwards preservation of cylindricality (see Lemma~\ref{Lem_eps_cyl}) the flow $\MM_j^0$ is also $\eps_j$-cylindrical at $(\bO,t)$ for all $j$.
Moreover, by the convexity of $\MM^0_j$ and $\rho^{\MM_{j}^0}(\bO, 0) = 1$ we have $\rho^{\MM_j^0}(\bO,t) \geq 1$.
It follows, using \eqref{eq_ai_bounds}, that Properties~\ref{Def_lead_mode_1} and \ref{Def_lead_mode_2} hold at $(\mathbf0,t)$ for large $j$ and for any
\begin{align}\label{eq:condition_a}
a\in (2a_j,\alpha_j)\subset \Big(\max\Big\{\beta_j,C_{1,j} \Big(\frac{\Vert \Qu (\MM^{0}_j) \Vert}{\rho^{\MM^{0}_j}(\bO,t)}\Big)^{10} \Big\}\alpha_j, \alpha_j\Big).
\end{align}
Note that by \eqref{eq_ai_bounds} the endpoints of the interval $(2a_j, \alpha_j)$ satisfy
\[ \frac{\alpha_j}{2a_j} \geq \frac{\alpha_j}{2\beta_j\alpha_j} = \frac1{2\beta_j} \lto \infty \]
and recall that by assumption $\Vert u_j \Vert_{\bO,0;\delta'} < a_j$.
So by iterating the last bound in Property~\ref{Def_lead_mode_1}, we obtain a bound of the following form for any $T > 0$ and large $j$ (depending on $T$):
\[ \sup_{t \in [-T,0]} \Vert u_j \Vert_{\bO,t;\delta'} \leq C(T,C_0) a_j. \]
Using Property~\ref{Def_lead_mode_2} and the fact that $D_j \to \infty$, this bound can be extended to a pointwise bound of the following form for any $R > 0$ and for large $j$ (depending on $T$ and $R$)
\[ \sup_{t \in [-T,0]} \sup_{\MM^0_{j,t} \cap \IB^{n+1}_{R}} |u_j|(\cdot, t) \leq C(T,R, C_0) a_j. \]
By standard parabolic theory, this bound implies similar bounds on all higher derivatives of $u$.

We can hence pass to a subsequence such that $\frac{u_j}{a_j}$ locally smoothly converges to a solution $u_\infty : \MM_{\infty} \to \IR$ of the linearized mean curvature flow equation, which satisfies $\Vert u_\infty \Vert_{\bO,0;\delta'} \le 1$. 
By our previous discussion, for all $t \leq 0$ and large enough $j$, Properties~\ref{Def_lead_mode_1} and \ref{Def_lead_mode_2} hold at $(\bO,t)$  for $\MM^0_j$ and $\frac{u_j}{a_j}$ whenever $a \in (2, \frac{\alpha_j}{a_j})$.
Since $\frac{\alpha_j}{a_j} \to \infty$, we obtain that the limit $u_\infty$ satisfies the assumptions of Lemma~\ref{Lem_cyl_u}.
Due to Lemmas~\ref{Lem_Vgeq0_Prop_1} and \ref{Lem_cyl_u} (we need to apply the second lemma for $\beta =2$), we obtain that Property~\ref{Def_lead_mode_1} even holds with $\delta, C_0$ replaced with $\frac12\delta, \frac12 C_0$.
This contradicts our assumption that Property~\ref{Def_lead_mode_1} fails at $(\bO,0)$ and for $a_j$ if $j$ is large enough.
\end{proof}
\medskip

Next, we show that, under different assumptions on the constants, we can generalize Property~\ref{Def_lead_mode_2} to the case in which $\beta$ is replaced with $\frac12 \beta$.

\begin{Lemma} \label{Lem_improve_Prop2}
If $\delta \leq \ov\delta$, $\delta' > 0$, $C_0 \geq \underline C_0(\delta')$, $\eps \leq \ov\eps$, $D \geq \underline D(\eps)$, $\alpha \leq \ov\alpha(\delta,\delta', C_0, \eps, D)$, $\beta \leq \ov\beta(\delta,\delta', C_0, \eps, D)$ and $C_1 \geq \underline C_1(\delta,\delta', C_0, \eps, D)$ then the following is true.
Suppose that a flow pair $(\MM^0, \MM^1)$ satisfies the $(\delta,\lb \delta',\lb C_0,\lb \eps,\lb D,\lb \alpha, \lb\beta,\lb C_1)$-leading mode condition.
Then Property~\ref{Def_lead_mode_2} in Definition~\ref{Def_lead_mode} even holds under the weaker condition that $\MM^{0}$ is $\eps$-cylindrical at $(\bq, t)$ and (notice the factor $\frac12$ in front of $\beta$):
\begin{equation} \label{eq_relaxed_cond_Prop_2}
  \Vert u\Vert_{\bq,t;\delta'} < a \qquad \text{for some} \quad a \in \bigg( \max \bigg\{ \frac12 \beta, C_1 \Big( \frac{\Vert \Qu(\MM^0)\Vert}{r} \Big)^{10}  \bigg\} \alpha, \alpha \bigg) 
\end{equation}
\end{Lemma}

\begin{proof}
We argue similarly as in the proof of Lemma~\ref{Lem_improve_Prop1}.
Choose and fix the constants $\delta \leq \ov\delta$, $\delta'> 0$, $C_0 \geq \underline C_0(\delta')$, $\eps \leq \ov\eps$ and $D \geq \underline{D}(\eps)$ to be strictly smaller/larger than required for Lemma~\ref{Lem_cap_control}.
We also impose additional bounds of the form $\eps \leq \ov\eps$ and $D \geq \underline{D}(\eps)$, which we will explain later.
Suppose that the lemma was false for these choices of constants and pick sequences $\alpha_j, \beta_j \to 0$ and $C_{1,j} \to \infty$.
Then we can find a sequence of counterexamples $(\MM_{j}^0, \MM^0_j)$ that satisfy the $(\delta,\lb \delta',\lb C_0,\lb \eps,\lb D,\lb \alpha_j, \lb\beta_j,\lb C_{1,j})$-leading mode condition, but for which Property~\ref{Def_lead_mode_2} fails at some $(\bq_j, t_j)$ at which $\MM_j^0$ is $\eps$-cylindrical and for an $a_j$, which satisfies the relaxed condition \eqref{eq_relaxed_cond_Prop_2}.

After parabolic rescaling and application of a time-shift, we may again assume without loss of generality that $(\bq_j,t_j) = (\bO, 0)$ is the origin and $\rho^{\MM_{j}^0}(\bO, 0) = 1$.
As in the proof of the last lemma, the factor $\frac12$ in \eqref{eq_relaxed_cond_Prop_2} in front of $\beta_j$ is again essential, so we must have $a_j \leq \beta_j \alpha_j$, which implies
\[ a_j \in (\tfrac12 \beta_j \alpha_j, \beta_j \alpha_j] \qquad \text{and} \qquad
C_{1,j} \Vert \Qu (\MM_{j}^0)\Vert^{10} \alpha_j < a_j \leq \beta_j \alpha_j. \]
The second inequality implies that for a subsequence
\[\frac{C_{1,j} \Vert \Qu(\MM_{j}^0)\Vert^{10}}{\beta_j} \to C_{1,\infty} \in [0,1] \qquad \text{and} \qquad \Vert \Qu(\MM_{j}^0) \Vert \to 0. \]
So by \cite[Proposition~\refx{Prop_Q_continuous}]{Bamler_Lai_PDE_ODI}, after passing to a subsequence, the flows $\MM_{j}^0$ smoothly converge to a limit $\MM_{\infty}$, which must be a time-shift and/or parabolic rescaling of $\MM_{\cyl}$ or $\IR^{k-1} \times \MM_{\bowl}$ with $\rho^{\MM_\infty}(\bO,0)=1$ or an affine plane or empty.
The last two cases can be excluded if we assume $\eps \leq \ov\eps$.
After passing to a subsequence, we may assume moreover that 
$$\frac{a_j}{\beta_j \alpha_j} \to a_\infty \in \big[\max\big\{\tfrac12,C_{1,\infty}\big\},1\big].$$

Consider now the graph functions $u_j$ of $(\MM_{j}^0, \MM^1_j)$.
Their rescalings $u^*_j := \frac{u_j}{\beta_j \alpha_j}$ satisfy Properties~\ref{Def_lead_mode_1} and \ref{Def_lead_mode_2} at $(\bq, t)$ whenever $\MM^0_j$ is $\eps$-cylindrical at $(\bq,t)$ and
\[ \Vert u^*_j \Vert_{\bq, t;\delta'} < a \qquad \text{for some} \quad a \in \bigg( \max \bigg\{ 1, \frac{C_{1,j}\Vert (\MM^0_j)\Vert^{10}}{\beta_j} (\rho^{\MM^0_j}(\bq,t ))^{-10} \bigg\}, \frac1{\beta_j} \bigg). \]
Since 
\[ \Vert u^*_j \Vert_{\bO,0;\delta'} = \frac{1}{\beta_j\alpha_j} \Vert u_j \Vert_{\bO,0;\delta'} < \frac{a_j}{\beta_j\alpha_j} \leq 1, \] 
we can iterate the last statement in Property~\ref{Def_lead_mode_1} to obtain local uniform bounds on $u^*_j$ over larger and larger bounded subsets of the $\eps$-cylindrical part of $\MM^0_j$.
If $\MM^0_j$ converges to a round shrinking cylinder, then this suffices to extract a subsequence such that we have local smooth convergence $u^*_j \to u_\infty$ to a solution of the linearized mean curvature flow equation.
If $\MM^0_j$ converges to translation and/or parabolic rescaling of $\IR^{k-1} \times \MM_{\bowl}$, then we can use Property~\ref{Def_lead_mode_2} to derive local uniform bounds on $u^*_j$ over the non-cylindrical, ``cap-like'', part as long as we choose $D \geq \underline{D}(\eps)$.
So we can also assume that we have local smooth subsequential convergence $u^*_j \to u_\infty$ in this case.
Taking the leading mode condition of $(\MM^0_j,\MM^1_j)$ to the limit, we obtain that whenever $\MM_{\infty}$ is $ \eps'$-cylindrical, for $\eps' < \eps$, at some $(\bq, t)$ and if
\[  \Vert u_\infty \Vert_{\bq,t;\delta'} < a \qquad \text{for some} \quad a \in \big( \max \big\{ 1, C_{1,\infty} \rho^{-10} (\bq,t) \big\} , \infty \big),  \]
then Properties~\ref{Def_lead_mode_1}, \ref{Def_lead_mode_2} in Definition~\ref{Def_lead_mode} hold for $u_\infty$ and our choice of constants $\delta,\delta', C_0$ and $D$.
Moreover,
\[ \Vert u_\infty  \Vert_{\bO,0;\delta'} \leq a_\infty \in \big[\max\big\{\tfrac12,C_{1,\infty}\big\},1\big]. \]
We are now in a position to apply Lemma~\ref{Lem_cap_control},
which implies that Property~\ref{Def_lead_mode_2} holds for $\MM_\infty$ and $u_\infty$ at $(\bO,0)$ for a strictly smaller constant $C_0$ and for $D=\infty$.
So Property~\ref{Def_lead_mode_2} must hold for $(\MM^0_j,\MM^1_j)$ at $(\bO, 0)$ for large $j$, in contradiction to our assumptions.
\end{proof}
\medskip

\begin{proof}[Proof of Lemma~\ref{Lem_induction_step}.]
This follows by combining Lemmas~\ref{Lem_improve_Prop1} and \ref{Lem_improve_Prop2}.
\end{proof}
\bigskip

\subsection{The leading mode condition in the case of bowls or cylinders}\label{subsec_leading_mode_bowl}
This subsection serves as a preparation for the proof of Lemma~\ref{Lem_start_induction}---the start of the induction for the leading mode condition.
Our goal will be to prove the following lemma, which establishes the leading mode condition in the case in which each of the flows $\MM^{0}$ and $\MM^1$ are cylinders or $\IR^{k-1} \times \MM_{\bowl}$ modulo translations and rescalings.

\begin{Lemma} \label{Lem_leading_mode_bowl}
If $\delta > 0$, $\delta' \leq \ov\delta'(\delta)$, $C_0 \geq \underline C_0 (\delta, \delta')$, $\eps \leq \ov\eps (\delta, \delta')$, $D \geq 0$, $\alpha \leq \ov\alpha(\delta, \delta',D)$, then the following is true.
Suppose that $(\MM^0, \MM^1)$ is a flow pair over a time-interval $I$, where each of the flows $\MM^0$ and $\MM^1$ is equal to the result of applying a combination of time-shift, translation, rescaling and rotation to either $\IR^{k-1} \times \MM_{\bowl}$ or $\MM_{\cyl}$, restricted to some interval $(-\infty,T)$; here the translation and rotation that produces $\MM^0$ is assumed to preserve the axis of rotation.
Then $(\MM^0, \MM^1)$ satisfies the $(\delta,\lb \delta',\lb C_0,\lb \eps,\lb D,\lb \alpha, \lb 0,\lb 0)$-leading mode condition.
\end{Lemma}

Note that the space of flow pairs $(\MM^0, \MM^1)$ to which the lemma applies can be described by finitely many parameters.
Thus, in principle, the leading mode condition could be verified by computing the infinitesimal variations in each parameter.
For technical reasons, however, we avoid direct calculations and instead proceed via a limit argument.
In order to make the family of variations more concrete, consider the $(n-k+1)$-dimensional bowl soliton $M_{\bowl}$, with tip at the origin $\bO$, and let $F : [0, \infty) \to [0, \infty)$ be its profile function, that is
\[ M_{\bowl} = \{ (x, F(x) \by) \;\; : \;\; x \geq 0, \; \by \in \IS^{n-k} \} \subset \IR^{k+1} = \IR \times \IR^{n-k+1}. \]
We need the following lemma.

\begin{Lemma} \label{Lem_warping}
We have the following bounds for $x \geq 1$
\[ c \sqrt{x} \leq F(x) \leq C \sqrt{x}, \qquad C^{-1} x^{-1/2} \leq F'(x) \leq C x^{-1/2}, \qquad |2xF'(x) - F(x) | \leq C \]
and $\lim_{x \to \infty} x^{-1/2} F(x) = 1$.
\end{Lemma}

\begin{proof}
Consider the rescaled flow whose time-slices are $\td\MM_\tau = e^{\tau/2} (M_{\bowl} + e^{-\tau} \mathbf e_1)$.
If we express these time-slices locally as $\Gamma_{\cyl}(u_\tau)$ for $u_\tau (\bx,\by) = u_\tau (\bx)$, then
\[ 1 + u_\tau(\bO) = e^{\tau/2} F(e^{-\tau}), \qquad  \frac{\partial u_\tau}{\partial x_1} (\bO) =  F'(e^{-\tau}). \]
The first bound and the limit statement immediately follow due to the asymptotic characterization \cite[Proposition~\refx{Prop_ab_exist}]{Bamler_Lai_PDE_ODI}.
The second and third bounds are equivalent to
\[ C^{-1} e^{\tau/2} \leq \frac{\partial u_\tau}{\partial x_1} (\bO) \leq C e^{\tau/2}, \qquad \bigg| 2e^{-\tau/2} \   \frac{\partial u_\tau}{\partial x_1} (\bO) - e^{-\tau/2}(1 + u_\tau(\bO)) \bigg| \leq C. \]
The first bound again follows easily from \cite[Proposition~\refx{Prop_ab_exist}]{Bamler_Lai_PDE_ODI}.
Using the fact that $\ov b_1 = \frac1{\sqrt{2}}$ from \cite[Remark~\refx{Rmk_bowl_constant}]{Bamler_Lai_PDE_ODI}, the left-hand side of the second bound has the following asymptotics for $\tau \to -\infty$
\[ \bigg| 2 e^{-\tau/2} \Big(  \frac1{\sqrt{2}} \cdot \frac1{\sqrt{2}} e^{\tau/2} + O( e^{\tau}) \Big) -  e^{-\tau/2} \big(1 + O(e^{\tau/2}) \big) \bigg| \leq C. \]
This concludes the proof of the lemma.
\end{proof}
\medskip

\begin{proof}[Proof of Lemma~\ref{Lem_leading_mode_bowl}.]
Let us first describe the parameterization of the space of possible flow pairs $(\MM^0, \MM^1)$.
Consider a vector $\bz = r \mathbf e \in \IR^k$, where $|\mathbf e| = 1$ and $r \geq 0$, a number $b \in \IR$ and vector $\by \in \IR^{n-k+1}$.
If $r > 0$, then we define $M(\bz,b, \by) \subset \IR^{n+1}$ to be the submanifold obtained from the rescaled and translated submanifold $$r (\IR^{k-1} \times M_{\bowl}) - r^{-1} b \, \mathbf e_k + \by$$ by a rotation that maps $\mathbf e_k$ to $\mathbf e$ and fixes all points in the $\bO^k \times \IR^{n-k+1}$-factor.
If $r = 0$, then we set $M(\bO, b,\by) := \sqrt{b} M_{\cyl} + \by$ if $b > 0$ and $M(\bO,b,\by) = \emptyset$ if $b \leq 0$.
Roughly speaking, if $\bz, \by \approx \bO$, then $M(\bz, b, \by)$ is a perturbation of a cylinder of scale $\approx \sqrt{b}$ whose linear mode is given by $\bz$ and other modes by $b^{-1/2} \by$.
We remark that if $|\mathbf e| = 1$ and $r > 0$, then
\begin{equation} \label{eq_explicit_F_formula}
 M(r \mathbf e, b, \by) =   \big\{ \big(\bq, r F( r^{-1} (\bq \cdot \mathbf e) + r^{-2} b) \by' + \by \big) \;\; : \;\; \bq \cdot \mathbf e \geq  - r^{-1} b, \; \bq \in \IR^k, \; \by' \in \IS^{n-k}  \big\}. 
\end{equation}

The next claim summarizes basic facts about these submanifolds.

\begin{Claim} \label{Cl_Mzby}
The following is true:
\begin{enumerate}[label=(\alph*)]
\item \label{Cl_Mzby_a} If $\mathbf v \in \IR^k$ is some vector and $\la > 0$, then
\begin{align*} 
M(\bz, b, \by) - \mathbf v &= M(\bz, b+ \bz \cdot \mathbf v, \by) \\
 \la \big( M(\bz,b, \by) \big) &= M(\la \bz,  \la^{2} b, \la \by) 
\end{align*}
\item \label{Cl_Mzby_b} For $i = 0,1$ there are $\bz^i \in \IR^k$, $b^i \in \IR$ and $\by^i \in \IR^{n-k+1}$ with $\by^0 = \bO$ such that 
\[ \MM_{t}^i = M(\bz^i, b^i-t, \by^i) \qquad \text{for all} \quad t \in \IR. \]
\item \label{Cl_Mzby_c} If $\MM_{t}^i$ is $\eps$-cylindrical at $(\bO,0)$, for $\eps \leq \ov\eps$, then $b^i > 0$ and $|\bz^i|, |\by^i| \leq \Psi(\eps) \sqrt{b^i}$ for some universal function $\Psi$ with $\Psi(\eps) \to 0$ as $\eps \to 0$.
Likewise, if $b^i > 0$ and $|\bz^i|, |\by^i| \leq \eps \sqrt{b^i}$, then $\MM^i_{t}$ is $\Psi(\eps)$-cylindrical at $(\bO,0)$ and $\MM_{t}^i$ is $\eps$-close to $\sqrt{b^i} M_{\cyl}$, for some similar function $\Psi$.
\end{enumerate}
\end{Claim}

\begin{proof}
Assertion~\ref{Cl_Mzby_a} can be verified directly, for example using \eqref{eq_explicit_F_formula}.
Due to Assertion~\ref{Cl_Mzby_a}, we can reduce 
Assertion~\ref{Cl_Mzby_b} to the cases in which $\MM^i$ is a round cylinder, which can be checked explicitly or when $\MM^i$ is a rotation and translation of $\IR^{k-1} \times \MM_{\bowl}$, which follows from the fact that $\MM_{\bowl}$ moves at speed 1.
Observe that $\by^0 = \bO$, as $\MM^0$ is assumed to be rotationally symmetric.
Assertion~\ref{Cl_Mzby_c} follows using Lemma~\ref{Lem_eps_cyl} after normalizing $b_i = 1$.
Note also that due to \eqref{eq_explicit_F_formula} and the last statement in Lemma~\ref{Lem_warping}, we have local Hausdorff convergence $M(\bz_j, b_j ,\by_j) \to M(\bO,1, \bO) = M_{\cyl}$ whenever $(\bz_j, b_j ,\by_j) \to (\bO,1, \bO)$.
\end{proof}

Next we bound the graph function $u$ of $(\MM^0, \MM^1)$ in terms of the parameters $\bz^i, b^i, \by^i$.

\begin{Claim} \label{Cl_u_in_terms_zby}
There is a constant $C^* > 0$ such that for all $(\bq,t) \in \IR^k \times \IR$
we have the following bound 
\begin{multline*}
 C^{*-1} \bigg(  \frac{|b^1  -b^0 + (\bz^1 - \bz^0) \cdot \bq |}{\sqrt{b^0-t}} + |\by^1| \bigg) - C^*| \bz^1 - \bz^0 | \leq  \sup_{\MM_{t}^0 (\bq)} |u|(\cdot, t) \\
 \leq C^* \bigg(  | \bz^1 - \bz^0 | +  \frac{|b^1  -b^0 + (\bz^1 - \bz^0) \cdot \bq |}{\sqrt{b^0-t}} + |\by^1| \bigg) 
\end{multline*}
 whenever 
\begin{equation} \label{eq_zby_condition}
   | \bz^1 - \bz^0 | +  \frac{|b^1  -b^0 + (\bz^1 - \bz^0) \cdot \bq |}{\sqrt{b^0-t}} + |\by^1| \leq C^{*-1} \sqrt{b-t}  \quad \text{and} \quad   C^* |\bz^0|^2 \leq b^0-t + \bz^0 \cdot \bq. 
\end{equation}
\end{Claim}

\begin{proof}
By Claim~\ref{Cl_Mzby} we may apply a time-shift by $t$ and spatial translation in the $\IR^k$-direction to both flows by $-\bq$ followed by parabolic rescaling.
This reduces the claim to the case in which $(\bq, t) = (\bO,0)$ and $b^0 = 1$.
By continuity, we may moreover assume that $r_i := |\bz^i| > 0$.
So it suffices to show that
\begin{equation*}
 C^{*-1} \big(  |b^1  -1  | + |\by^1| \big) - C^* | \bz^1 - \bz^0 | \leq  \sup_{\MM_{0}^0 (\bO)} |u|(\cdot, t) 
 \leq C^* \big(  | \bz^1 - \bz^0 | +  |b^1  -1  | + |\by^1| \big) 
\end{equation*}
whenever 
 \[  | \bz^1 - \bz^0 | +   |b^1 - 1  | + |\by^1| \leq C^{*-1}  \qquad \text{and} \qquad   C^*|\bz^0|^2 \leq 1. \]
If we choose $C^*$ large enough, then these conditions guarantee that $r_0, |r_1 - r_0|, |b^1 -1|$ and $|\by^1|$ are sufficiently small.
So by Claim~\ref{Cl_Mzby} we may assume that $\MM^0$  and $\MM^1$ are sufficiently cylindrical at $(\bO,0)$, which allows us to bound
$\sup_{\MM_{0}^0 (\bO)} |u|(\cdot, 0)$ from above and below by a universal constant times the distance between the intersections of $\MM_{0}^i$ with $\bO^k \times \IR^{n-k+1}$.
Due to \eqref{eq_explicit_F_formula} these spheres have radii $r_i F(r_i^{-2} b^i)$ and are offset by $|\by^1|$.
So for some dimensional constant $C > 0$ (recall that $b^0 = 1$)
\begin{equation*}
 C^{-1} \big( |r_1 F(r_1^{-2} b^1) - r_0 F(r_0^{-2} )| + |\by^1| \big) \leq
 \sup_{\MM^{0}_t (\bO)} |u|(\cdot, 0) 
\leq C \big( |r_1 F(r_1^{-2} b^1) - r_0 F(r_0^{-2} )| + |\by^1| \big).
\end{equation*}
Hence the claim follows once we can show that for sufficiently small $r_0, r_1$ and $|b^1-1|$ we have
\begin{equation} \label{eq_Fbr_bound}
 C^{-1}  |b^1-1| - C|r_1 - r_0| \leq |r_1 F(r_1^{-2} b^1) - r_0 F(r_0^{-2} )| \leq C\big( |r_1-r_0| + |b^1-1| \big). 
\end{equation}
To see this let $r_s := (1-s) r_0 + s r_1$ and $b_s := (1-s)  + s b^1$.
Then
\begin{equation} \label{eq_fbr_der}
  \frac{d}{ds} r_s F(r_s^{-2} b_s) 
= (r_1 - r_0) \big( F(r_s^{-2} b_s)  - 2 r_s^{-2}b_s F'(r_2^{-2} b_s ) \big) + (b^1 - b^0) \big( r_s^{-1} F'(r_s^{-2} b_s) \big). 
\end{equation}
By Lemma~\ref{Lem_warping} we have bounds of the form $|F(r_s^{-2} b_s)  - 2 r_s^{-2}b_s F'(r_2^{-2} b_s)| \leq C$ and $C^{-1} \leq r_s^{-1} F'(r_s^{-2} b_s )\leq C$, so \eqref{eq_Fbr_bound} follows by integrating \eqref{eq_fbr_der}.
\end{proof}

We can now establish Property~\ref{Def_lead_mode_1} from Definition~\ref{Def_flow_pair}.

\begin{Claim} \label{Cl_bowl_Prop1}
If $\delta$, $\delta' \leq \ov\delta'(\delta)$, $C_0 \geq \underline C_0 (\delta, \delta')$, $\eps \leq \ov\eps (\delta, \delta')$, $\alpha \leq \ov\alpha(\delta, \delta')$, then Property~\ref{Def_lead_mode_1} from Definition~\ref{Def_flow_pair} holds whenever $\MM^0$ is $\eps$-cylindrical at $(\bq,t)$ and
\[ \Vert u \Vert_{\bq,t;\delta'} \leq a \qquad \text{for some} \quad a \in (0, \alpha). \]
\end{Claim}

\begin{proof}
Fix $\delta > 0$ and $\delta' \leq \ov\delta' (\delta)$ according to Lemma~\ref{Lem_Vgeq0_Prop_1} and choose sequences $\eps_j, \alpha_j \to 0$ and $C_{0,j} \to \infty$.
If the statement of the claim was wrong, then we could find a sequence of flow pairs $(\MM^0_j, \MM^1_j)$ corresponding to parameters $(\bz^0_j, b^0_j, \bO)$ and $(\bz^1_j, b^1_j, \by^1_j)$ that violate the claim for $\eps_j, C_{0,j}$ and $\alpha_{0,j}$ at some $(\bq_j, t_j)$ and for some $a_j < \alpha_j$ with $\Vert u_j \Vert_{\bq_j, t_j; \delta'} \leq a_j$.
Without loss of generality, we may assume that $(\bq_j, t_j) = (\bO,0)$ and $b^0_j = 1$.
Since $\MM^0_j$ is $\eps_j$-cylindrical at $(\bO,0)$ and $\eps_j \to 0$, we get that $\bz_j^0, \by_j^0 \to \bO$.
Since the graph functions $u_j$ satisfy $\Vert u_j \Vert_{\bO,0;\delta'} \leq a_j \to 0$, we must also have $\bz_j^1, \by_j^1 \to \bO$ and $b^1_j \to 1$.
So the condition \eqref{eq_zby_condition} from Claim~\ref{Cl_u_in_terms_zby} is satisfied for all $(\bq,t) \in P_i := \IB^k_{R_i} \times [-T_{i},0]$ for $R_i, T_i \to \infty$.
So if we set $a'_j :=  |\bz^1_j - \bz^0_j| + |b^1_j - 1| + |\by^1_j|$, then for $(\bq,t)$ we have, after possibly adjusting $C^*$,
\begin{equation} \label{eq_Csloweranduppersupap}
 C^{*-1} \frac1{a'_j} \bigg(  \frac{|b^1_j  - 1 + (\bz^1_j - \bz^0_j) \cdot \bq |}{\sqrt{1-t}} + |\by^1_j| \bigg) - C^* \frac{|\bz^1_j - \bz^0_j|}{a'_j} \leq  \sup_{\MM_{t}^0 (\bq)} \frac{|u|(\cdot, t)}{a'_j} 
 \leq C^* \bigg(  1 +  \frac{1 + | \bq |}{\sqrt{1-t}}  \bigg) .
\end{equation}
Due to local derivative estimates, we find that $\frac{u_j}{a'_j}$ subsequentially converges locally smoothly to $u'_\infty$ of the linearized mean curvature flow equation on a round shrinking cylinder satisfying the bound
\[ \sup_{\MM_{\cyl,t} (\bq)} |u'_\infty|(\cdot, t) 
 \leq C^* \bigg(  1 +  \frac{1 + | \bq |}{\sqrt{1-t}}  \bigg) \]
We can now argue as in the proof of Lemma~\ref{Lem_cyl_u} that all time-slices of $u'_\infty$ are given by elements of $\sV_{\geq 0}$.
If $\liminf_{j \to \infty} \frac{a_j}{a'_j} > 0$, then for a subsequence we also have smooth convergence $a_j^{-1} u_j = \frac{a_j}{a'_j} \cdot (a_j')^{-1} u_j$ to a multiple of $u_\infty$.
In this case bounds \eqref{eq_close_to_leading_mode} and \eqref{eq_dbounds_norm} from Property~\ref{Def_lead_mode_1} must hold for large $j$ due to Lemma~\ref{Lem_Vgeq0_Prop_1} and the last part of this property must hold since $C_{0,j} \to \infty$.
This contradicts our assumption.

Consider now the case $\frac{a_j}{a'_j} \to 0$, so $\frac{1}{a'_j} \Vert u_j \Vert_{\bO,0;\delta'} \leq \frac{a_j}{a'_j} \to 0$, which implies that $u'_\infty \equiv 0$.
In this case the lower bound in \eqref{eq_Csloweranduppersupap} at $t = 0$ implies that for all $\bq \in \IR^k$
\begin{equation} \label{eq_limsupbzqy}
 \limsup_{j \to \infty} \bigg( C^{*-1} \frac1{a'_j} \big(  |b^1_j  - 1 + (\bz^1_j - \bz^0_j) \cdot \bq | + |\by^1_j| \big) - C^* \frac{|\bz^1_j - \bz^0_j|}{a'_j}  \bigg) \leq 0. 
\end{equation}
After passing to a subsequence, we may assume that we have convergence
\[ \frac{\bz^1_j - \bz^0_j}{a'_j} \to \bz', \quad \frac{b^1_j - 1}{a'_j} \to b', \quad \frac{\by^1_j}{a'_j} \to \by' . \]
Then \eqref{eq_limsupbzqy} implies that for all $\bq \in \IR^k$
\[ C^{*-1} \big(  |b'+ \bz' \cdot \bq | + |\by'| \big) - C^* |\bz'| \leq 0 \]
Choosing $\bq = s \bz'$ for $s \gg 1$ implies $\bz' = \bO$ and thus $b' = 0$ and $\by' = \bO$, which is impossible due to the definition of $a'_j$.
\end{proof}

Next, we establish Property~\ref{Def_lead_mode_2} from Definition~\ref{Def_flow_pair}.

\begin{Claim}
If $\delta \leq \ov\delta$, $\delta' \leq \ov\delta'(\delta)$, $C_0 \geq \underline C_0 (\delta, \delta')$, $\eps \leq \ov\eps (\delta, \delta')$, $D > 0$, $\alpha \leq \ov\alpha(\delta, \delta', D)$, then Property~\ref{Def_lead_mode_2} from Definition~\ref{Def_flow_pair} holds whenever $\MM^0$ is $\eps$-cylindrical at $(\bq,t)$ and
\[ \Vert u \Vert_{\bq,t;\delta'} \leq a \qquad \text{for some} \quad a \in (0, \alpha). \]
\end{Claim}

\begin{proof}
Choose and fix $\delta \leq \ov\delta$, $\delta' \leq \ov\delta'(\delta)$ according to Claim~\ref{Cl_bowl_Prop1} and fix $D > 0$.
If the claim was wrong, then we could find flow pairs $(\MM^0_j, \MM^1_j)$ corresponding to parameters $(\bz^0_j, b^0_j, \bO)$ and $(\bz^1_j, b^1_j, \by^1_j)$ that violate the claim for some $(\bq_j, t_j)$ and $a_j < \alpha_j \to 0$ and for constants $\eps_j, C_{0,j} > 0$.
We may assume that $\eps_j \leq \ov\eps(\delta, \delta')$ and $C_{0,j} \geq \underline C_0(\delta, \delta')$ according to Claim~\ref{Cl_bowl_Prop1}.
In the course of the proof we will impose additional bounds of the same type.
This is not an issue, because in principle we could have imposed these requirements in the beginning of the proof.

Consider first the case in which $\MM^0_j$ is $\eps'_j$-cylindrical at $(\bq_j, t_j)$ for some sequence $\eps'_j \to 0$.
In this case, we may apply a translation in space and/or time and parabolic rescaling and assume that $\MM^0_j$ converges locally smoothly to a round cylinder.
Since $a_j \to 0$, the same is true for $\MM^1_j$.
We can now apply the discussion from Claim~\ref{Cl_bowl_Prop1}, so the rescaled graph functions $a_j^{-1} u_j$ converge at time $0$ to an element of $\sV_{\geq 0}$.
Hence Property~\ref{Def_lead_mode_2} follows for large $j$ as long as $C_{0,j} \geq \underline C_0 (\delta, \delta')$.

Let us now assume that for a subsequence $(\bq_j, t_j)$ is \emph{not} $\eps'$-cylindrical for some uniform $\eps' > 0$.
In this case, we may apply a rotation, a translation in space and/or time and parabolic rescaling and assume that $\MM^0_j = \MM^0 = \IR^{k-1} \times \MM_{\bowl}$ and $(\bq_j, t_j) = (x_j \mathbf e_k ,0)$ for some $x_j \geq 0$.
Since $\MM^0$ is not $\eps'$-cylindrical at this point, we obtain that $x_j$ is uniformly bounded from above, so after passing to a subsequence $x_j \to x_\infty$.
Since Claim~\ref{Cl_bowl_Prop1} provides local derivative bounds in terms of $a_j$ on the graph function $u_j$ near $(x_j \mathbf e_k, 0)$, we find that $\MM^1_j = \la_j S_j \MM^0 + \by'_j$ for some $\lambda_j > 0$, $S_j \in O(n+1)$, $\by'_j \in \IR^{n+1}$ with
\[ |\lambda_j-1|, \; |S_j - \id|, \; |\by'_j| \leq C' a_j, \]
where $C'$ is uniform in $j$.
So after passing to a subsequence, we have convergence of the rescaled graph functions $a_j^{-1} u_j \to u_\infty$ to a limit with the following property:
There is an affine linear vector field $\bY$ on $\IR^{n+1}$---a linear combination of a Killing field and a dilational vector field---such that for al $\bp \in \MM^0_0$ (the time-$0$-slice of $\MM^0$) the value $u_\infty(\bp,0)$ is equal to the normal component of $\bY(\bp)$.

We now claim that for all $\bq' \in \IR_+ \times \IR^k$ we have with $\bq_\infty := (e_\infty \mathbf e_k, \bO)$
\begin{equation} \label{eq_supuinftydesiredM0}
 \sup_{\MM^0_0(\bq)} |u_\infty|(\cdot, 0) \leq C(\delta,\delta') \exp \bigg(\frac{|\bq'-\bq_\infty|}{\rho^{\MM^0} (\bq_\infty,0)} \bigg) \Vert u_\infty \Vert_{\bq_\infty, 0; \delta'} \, \rho^{\MM^0} (\bq_\infty,0). 
\end{equation}
Once we can show this, we know that Property~\ref{Def_lead_mode_2} must hold for large $j$ as long as we assumed a bound of the form $C_{0,j} \geq \underline C_0(\delta, \delta')$.

Consider the bound $\ov\eps(\delta, \delta')$ from Claim~\ref{Cl_bowl_Prop1} and choose $X = X(\delta, \delta')$ such that $\MM^0$ is $\ov\eps$-cylindrical at $(\bq,0)$ if and only if $x_k (\bq) \geq  X$.
Since we have assumed that $\eps_j \leq \ov\eps(\delta, \delta')$, we must have $x_\infty \geq  X$.
Passing Claim~\ref{Cl_bowl_Prop1} to the limit implies that for all $\bq \in [X, \infty) \times \IR^{k-1}$ we have
\begin{equation} \label{eq_derboundlargerX}
 \big| \partial_{\bq} \Vert u_\infty \Vert_{\bq, 0; \delta'} \big| \leq \delta \, \rho^{\MM^0} (\bq, 0) \Vert u_\infty \Vert_{\bq, 0; \delta'}. 
\end{equation}
On the other hand, since $u_\infty$ is the normal projection of an ambient affine linear vector field $\bY$, we find that for all $\bq = (x, \bq^*) \in [0,  X] \times \IR^{k-1}$ we have
\[ \sup_{\MM^0_0(\bq)} |u_\infty|(\cdot, 0) \leq C(\delta, \delta') \Vert u_\infty \Vert_{( X, \bq^*), 0; \delta'}. \]
So in order to show \eqref{eq_supuinftydesiredM0} it suffices to prove that for any $\bq' = (x', \bq^*) \in [ X,\infty) \times \IR^{k-1}$
\[ \Vert u_\infty \Vert_{\bq', 0; \delta'} \leq C(\delta,\delta') \exp \bigg(\frac{|\bq'-\bq_\infty|}{\rho^{\MM^0} (\bq_\infty,0)} \bigg) \Vert u_\infty \Vert_{\bq_\infty, 0; \delta'} \, \rho^{\MM^0} (\bq_\infty,0). \]
Integrating \eqref{eq_derboundlargerX} first along the segment $s \mapsto (x_\infty, s \bq^*)$ and then along $s \mapsto (s, \bq^*)$ implies
\[ \Vert u_\infty \Vert_{\bq', 0; \delta'} 
\leq \exp \bigg( \frac{\delta |\bq^*|}{\rho^{\MM^0} (\bq_\infty,0)} + \delta \int_{x_\infty}^{x'}\frac{ds }{\rho^{\MM^0} ((s, \bq^*),0)} \bigg) \Vert u_\infty \Vert_{\bq_\infty, 0; \delta'} . \]
So it remains to show that
\begin{equation*}
 \exp \bigg( \frac{\delta |\bq^*|}{\rho^{\MM^0} (\bq_\infty,0)} + \delta \bigg| \int_{x_\infty}^{x'}\frac{ds }{\rho^{\MM^0} ((s, \bq^*),0)} \bigg| \bigg) 
\leq C(\delta,\delta') \exp \bigg(\frac{|\bq'-\bq_\infty|}{\rho^{\MM^0} (\bq_\infty,0)} \bigg)  \rho^{\MM^0} (\bq_\infty,0). 
\end{equation*}
Since $\frac12 |\bq^*| + \frac12 |x' - x_\infty| \leq |\bq'-\bq_\infty|$  since we may assume $\delta \leq \frac12$ and since $\rho^{\MM^0} (\bq_\infty,0)$ is bounded from below, we may reduce this to showing that
\[  \exp \bigg(  \delta \bigg| \int_{x_\infty}^{x'}\frac{ds }{\rho^{\MM^0} ((s, \bq^*),0)} \bigg| \bigg) 
\leq C(\delta,\delta') \exp \bigg(\frac{|x'-x_\infty|}{2\rho^{\MM^0} (\bq_\infty,0)} \bigg)  .  \]
Since $\rho^{\MM^0} ((s, \bq^*),0) = F(s) \sim \sqrt{s}$ (see Lemma~\ref{Lem_warping}), this can be reduced to the following bound for some universal $C'', c'' > 0$
\[ C'' \delta \big| \sqrt{x'} - \sqrt{x_\infty} \big| \leq c'' \frac{|x'-x_\infty|}{\sqrt{x_\infty}}.  \]
If we divide both sides by $\sqrt{x_\infty}$, then this becomes 
\[ C'' \delta \bigg| \sqrt{\frac{x'}{x_\infty}} - 1 \bigg| \leq c''  \bigg|\frac{x'}{x_\infty} - 1 \bigg|, \]
which is trivially true for $\delta \leq \ov\delta$.
\end{proof}

This concludes the proof of the lemma.
\end{proof}
\bigskip

\subsection{Proof of the start of the induction, Lemma~\ref{Lem_start_induction}} \label{subsec_start_induction}
We will prove Lemma~\ref{Lem_start_induction} by reducing it to Lemma~\ref{Lem_leading_mode_bowl} via a limit argument in which $C_1 \to \infty$.

\begin{proof}[Proof of Lemma~\ref{Lem_start_induction}.]
Fix $\delta > 0$, $\delta' \leq \ov\delta'(\delta)$, $C_0 \geq \underline C_0 (\delta, \delta')$, $\eps \leq \ov\eps(\delta, \delta')$, $D \geq 0$, $\alpha \leq \ov\alpha(\delta, \delta', C_0, D)$ so that strict versions of the bounds required in Lemma~\ref{Lem_leading_mode_bowl} hold and fix also $\beta, C^* > 0$.
If the statement of the lemma was assertion was false, then there are constants $C_{1,j}\to\infty$ and flow pairs $(\MM^0_j,\MM^1_j)$ satisfying \eqref{eq_scale_l_scale_rot} which 
fail the $(\delta, \delta', C_0, \eps, D, \alpha, \beta, C_{1,j})$-leading mode condition. 
After suitable rescaling and translation, we may assume that this condition fails at $(\bq,t)= (\bO,0)$ and that $\rho^{\MM^0_j}(\bO,0)=1$.
So we assume that $\MM_j^0$ is $\eps$-cylindrical at $(\bO,0)$ and there are numbers $a_j$ such that the graph function $u_j$ satisfies 
\begin{equation*}
\Vert u_j \Vert_{\bO,0;\delta'} < a_j \qquad \text{for some} \quad a_j \in \Big( \max \Big\{ \beta, C_{1,j} \, {\Vert \Qu (\MM^0_j)\Vert^{10}}  \Big\} \alpha, \alpha \Big),
\end{equation*}
but Property~\ref{Def_lead_mode_1} or \ref{Def_lead_mode_2} in Definition \ref{Def_lead_mode} fails at $(\mathbf0,0)$.
Since 
$$C_{1,j} \Vert \Qu (\MM^0_j)\Vert^{10} \alpha\le a_j\le \alpha, $$ 
we obtain using \eqref{eq_scale_l_scale_rot} that
\begin{equation} \label{eq_scalestozero}
\Qu(\MM^0_j), \; \Qu(\MM^1_j) \lto 0.
\end{equation}
Moreover, after passing to a subsequence we may assume 
\[
a_j\to a_\infty\in \big[  \beta \alpha, \alpha \big].\]
Note that it is crucial here that $a_\infty > 0$, since $\beta > 0$.
Due to \eqref{eq_scalestozero}, we can use \cite[Proposition~\refx{Prop_Q_continuous}]{Bamler_Lai_PDE_ODI} and obtain that for a subsequence we have local smooth convergence $\MM^i_j \to \MM_\infty^i$, where each limit $\MM_\infty^i$ is either as in Lemma~\ref{Lem_leading_mode_bowl} or an affine plane or empty.
The last two options can be excluded by assuming $\eps \leq \ov\eps$ and $\alpha \leq \ov\alpha$.
However, this Lemma~\ref{Lem_leading_mode_bowl} implies that $(\MM^\infty_0, \MM^\infty_1)$ satisfies the $(\delta, \delta', C_0, \eps, D, \alpha, 0, 0)$-leading mode condition, even after slightly decreasing/increasing the constants.
So Properties~\ref{Def_lead_mode_1} and \ref{Def_lead_mode_2} must hold at $(\bO,0)$ for large enough $j$, which yields the desired contradiction.
\end{proof}
\bigskip

\subsection{Proof of the leading mode condition} \label{subsec_leading_mode_condition_proof}
Now we prove Proposition \ref{Prop_leading_mode_condition} by combining Lemmas~\ref{Lem_start_induction} and \ref{Lem_induction_step}.

\begin{proof}[Proof of Proposition \ref{Prop_leading_mode_condition}.]
The proposition lists dependencies of the constants 
\[ \delta, \; \delta', \; C_0, \; \eps, \; D, \; \alpha \]
from Lemmas~\ref{Lem_start_induction} and \ref{Lem_induction_step}, presented in this order, with redundancies removed.
For example, since we are assuming a bound of the form $\delta' \leq \ov\delta'(\delta)$, we may assume that this bound implies $\delta' \leq \delta$.
So it suffices to require a bound of the form $C_0 \geq \underline{C}_0 (\delta')$ instead of $C_0 \geq \underline{C}_0 (\delta, \delta')$.

Next, choose $\beta_0 \leq \ov\beta(\alpha)$ according to Lemma~\ref{Lem_induction_step} (here and in the following we will again remove redundant dependencies).
Assume moreover that $C_1 \geq \underline{C}_1 (\alpha, \beta_0, C^*)$ according to Lemma~\ref{Lem_start_induction} and also $C_1 \geq \underline{C}_1(\alpha)$ according to Lemma~\ref{Lem_induction_step}.
Note that the second bound does not depend on $\beta$ or $\beta_0$.

Under these choices, Lemma~\ref{Lem_start_induction} establishes  the $(\delta,\lb \delta',\lb C_0,\lb \eps,\lb D,\lb \alpha, \lb \beta_0,\lb C_1)$-leading mode condition for all flow pairs.
On the other hand if $\beta \leq \beta_0$, then Lemma~\ref{Lem_induction_step} shows that the $(\delta,\lb \delta',\lb C_0,\lb \eps,\lb D,\lb \alpha, \lb \beta,\lb C_1)$-leading mode condition implies the $(\delta,\lb \delta',\lb C_0,\lb \eps,\lb D,\lb \alpha, \lb \frac12 \beta,\lb C_1)$-leading mode condition.
Iterating this implication implies the $(\delta,\lb \delta',\lb C_0,\lb \eps,\lb D,\lb \alpha, \lb (\frac12)^i\beta_0,\lb C_1)$-leading mode condition for all $i$, so the proposition follows for $i \to \infty$.
\end{proof}
\bigskip

\subsection{Proof of Proposition~\ref{Prop_diff_properties}} \label{subsec_lead_mode_last}
The statement of Proposition~\ref{Prop_diff_properties} quantifies closeness of the flows $\MM^0$ and $\MM^1$ in terms of the differences $u_1 - u_0$ of the graph functions over the round cylinder $M_{\cyl}$.
By contrast, in our discussion so far we have quantified this difference in terms of the graph function $u$ over $\MM^0$, which represented $\MM^1$ as a graph over $\MM^0$.
The following lemma addresses this technical point.
It allows us to convert bounds on $u_1 - u_0$ into bounds on $u$ and vice versa.

\begin{Lemma} \label{Lem_tech_u_v}
There is a constant $C > 0$ such that the following is true.
Consider two smooth functions $u_1, u_0 : \IB^k_{100} \times \IS^{n-k} \to \IR$ over subsets of the standard cylinder and let $M_i = \Gamma_{\cyl}(u_i) \subset \IR^{n+1}$ be the corresponding graphs over the cylinder $M_{\cyl}$.
We assume $u_0$ is chosen such that $M_0$ is rotationally symmetric, so $u_0 (\bx,\by) = u_0(\bx)$.

Assume that $|\nabla^m u_i| \leq C^{-1}$ for $m=0,\ldots, 100$ and $i=0,1$.

Let $u : D \to \IR^{1}$, $D \subset M_0$, be the graph function of $M_1$ over $M_0$, as defined in Definition~\ref{Def_flow_pair}.
Note that for sufficiently large $C$ the normal injectivity radius of $M_0$ at points $(\bx, \by) \in M_0$ with $|\bx| < 10$ can be assumed to be close enough to the radius of $\IS^{n-k}$.
So the definition of the graph function $u(\bx, \by)$ depends only on the geometry $M_1, M_0$ in a bounded neighborhood of the origin.
Then the following is true for the cylindrical model $\td u : \td D \to \IR$ at $\bO$ (see Definition~\ref{Def_cyl_model}):
\begin{enumerate}[label=(\alph*)]
\item \label{Lem_tech_u_v_a} $\Vert  \td u \Vert_{L^2(\IB^k_{1} \times \IS^{n-k})} \leq C \Vert u_1 - u_0 \Vert_{L^2(\IB^k_{10} \times \IS^{n-k})}$
\item \label{Lem_tech_u_v_b} $\Vert u_1 - u_0 \Vert_{C^{10}(\IB^k_{0.1} \times \IS^{n-k})} \leq C \Vert  \td u \Vert_{C^{10}( \IB^k_{1} \times \IS^{n-k})}$
\end{enumerate}
\end{Lemma}

\begin{proof}
After application of a slight rescaling and shrinking/enlarging the domains, and assuming that $C$ is sufficiently large, we may assume without loss of generality that $u_0(\bx,0)=0$, so $\rho^{M_0}(\bO) = 1$ and the cylindrical model $\td u$ has the form
\[ \td u(\bx,\by) = u \big(\bx, (1+u_0(\bx)) \by \big), \qquad (\bx,\by) \in \IB^k_{99} \times \IS^{n-k}. \]
Next express the outward normal vector to $M_0$ at $(\bx, (1+u_0(\bx)) \by )$ as $\nu_{\bx} = (\mathbf v (\bx), (1+a(\bx)) \by)$ (recall that $M_0$ is rotationally symmetric), where $\mathbf v$ and $a$ can be assumed to be small in the $C^{99}$-norm if $C$ is chosen large enough.
By definition of $u$, for any $(\bx,\by) \in \IB^k_{90} \times \IS^{n-k}$ there is a point $(\ov\bx,\ov\by) \in \IB^k_{91} \times \IS^{n-k}$ such that
\[ \big(\ov\bx,(1+u_1(\ov\bx,\ov\by))\ov\by \big) = \big(\bx,(1+u_0(\bx))\by \big) + \td u(\bx,\by) \nu_{\bx}, \]
which implies $\ov\by = \by$ and
\begin{align*}
\ov\bx &= \bx +  \td u(\bx,\by) \mathbf v(\bx) \\
1+u_1(\ov\bx,\by) &= 1+u_0(\bx) + \td u(\bx,\by) (1+a(\bx)),
\end{align*}
so
\begin{equation} \label{eq_u1mu0tdu}
 (u_1-u_0)(\bx,\by) = \td u(\bx,\by) + u_1(\bx +  \td u(\bx,\by) \mathbf v(\bx),\by) - u_1(\bx,\by) +  \td u(\bx,\by) a(\bx). 
\end{equation}
Since $u_1, \ov v$ and $a$ are bounded in the $C^{99}$-norm, this implies Assertion~\ref{Lem_tech_u_v_b}.
On the other hand, we may assume that for sufficiently large $C$ these norms are sufficiently small so that we have
\[ \big\Vert u_1(\bx +  \td u(\bx,\by) \mathbf v(\bx),\by) - u_1(\bx,\by) +  \td u(\bx,\by) a(\bx) \big\Vert_{L^2(\IB^k_{9} \times \IS^{n-k})} \leq \tfrac12 \Vert \td u \Vert_{L^2(\IB^k_{9} \times \IS^{n-k})}. \]
Combining this with \eqref{eq_u1mu0tdu} implies
\[ \Vert u_1 - u_0 \Vert_{L^2(\IB^k_{2} \times \IS^{n-k})} 
\geq \Vert \td u \Vert_{L^2(\IB^k_{9} \times \IS^{n-k})} - \tfrac12 \Vert \td u \Vert_{L^2(\IB^k_{9} \times \IS^{n-k})}
= \tfrac12 \Vert \td u \Vert_{L^2(\IB^k_{9} \times \IS^{n-k})}, \]
proving Assertion~\ref{Lem_tech_u_v_a}.
\end{proof}
\bigskip

\begin{proof}[Proof of Proposition~\ref{Prop_diff_properties}.]
Fix the constant $C^*$ and choose constants $\delta,\lb \delta',\lb C_0,\lb \eps,\lb D,\lb \alpha, \lb C_1$ depending on $C^*$ such that the flow pair $(\MM^{ 0}, \MM^{1})$ satisfies the $(\delta,\lb \delta',\lb C_0,\lb \eps,\lb D,\lb \alpha, \lb 0,\lb C_1)$-leading mode condition due to Proposition~\ref{Prop_leading_mode_condition}.
We may assume in addition that:
\begin{itemize}
\item $\alpha \leq 2^{-5}$ and $\delta \leq 0.01$
\item $\eps \leq \ov\eps$ is chosen small enough such that whenever $\MM^{0}$ is $\eps$-cylindrical at some point $(\bq, t)$, then $\partial_t \rho^2(\bq,t) \leq - 0.99$.
Note that in the case of the round shrinking cylinder this derivative is equal to $-1$, so the existence of $\ov\eps$ follows from a simple limit argument via Lemma~\ref{Lem_eps_cyl}.
\item $D \geq \underline{D}(\eps)$ is chosen large enough such that the following is true. Recall that the set of points $\bq \in \IR^k$ such that $\IR^{k-1} \times M_{\bowl}$ is $\eps$-cylindrical at $(\bq, 0)$ is of the form $\IR^{k-1} \times [q_\eps, \infty)$ for some $q_\eps > 0$.
We then require that $D \rho((\bO^{k-1}, q_\eps)) > q_\eps$.
\end{itemize}
We will henceforth consider the constants $\delta,\lb \delta',\lb C_0,\lb \eps,\lb D,\lb \alpha, \lb C_1$ as fixed and omit dependencies on these constants as well as on the dimension $n$.
We may still freely choose the constant $C$ from the statement of the proposition.

The bound $|\nabla^m u_{0,\tau}| \leq C^{-1}$ on $\IB^k_R \times \IS^{n-k}$ in Assertion~\ref{Prop_diff_properties_a} implies, for sufficiently large $C$, that $\MM^{0}$ is $\eps$-cylindrical at $(\bq, -e^{-\tau})$ for all $\bq \in \IB^k_{R- \frac12 C}$.
Likewise, in Assertion~\ref{Prop_diff_properties_b}, the fact that $\MM^{ 0}$ is asymptotically cylindrical implies that it must be $\eps$-cylindrical at $(\bO,t)$ for $t \ll 0$.
Using Lemma~\ref{Lem_tech_u_v} we can therefore reduce the proposition to a statement in which $v$ is replaced with the graph function $u$ of the flow pair $(\MM^{ 0}, \MM^{1})$ and which does no longer involve the functions $u_0$ and $u_1$.
So we need to show the following statements for a constant $C$, which may depend on $A$ in Part~\ref{item_bp_diff}:
\begin{enumerate}[label=(\alph*$'$)]
\item \label{item_ap_diff} Suppose that for some $t < T$ and $R > 1$ the following is true for all $\bq \in \IB^k_{\sqrt{-t} R}$:
\begin{enumerate}[label=(\roman*)]
\item \label{item_bet_i} $\MM^0$ is $\eps$-cylindrical at $(\bq, t)$
\item \label{item_bet_ii} $|(-t)^{-1/2} \rho(\bq,t) -  1 | \leq C^{-1}$
\item \label{item_bet_iii} $\Vert u \Vert_{\bq,t; \delta'} < \alpha$
\item \label{item_bet_iv} $C_1 \Vert \Qu(\MM^0) \Vert^{10} \rho^{-10} (\bq,t)  < \alpha$
\end{enumerate}
Then for any $\bq \in \IB^k_{\sqrt{-t} (R-C)}$ and $m = 0,\ldots, 10$ the cylindrical models $\td u_{\bO,t}$ and $\td u_{\bq, t}$, taken at $(\bO,t)$ and $(\bq,t)$ satisfy
\begin{equation} \label{eq_exp_spat_bound}
\qquad \sup_{\IB^k_1 \times \IS^{n-k}} |\nabla^m \td u_{\bq,t}| \leq C  \exp\bigg( {\frac{|\bq|}{\sqrt{-t}}} \bigg) \big( \Vert \td u_{\bO,t} \Vert_{L^2(\IB^k_{1} \times \IS^{n-k})} + \Vert \Qu(\MM^0)\Vert^{10} (-t)^{-5} \big). 
\end{equation}
Note that we could assume \ref{item_bet_iv}, because if the reverse bound holds, then \eqref{eq_Harnack_inprop} follows trivially from the derivative bounds on $u_0$ and $u_1$.
\item  \label{item_bp_diff}
The conclusion from Assertion~\ref{Prop_diff_properties_b} holds if
\begin{equation} \label{eq_limit_fast_dec_L2}
 \liminf_{t \to -\infty} (-t)^2 \Vert \td u_{\bO,t} \Vert_{L^2(\IB^k_{1} \times \IS^{n-k})} < \infty. 
\end{equation}
\end{enumerate}

Let us simplify these norms on $u$ even further.
\begin{Claim} \label{Cl_more_norms}
There is a uniform constant $C' > 0$ such that if \ref{item_bet_i}--\ref{item_bet_iv} in Part~\ref{item_ap_diff} hold and if $C \geq \underline C$, then
\begin{align}
 \Vert  \td u_{\bq,t} \Vert_{C^{10} (\IB_1^k \times \IS^{n-k})} &\leq C' \Vert u \Vert_{\bq,t; \delta'} +  C' \Vert \Qu(\MM^0) \Vert^{10} (-t)^{-5},  \label{eq_C10_qt} \\ 
\Vert u \Vert_{\bq,t; \delta'} &\leq C' \Vert \td u_{\bq,t} \Vert_{L^2(\IB^k_{1} \times \IS^{n-k})}+  C' \Vert\Qu(\MM^0)\Vert^{10} (-t)^{-5}. \label{eq_qt_L2}
\end{align}
Moreover, if $\Vert u \Vert_{\bq,t; \delta'} \geq C_1 \Vert \Qu(\MM^0)\Vert^{10} (-t)^{-5}$, then
\begin{equation} \label{eq_partialq_rest}
\big| \partial_{\bq} \Vert u \Vert_{\bq,t; \delta'} \big| < (-t)^{-1/2} \Vert u \Vert_{\bq,t; \delta'}. 
\end{equation}
\end{Claim}

\begin{proof}
Write $\td u = \td u_{\bq,t}$, let $U = U_{\bq,t}$ be the leading mode approximation from Definition~\ref{Def_lead_mode_approx} and set $\Omega := \IB^k_1 \times \IS^{n-k}$.
Set
$$ a := \max \big\{ \Vert u \Vert_{\bq,t; \delta'},  C_1 \Vert \Qu(\MM^0)\Vert^{10} \rho^{-10}(\bq,t) \big\}  < \alpha.$$
and apply Property~\ref{Def_lead_mode_1} from Definition~\ref{Def_lead_mode} for some slightly larger number (in order to fulfill the strict inequality in \eqref{eq_lead_mode_conditiona}).
We obtain that $\Vert \td u - U \Vert_{C^{10}(\Omega)} \leq 0.1 a$.
It follows using \eqref{eq_q_delta_norm_def} that for some generic constant $C'$
\begin{equation} \label{eq_C10_derivation}
 \Vert \td u \Vert_{C^{10}(\Omega)}
\leq \Vert \td u - U \Vert_{C^{10}(\Omega)} + \Vert  U \Vert_{C^{10}(\Omega)}
\leq 0.1 a + C' \Vert u \Vert_{\bq,t;\delta'}
\leq C' a.  
\end{equation}
This shows \eqref{eq_C10_qt}.

To see \eqref{eq_qt_L2}, it suffices to consider the case in which $\Vert u \Vert_{\bq,t; \delta'} > C_1 \Vert \Qu(\MM^0)\Vert^{10} \rho^{-10}(\bq,t)$, because otherwise the bound is trivially true.
So $a = \Vert u \Vert_{\bq,t; \delta'}$.
Hence, if we set $u' := a^{-1}  u$ and $\td u' := a^{-1} \td u$, then $\Vert  u' \Vert_{\bq,t; \delta'} = 1$ and by \eqref{eq_C10_derivation} we have $\Vert \td u' \Vert_{C^{10}(\Omega)} \leq C'$.
Our goal is to establish a uniform lower bound on $\Vert \td u' \Vert_{L^2(\Omega)}$ based on these bounds.
This can be accomplished by a basic limit argument.

The bound \eqref{eq_partialq_rest} is a direct consequence of \eqref{eq_dbounds_norm} in Property \ref{Def_lead_mode_1} seeing condition \ref{item_bet_ii}.
\end{proof}

Claim~\ref{Cl_more_norms} allows us to reduce \eqref{eq_exp_spat_bound} to a bound of the form
\[ \Vert u \Vert_{\bq,t; \delta'} \leq C \exp\bigg( {\frac{|\bq|}{\sqrt{-t}}} \bigg) \big(  \Vert u \Vert_{\bO,t; \delta'} +  \Vert \Qu(\MM^0)\Vert^{10} (-t)^{-5} \big). \]
The strict version of this bound clearly holds for $\bq = \bO$, as long as $C > 1$.
So if it was not true for all $\bq \in \IB^k_{\sqrt{-t} R}$, then we could find $\bq \in \IB^k_{\sqrt{-t} R}$ with minimal $|\bq|$ at which equality holds.
If $C \geq C_1$, then this implies $\Vert u \Vert_{\bq,t; \delta'} \geq C_1 \Vert \Qu(\MM^0)\Vert^{10} (-t)^5$, so we can apply \eqref{eq_partialq_rest}, which yields a contradiction to the minimality of $|\bq|$.
This finishes the proof of Assertion~\ref{item_ap_diff}.

To see Assertion~\ref{item_bp_diff} assume that \eqref{eq_limit_fast_dec_L2} holds.
Since both flows $\MM^0$ and $\MM^1$ are asymptotically cylindrical, the condition \ref{item_bet_i}--\ref{item_bet_iv} must apply for $\bq = \bO$ and $t \ll 0$, so by Claim~\ref{Cl_more_norms} we even have
\begin{equation} \label{eq_limit_ot_norm}
 \liminf_{t \to -\infty} (-t)^2 \Vert u \Vert_{\bO,t;\delta'}< \infty. 
\end{equation}

\begin{Claim} \label{Cl_u_bound_everywhere}
For any $(\bq,t) \in \IR^k \times (-\infty,T)$ the following is true.
If $\MM^0$ is $\eps$-cylindrical at $(\bq, t)$ and
\begin{equation} \label{eq_C1scalelessalpha}
  C_1 \Big( \frac{\Vert \Qu(\MM^0) \Vert}{\rho(\bq,t)} \Big)^{10} < \alpha, 
\end{equation}
then
\[ \Vert u \Vert_{\bq,t;\delta'} \leq  C_1 \Big( \frac{\Vert \Qu(\MM^0) \Vert}{\rho(\bq,t)} \Big)^{10}. \]
\end{Claim}

\begin{proof}
Let $r := \rho(\bq, t)$. Condition \ref{item_bet_iii}  
allows us to apply  Property~\ref{Def_lead_mode_1} from the leading mode condition, Definition~\ref{Def_lead_mode}, at points $(\bO,t_j)$ for $t_j \to -\infty$ and for $a_j$ such that $\limsup_{j \to \infty} (-t_j)^2 a_j < \infty$.
So for large $j$ we obtain that 
$$\Vert u \Vert_{\bq,t_j;\delta'} \leq C_0\Big(\Vert u \Vert_{\bO,t_j;\delta'}+  C_1 \Big( \frac{\Vert \Qu(\MM^0) \Vert}{\rho(\bO,t_j)} \Big)^{10}\Big),$$ 
which together with \eqref{eq_limit_ot_norm} implies
\[ \liminf_{t'\to -\infty} (-t')^2 \Vert u \Vert_{\bq,t';\delta'}< \infty. \]
So if we set $f(\tau) := \Vert u \Vert_{\bq,t+r^2-r^2 e^{-\tau};\delta'}$, then
\begin{equation} \label{eq_limtauj}
 \liminf_{\tau \to -\infty} e^{-2\tau} f(\tau) < \infty. 
\end{equation}
Recall that by our choice of $\eps$ we have $\partial_t \rho^2(\bq,t) \leq - 0.99$, so for $t' = t+ r^2 - r^2 e^{-\tau} \leq t$
\begin{equation} \label{eq_rhosquaredqtp}
 \rho^2(\bq, t') \geq r^2 + \tfrac12 (t-t') \geq \tfrac12 (r^2 + t - t') \geq 0.99 r^2 e^{-\tau}. 
\end{equation}
It follows that we can apply Property~\ref{Def_lead_mode_1} for any $t' = t+ r^2 - r^2 e^{-\tau} \leq t$ if
\[ \Vert u \Vert_{\bq,t';\delta'} < a \qquad \text{and} \qquad a \in \bigg( 2^{5} C_1  \Big(\frac{\Vert \Qu(\MM^0) \Vert}{r}\Big)^{10}  e^{5\tau} \alpha , \alpha \bigg). \]
So the time derivative bound in \eqref{eq_dbounds_norm} and \eqref{eq_rhosquaredqtp} imply that for $\tau \leq 0$
\begin{equation} \label{eq_fpbound}
  f'(\tau)\leq 0.99^{-1}(1+\delta) f(\tau) \leq 1.1 f(\tau) \qquad \text{if} \quad f(\tau) \in \bigg(  C_1  \Big(\frac{\Vert \Qu(\MM^0) \Vert^{10}}{r} \Big) e^{5\tau}  , \alpha \bigg). 
\end{equation}
Note that the bound \eqref{eq_C1scalelessalpha} guarantees the the endpoints of this interval are in the right order for $\tau \leq 0$.

Due to \eqref{eq_limtauj} we can fix a sequence $\tau_j \to -\infty$ such that $f(\tau_j) \leq C'' e^{2\tau_j}$ for some uniform $C''$, which may depend on the flow $\MM^0, \MM^1$, but not on time.

We first show that $f(\tau) < \alpha$ for $\tau \leq 0$.
Suppose this was false and choose $\tau^*_j \in (\tau_j, 0]$ maximal such that $f < \alpha$ on $[\tau_j, \tau^*_j)$.
So $f(\tau^*_j) = \alpha$ for large $j$.
Next, choose $\tau^{**}_j \in [\tau_j, \tau^{*}_j)$ minimal such that the condition in \eqref{eq_fpbound} holds for all $\tau \in (\tau^{**}_j, \tau^*_j)$.
Integrating the differential inequality in \eqref{eq_fpbound} gives
\[ f(\tau^{**}_j) \geq e^{-1.1 (\tau^*_j - \tau^{**}_j)} \alpha
\geq e^{1.1 \tau^{**}_j} \alpha 
> \Big(\frac{\Vert \Qu(\MM^0) \Vert^{10}}{r} \Big) e^{5\tau^{**}_j} \]
for large $j$.
Therefore, $\tau^{**}_j = \tau_j$ for large $j$ and we obtain using the first and third term in the previous string of inequalities
\[ C'' e^{2\tau_j} \geq f(\tau_j) \geq e^{1.1 \tau_j} \alpha, \]
which is false for large $j$.
Hence we have $f(\tau) < \alpha$ for all $\tau \leq 0$.

Next, suppose that for some $\tau \leq 0$ we have
\begin{equation} \label{eq_f_lower_cont}
 f(\tau) >  C_1  \Big(\frac{\Vert \Qu(\MM^0) \Vert^{10}}{r} \Big) e^{5\tau} .
\end{equation}
Choose $\tau^* \in [-\infty, 0)$ minimal such that the same bound holds on $(\tau^*,\tau]$.
Integrating the differential bound in \eqref{eq_fpbound} over this time-interval implies that for all $\tau' \in (\tau^*,\tau]$
\begin{equation} \label{eq_f_lower_exp}
 f(\tau') > e^{-1.1(\tau-\tau')} C_1  \Big(\frac{\Vert \Qu(\MM^0) \Vert^{10}}{r} \Big) e^{5\tau}
\geq C_1  \Big(\frac{\Vert \Qu(\MM^0) \Vert^{10}}{r} \Big) e^{5\tau'}. 
\end{equation}
If $\tau^* > -\infty$, then setting $\tau' = \tau^*$ yields a contradiction, so the first inequality in \eqref{eq_f_lower_exp} must hold for all $\tau' \leq \tau$, which contradicts the bound $f(\tau_j) \leq C'' e^{2\tau_j}$ for large $j$.
It follows that \eqref{eq_f_lower_cont} is false for all $\tau \leq 0$ and particularly
\[ f(0) \leq  C_1  \Big(\frac{\Vert \Qu (\MM^0) \Vert^{10}}{r} \Big) ,\]
which proves the claim.
\end{proof}

We can now establish the conclusion from Assertion~\ref{Prop_diff_properties_b} via a limit argument.
Let $(\MM^0_j, \MM^1_j)$ be a sequence of flow pairs that satisfy the bounds that we have derived so far for uniform constants; note that this includes the bound $\Vert \Qu (\MM^1_j) \Vert \leq C^* \Vert \Qu (\MM^0_j) \Vert$.
Let $\eps'_j \to 0$ and $C_j \to \infty$ and consider  a sequence of rescalings $r_j^{-1} (\MM^0_{t_j} - \bq_j)$ with $r_j \geq C_j \Vert \Qu(\MM^0_j) \Vert$, which are $\eps'_j$-close to $M_{\cyl}$ or to a rotation of $\IR^{k-1} \times M_{\bowl}$, but assume that conclusion in Assertion~\ref{Prop_diff_properties} is violated for uniform $A$.
Without loss of generality, we may assume that $(\bq_j, t_j) = (\bO, 0)$ and $r_j = 1$, which implies $\Vert \Qu(\MM^0_j) \Vert \to 0$ and hence also $\Vert \Qu (\MM^1_j) \Vert \to 0$.
So by \cite[Proposition~\refx{Prop_Q_continuous}]{Bamler_Lai_PDE_ODI} we can pass to subsequence and assume that $\MM^0_j \to \MM^0_\infty$ and $\MM^1_j \to \MM^1_\infty$ in the Brakke sense.
Since $\eps'_j \to 0$ we know that the limit $\MM^0_\infty$ is a round shrinking cylinder or a bowl soliton times a Euclidean factor.
So the first convergence is local smooth; the local smoothness of the second convergence follows via the next paragraph.

The fact that $\Qu(\MM^0_j) \to 0$ also implies that condition \eqref{eq_C1scalelessalpha} from Claim~\ref{Cl_u_bound_everywhere} holds for any $(\bq,0)$ at which $\MM^0_j$ is $\eps$-cylindrical, as long as $j \geq \underline j(|\bq|)$.
We may also apply Property~\ref{Def_lead_mode_2} from Definition~\ref{Def_lead_mode} for any such $(\bq, 0)$ and $a = 2 C_1 (\Vert \Qu (\MM^0_j) \Vert/\rho(\bq,0))^{10}$ to derive bounds on $\supp_{\MM^{0,\reg}_{j,0} (\bq')} |u|$ whenever $|\bq'-\bq| < D \rho(\bq,0)$.
Due to our choice of $D$ this also implies bounds at points $\bq'$ where $\MM^0_j$ is not $\eps$-cylindrical, for large $j$.
The desired bound on the Hausdorff distance follows from this directly for large enough $j$, in contradiction to our assumptions.
This finishes the proof of the proposition.
\end{proof}

\bigskip

\section{Uniqueness of cylindrical flows}\label{sec_uniqueness}
\subsection{Overview and statement of the main results}
Our main result will be:

\begin{Proposition} \label{Prop_M0isM1}
Let $\MM^0, \MM^1$ be two asymptotically $(n,k)$-cylindrical mean curvature flows defined over the same time-interval.
Assume that $\MM^0$ is smooth, convex, rotationally symmetric  and has uniformly bounded second fundamental form.
If both flows satisfy the conclusion of Assertion~\ref{Prop_diff_properties_b} in Proposition~\ref{Prop_diff_properties} for any $A > 0$ and for uniform constants $\eps(A), C(A)$, then $\MM^0=\MM^1$.
\end{Proposition}

The proof of Proposition~\ref{Prop_M0isM1} is based on a comparison principle.
We will first show that there exists a (possibly large) $\Delta T > 0$ such that $\MM^1_t$ lies between the time-shifted flows $\MM^0_{t \pm \Delta T}$ for all $t$.
Then we will use the strong maximum principle to show that the set of such offsets $\Delta T$ is both open and closed, and hence equal to $\IR_+$.
In both steps, we will use Assertion~\ref{Prop_diff_properties_b} of Proposition~\ref{Prop_diff_properties}.
Letting $\Delta T \to 0$ yields $\MM^0 = \MM^1$.

\bigskip

\subsection{Preparation}
The following lemma expresses a consequence of Assertion~\ref{Prop_diff_properties_b} in Proposition~\ref{Prop_diff_properties}.
It roughly states that $\MM^1$ can intersect a time-shift of $\MM^1$ only if the shift is sufficiently small, and any such intersection must occur where $\MM^0$ still exhibits cylindrical regions at controlled  scales.

\begin{Lemma} \label{Lem_pjinintersection}
Suppose that $\MM^0$ and $\MM^1$ are two asymptotically $(n,k)$-cylindrical mean curvature flows in $\IR^{n+1} \times (-\infty,T)$ that satisfy the assumptions of Proposition~\ref{Prop_M0isM1}.
Suppose that there is a sequence $\Delta T_j \in \IR$ with $|\Delta T_j | \geq c > 0$ and points
\begin{equation} \label{eq_pjinintersection}
 \bp_j \in (\spt \MM^0)_{t_j + \Delta T_j} \cap (\spt \MM^1)_{t_j}. 
\end{equation}
Then $|\Delta T_j|$ is uniformly bounded and, after passing to a subsequence, we have local smooth convergence $\MM^0 - (\bp_j, t_j + \Delta T_j) \to \MM_\infty$ over the time-interval $(-\infty, 0]$ and the limit $\MM_\infty$ is asymptotically $(n,k)$-cylindrical.
\end{Lemma}

Note that this implies that for large $|\Delta T|$
\begin{equation} \label{eq_disjoint_DeltaT}
(\spt \MM^0)_{t + \Delta T} \cap (\spt \MM^1)_t = \emptyset \qquad \text{for all} \quad t < \min \{ T, T-\Delta T \}.
\end{equation}
\medskip

\begin{proof}
Consider the shifted flows
\[ \MM^{\prime, 0}_j := \MM^0 - (\bp_j, t_j), \qquad
\MM^{\prime, 1}_j := \MM^1 - (\bp_j, t_j), \]
so $(\bO, \Delta T_j) \in \spt \MM^{\prime, 0}_j$ and $(\bO, 0) \in \spt \MM^{\prime, 1}_j$.

Let $\Theta_{\IR^k \times \IS^{n-k}}$ be the entropy of the round shrinking cylinder and fix some $\Theta_0 \in (1 ,\Theta_{\IR^k \times \IS^{n-k}})$.
Recall that the Gaussian area $\Theta^{\MM^{\prime,0}_j}_{(\bO, \Delta T_j)} (\theta)$ of $\MM^{\prime,0}_j$ based at $(\bO, \Delta T_j)$, is continuous and non-decreasing in $\theta$ and it converges to $\Theta_{\IR^k \times \IS^{n-k}}$ as $\theta \to \infty$.
For each $j$ we choose $\theta_j  \geq |\Delta T_j|$ minimal with the property that $\Theta^{\MM^{\prime,0}_j}_{(\bO, \Delta T_j)} (\theta_j) \geq \Theta_0$.
So either $\theta_j = |\Delta T_i|$ or $\theta_j > |\Delta T_i|$ and $\Theta^{\MM^{\prime,0}_j}_{(\bO, \Delta T_j)} (\theta_j) = \Theta_0$.

Suppose first that $\theta_j$ is uniformly bounded.
Then $|\Delta T_j|$ is also uniformly bounded.
By \cite[Propositions~\refx{Prop_Q_basic_properties} and \refx{Prop_Q_continuous}]{Bamler_Lai_PDE_ODI} we have subsequntial convergence $\MM^{\prime,0}_j \to \MM_\infty$ in the Brakke sense, which can be upgraded to local smooth convergence due to the uniform bound on the second fundamental form and the fact that $\Theta^{\MM^{\prime,0}_j}(\infty) < 2$.
The limit is either asymptotically cylindrical, an affine plane or empty.
The last two properties can be ruled out by taking the bound $\Theta^{\MM^{\prime,0}_j}_{(\bO, \Delta T_j)} (\theta_j) \geq \Theta_0 > 1$ to the limit.

So it remains to assume
\[ \theta_j \to \infty \]
and derive a contradiction.
Consider the parabolically rescaled flows $\theta ^{-1/2}_j \MM^{\prime,0}_j$ and $\theta ^{-1/2}_j \MM^{\prime,1}_j$ and observe that, for $i = 0,1$, by \cite[Propositions~\refx{Prop_Q_basic_properties}]{Bamler_Lai_PDE_ODI}
\begin{equation} \label{eq_scales_to_0}
  \Qu ( \theta ^{-1/2}_j \MM^{\prime,i}_j ) = \theta_j^{-1/2} \Qu (\MM^{i}) \lto 0.\end{equation}
So by \cite[Proposition~\refx{Prop_Q_continuous}]{Bamler_Lai_PDE_ODI} we can pass to a subsequence such that the parabolically rescaled flows converge locally smoothly:
\[ \theta_j^{-1/2}\MM^{\prime,i}_j \lto \MM^{\prime,i}_\infty, \]
where each of the limits is either empty or homothetic to one of the following three models: an affine plane, $\MM_{\cyl}$ and $\IR^{k-1} \times \MM_{\bowl}$.
We may furthermore pass to a subsequence such that $\theta_j^{-1} \Delta T_j \to t^*$ with $|t^*| \leq 1$.
Passing the condition on the Gaussian area to the limit implies that
\[ \Theta^{\MM^{\prime,0}_\infty}_{(\bO, t^*)} (\theta) \geq \Theta_0 \qquad \text{for} \quad \theta > 1 \]
This rules out the cases in which $\MM^{\prime,0}_\infty$ is empty or an affine plane.
Hence $\MM^{\prime,0}_\infty$ must be homothetic to $\MM_{\cyl}$ or $\IR^{k-1} \times \MM_{\bowl}$.

Taking the conclusion from Proposition~\ref{Prop_diff_properties}\ref{Prop_diff_properties_b} to the limit, using \eqref{eq_scales_to_0}, we obtain $\MM^{\prime,0}_\infty = \MM^{\prime,1}_\infty$.
Since the flow $\MM^{\prime,0}_\infty = \MM^{\prime,1}_\infty$ is convex, its time-slices must be pairwise disjoint.
Therefore, as $(\bO, 0), (\bO, t^*)$ are both points contained in its support, we obtain that $t^* = 0$, so $\theta_j^{-1} \Delta T_j \to 0$.
It follows that for large $j$ we have $\theta_j > |\Delta T_j|$ and $\Theta^{\MM^{\prime,0}_j}_{(\bO, \Delta T_j)} (\theta_j) = \Theta_0$.
Passing this to the limit yields $\Theta^{\MM^{\prime,0}_\infty}_{(\bO, 0)} (1) = \Theta_0$.
Since $(\bO,0) \in \spt \MM^{\prime,0}_\infty$, this implies that if $\MM^{\prime,0}_\infty$ is a round shrinking cylinder, then it must go extinct at a \emph{positive} time.
So $\MM^{\prime, 0}_\infty$ must be smooth at time $0$.

It follows that for large $j$ we can apply Proposition~\ref{Prop_diff_properties}\ref{Prop_diff_properties_b} to $\theta_j^{-1/2} \MM^{\prime,0}_j$ and $\theta_j^{-1/2} \MM^{\prime,1}_j$ based at time $0$.
This implies that there is a neighborhood $U \subset \IR^{n+1}$ of the origin such that for some uniform constant $C' > 0$
\[ d_H \big( (\spt \theta_j^{-1/2} \MM^{\prime,0}_j)_0  \cap U, (\spt \theta_j^{-1/2} \MM^{\prime, 1}_j)_0 \cap U \big) \leq C'\Vert \Qu(\MM^{\prime,0}_j)\Vert^{10} = C'\theta_j^{-5} \Vert\Qu(\MM^{0})\Vert^{10} . \]
As the origin is contained in the second subset, this implies
\begin{equation} \label{eq_closetoorigin}
  d \big( \bO, (\spt \theta_j^{-1/2} \MM^{\prime,0}_j)_0 \big) \leq C'\theta_j^{-5} \Vert \Qu(\MM^{0})\Vert^{10} . 
\end{equation}
Note that $\Vert \Qu (\MM^0) \Vert$ is independent of $j$.
On the other hand, due to smooth convergence, we have uniform lower bounds on the mean curvature of $\theta_j^{-1/2} \MM^{\prime,0}_j$ near $(\bO,0)$.
This implies that the time-slices of $\theta_j^{-1/2} \MM^{\prime,0}_j$ move at a speed that is uniformly bounded from below and hence $(\spt \MM^{\prime,0}_j)_{\theta_j^{-1} \Delta T_j}$ must have distance at least $c' \theta_j^{-1} |\Delta T_j|$ from $(\spt \theta_j^{-1/2} \MM^{\prime,0}_j)_{0}$ near $(\bO,0)$ for some uniform $c' > 0$.
So since this subset contains the origin, we obtain from \eqref{eq_closetoorigin} that
\[ C'\theta_j^{-5} \Vert \Qu(\MM^{0})\Vert^{10} \geq c' \theta_j^{-1} |\Delta T_j| \geq c' c \theta_j^{-1}, \]
which contradicts the fact that $\theta_j \to \infty$. 
\end{proof}
\bigskip

\subsection{Shrinking \texorpdfstring{$\Delta T$}{Delta T}}
Suppose that $\MM^0$ and $\MM^1$ are two asymptotically $(n,k)$-cylindrical mean curvature flows in $\IR^{n+1} \times (-\infty,T)$ that satisfy the assumptions of Proposition~\ref{Prop_M0isM1} for some constants $\eps, C$.
Consider the set $\mathcal T \subset \IR \setminus \{ 0 \}$ of offsets $\Delta T$ such that \eqref{eq_disjoint_DeltaT} holds.
Lemma~\ref{Lem_pjinintersection} shows that $(-\infty, \Delta T_-) \cup (\Delta T_+, \infty) \subset \mathcal T$ for some $\Delta T_- \leq 0$ and $\Delta T_+ \geq 0$.
Let us now assume that $\Delta T_\pm$ are chosen such that $|\Delta T_\pm|$ is minimal.

\begin{Lemma} \label{Lem_DeltaT_0}
$\Delta T_\pm = 0$.
\end{Lemma}

\begin{proof}
We will explain the proof of $\Delta T_+ = 0$; the proof for $\Delta T_-$ works the same way.
Suppose by contradiction that $\Delta T_+ > 0$.
Then we can find a sequence $\Delta T_j \nearrow \Delta T_+$ and points $(\bp_j, t_j)$ as in \eqref{eq_pjinintersection}.
After passing to a subsequence, we have the following convergence
\[ \MM^0_j := \MM^0 - (\bp_j, t_j + \Delta T_j) \lto \MM_\infty^0, \qquad  
\MM^1_j := \MM^1 - (\bp_j, t_j) \lto \MM_\infty^1, \]
where, by Lemma~\ref{Lem_pjinintersection}, the first convergence is locally smooth and the limit $\MM^0_\infty$ is asymptotically cylindrical and the second convergence is in the Brakke sense.
Then $(\bO,0) \in \spt \MM_\infty^0$ and $\spt \MM^1_\infty$ and by using the convexity of $\MM^0_\infty$ we obtain
\[ (\spt \MM^0_\infty)_{t + \Delta T'} \cap (\spt \MM^1_\infty)_t = \emptyset \qquad \text{for all} \quad t < 0, \quad \Delta T' > 0. \]
Note that this means that $(\spt \MM^1_\infty)_t$ is disjoint from a tubular neighborhood on one side of $(\spt \MM^0_\infty)_t$.
So we can apply the strong avoidance principle \cite[Theorem~3.4]{Choi_Haslhofer_Hershkovits_White_22} (see also \cite[Section~14]{Chodosh_MCF_notes}) to obtain that a component of $\MM_\infty^0 |_{(-\infty,0)}$ is contained in $\MM^1_\infty|_{(-\infty,0)}$.
But both flows are connected, since otherwise we could express them as the disjoint union of a component of $\MM_\infty^0 |_{(-\infty,0)}$ with another ancient flow, which contradicts the fact that $\Theta^{\MM^i_\infty}(\infty) < 2$.
Therefore, $\MM_\infty^0 |_{(-\infty,0)}=\MM^1_\infty|_{(-\infty,0)}$.

Let $r > 0$ be a large constant whose value we will determine in a moment.
Since $\MM_\infty^0$ is asymptotically cylindrical, we can find a point $(\bq_r,t_r) \in \IR^{n+1} \times \IR_{\leq 0}$ such that $r^{-1}(\MM^0_{\infty}-(\bq_r,t_r))$ is $\frac{\eps'}{4}$-close to $M_{\cyl}$ at $t=0$. So $r^{-1}(\MM^0-(\bp_j+\bq_r,t_j+\Delta T_j+t_r))$ is $\frac{\eps'}{2}$-close to $M_{\cyl}$ at $t=0$ for large $j$, and by choosing $r\gg \Delta T_j$, it follows that $r^{-1}(\MM^0-(\bp_j+\bq_r,t_j+t_r))$ is $\eps'$-close to $M_{\cyl}$ at $t=0$ for large $j$. So Proposition~\ref{Prop_diff_properties}\ref{Prop_diff_properties_b} applies to the flows $\MM^0_j - (\bO, \Delta T_j) = \MM^0 - (\bp_j, t_j)$ and $\MM^1_j= \MM^1 - (\bp_j, t_j)$ for sufficiently large $j$, $A = 100 \sqrt{n}$ and for some uniform constant $C$.
We obtain
\[ d_H \big( (\spt \MM^0_j)_{t_r + \Delta T_j} \cap B(\bq_r, Ar), (\spt \MM^1_j)_{t_r } \cap B(\bq_r , Ar) \big) \leq C\Big( \frac{\Vert \Qu (\MM^0) \Vert}{r} \Big)^{10} r. \]
As $\MM^0_j$ and $\MM^1_j$ converge to the same limit, this implies that for some uniform $C' > 0$ and large $j$
\begin{equation} \label{eq_dH_shifted}
 d_H \big( (\spt \MM^0_j)_{t_r + \Delta T_j} \cap B(\bq_r, Ar), (\spt \MM^0_j)_{t_r } \cap B(\bq_r , Ar) \big) \leq C' \Big( \frac{1}{r} \Big)^{10} r + \delta_j,
\end{equation}
for some $\delta_j\to 0$.
We can now argue as in the proof of Lemma~\ref{Lem_pjinintersection}.
Since $(\spt \MM^0)_{t} \cap B(\bq_r + \bp_j, Ar)$ is close to a cylinder at scale $r$, for $t \in [t_r + t_j, t_r + t_j + \Delta T_j]$, it must move at least at a velocity of $c' r^{-1}$ for some uniform $c' > 0$.
Hence both subsets in \eqref{eq_dH_shifted} must have a separation of at least $c' r^{-1} \Delta T_j$, which implies that for large $j$
\[ c' r^{-1} \Delta T_j \leq C\Big( \frac{1}{r} \Big)^{10} r. \]
We hence obtain a contradiction for large enough $r$.
\end{proof}
\medskip

\subsection{Proof of Proposition~\ref{Prop_M0isM1}}

\begin{proof}[Proof of Proposition~\ref{Prop_M0isM1}]
By Lemma~\ref{Lem_DeltaT_0} we know that \eqref{eq_disjoint_DeltaT} is true for all $\Delta T \neq 0$.
Since both flows are asymptotically cylindrical, there is a point $\bp \in \IR^{n+1}$ that must be contained in both $(\spt \MM^i)_{t_i}$, $i=1,2$, for some times $t_i \in \IR$.
So due to \eqref{eq_disjoint_DeltaT} we must have $t_0 = t_1$, which implies $(\bp, t_0) \in \spt \MM^0$ and $\spt \MM^1$.
Hence as in the proof of Lemma~\ref{Lem_DeltaT_0}, we can apply the strong avoidance principle \cite[Theorem~3.4]{Choi_Haslhofer_Hershkovits_White_22} (see also \cite[Section~14]{Chodosh_MCF_notes}) to obtain that a component of $\MM_\infty^0 |_{(-\infty,t_0)}$ is contained in $\MM^1_\infty|_{(-\infty,t_0)}$, which implies that $\MM^0|_{(-\infty,t_0)} = \MM^1|_{(-\infty,t_0)}$.
Since $\spt \MM^1$ is disjoint from a neighborhood of $\spt \MM^0$ due to \eqref{eq_disjoint_DeltaT} and since $\Theta^{\MM^i}(\infty) < 2$ the tangent cones of $\MM^1$ at all points of $(\spt \MM^0) \cap (\spt \MM^1)$ must be multiplicity one planes.
This implies that $(\spt \MM^0) \cap (\spt \MM^1) \subset \MM^{1,\reg}$ and that $(\spt \MM^0) \cap (\spt \MM^1)$ is an open subset of $\spt \MM^0$.
Since the subset is also closed and since $\spt \MM^0$ is connected, as argued in the proof of Lemma~\ref{Lem_DeltaT_0}, we must have $(\spt \MM^0) \cap (\spt \MM^0) = \spt \MM^0$.
Since $\spt \MM^1$ is connected by the same reason, we obtain $\spt \MM^0 = \spt \MM^1$.
Again, since $\Theta^{\MM^i}(\infty) < 2$, both flows must agree.
\end{proof}
\bigskip

\section{Asymptotic bound on the difference of two ancient flows} \label{sec_diff}
\subsection{Statement of the main results}
In this section we consider two asymptotically cylindrical flows $\MM^0$ and $\MM^1$ with the same quadratic mode at $-\infty$, so $\Qu(\MM^0) = \Qu(\MM^1)$, and study the deviation of their rescaled versions $\td\MM^0$ and $\td\MM^1$ as $\tau \to -\infty$.
We will show that this difference is governed by a dominant mode from the space $\sV_{> 0} = \sV_{\frac12} \oplus \sV_{1}$ of (rotationally symmetric and oscillatory) unstable modes, which decays exponentially as $\tau \to -\infty$.
If this dominant mode vanishes, then we must have $\MM^0 = \MM^1$.
We will also study how this dominant mode changes as we adjust $\MM^1$ by translations in space and time.
In Section~\ref{sec_proofs_I} we will use these results to classify asymptotically cylindrical flows and prove our main results.

Again, we will fix dimensions $1 \leq k \leq n-1$ for the remainder of this section and omit dimensional dependencies.
We will first state all our main results in this subsection and then carry out all proofs in the later subsections.
\medskip

Our first proposition forms the foundation of our discussion.
As in \cite{Bamler_Lai_PDE_ODI}, we characterize the rescaled flows $\td\MM^i$, $i=0,1$, by a graph function $u^i \in C^\infty( \DD^i )$ over the round cylinder.
We recall that asymptotics of $u^i_\tau$, as $\tau \to \infty$, has been characterized up to arbitrary polynomial order by our prior work; see \cite[Proposition~\refx{Prop_same_Q_close}]{Bamler_Lai_PDE_ODI}.
Our goal is now to study the difference
\[ v_\tau := u^1_\tau - u^0_\tau \]
through a similar process and to establish \emph{exponential} asymptotics characterizations for $v_\tau$.
As we mentioned in the introduction (see Subsection~\ref{subsec_structure}), the standard PDE-ODI principle is not enough to carry out such a fine analysis as it relies on a relatively coarse pseudolocality property.
However, luckily, we can use the Harnack-type estimate from Proposition~\ref{Prop_diff_properties} instead.
This leads to an ODI of the semi-stable mode $V^+(\tau) \in \sV_{\geq 0}$ of $v_\tau$ with an error term that is small \emph{in comparison with} $\Vert V^+(\tau) \Vert$.
Interestingly, due to this additional ingredient, our proof becomes far more straight-forward than the that of the PDE-ODI principle.

To describe this ODE, consider the second Taylor polynomial $Q_2^+ : \sV_{\geq 0} \to \sV_{\geq 0}$ of the non-linear term in the evolution equation for the mean curvature flow as graph over a cylinder (see \cite[Subsection~\refx{subsec_PDE_ODI_statement}]{Bamler_Lai_PDE_ODI} for further details).
By definition, this polynomial cannot have a constant or linear term, so it must be a quadratic form and we can define the associate bilinear map $Q_2^+ : \sV_{\geq 0} \times \sV_{\geq 0} \to \sV_{\geq 0}$ via polarization
\[ Q^+_2(V_1^+, V_2^+) := \tfrac12 \big(Q^+_2(V_1^+ + V_2^+) - Q^+_2(V_1^+ ) - Q^+_2( V_2^+) \big). \]
The next proposition shows that the projection $V^+_\tau \in \sV_{\geq 0}$ roughly obeys the an evolution equation of the following form:
\[ \partial_\tau V^{+} - L V^{+} - 2 Q_2^{+} (U^0_0,V^{+}) = O(|\tau|^{-2} \Vert V^+ \Vert), \]
where $U^0_0$ is the quadratic mode of $\MM^0$.
\medskip

\begin{Proposition} \label{Prop_diff_V}
There is a constant $C > 0$ such that the following is true.
Let $\MM^0$ and $\MM^1$ be two asymptotically $(n,k)$-cylindrical mean curvature flows, where we assume that $\MM^0$ is convex and rotationally symmetric.
Consider the corresponding rescaled flows $\td\MM_0, \td\MM_1$, so $\td\MM^{i,\reg}_\tau = e^{\tau/2} \MM^{i,\reg}_{-e^{-\tau}}$ and set
\[ R(\tau) := 10 \sqrt{\log|\tau| }. \]
There is a constant $C  > 0$ and a time $\ov\tau \in \IR$ (which may both depend on $\MM^0,\MM^1$ and $A, \eta$) such that for $\tau \leq \ov\tau$ the following is true:
\begin{enumerate}[label=(\alph*)]
\item \label{Prop_diff_V_a} There are smooth functions $u^i_\tau : \DD_\tau^i \to \IR$ with $\IB^k_{2R(\tau) - 1} \times \IS^{n-k} \subset \DD^i_\tau \subset \IB^k_{2R(\tau)} \times \IS^{n-k}$, for $i = 0,1$, such that
\[ \Gamma_{\cyl}(u^i_\tau) = (\spt \td\MM^i)_\tau \cap \IB^{n+1}_{R(\tau)} \subset \td\MM^{i,\reg}_\tau. \]
Set $v_\tau := u^1_\tau - u^0_\tau \in C^\infty (\DD^0_\tau \cap \DD^1_\tau)$ and define
\[ \qquad V^{+} (\tau) := \PP_{\sV_{\geq 0}}(v_\tau \omega_{R(\tau)}), \qquad V^{-} (\tau) := \PP_{\sV_{ < 0}}(v_\tau \omega_{R(\tau)}), \qquad \VV^-(\tau) := \Vert V^-(\tau) \Vert_{L^2_f}. \]
Here we take norms and projections with respect to the weighted $L^2_f$-inner product, as we did in \cite{Bamler_Lai_PDE_ODI}.
\item \label{Prop_diff_V_b}  Let $\ov U^0 : (-\infty, T) \to \IR^{k \times k}_{\geq 0}$ be the solution to the ODE \cite[(\refx{eq_barU_ODE})]{Bamler_Lai_PDE_ODI} corresponding to $\Qu(\MM^0)$ via \cite[Lemma~\refx{Lem_ODE_UQ}]{Bamler_Lai_PDE_ODI}.
Only the following asymptotics will be relevant (see \cite{Bamler_Lai_PDE_ODI}):
\begin{equation} \label{eq_ovU_asymp}
 \ov U^0 (\tau) =  \frac1{\sqrt{2}} \tau^{-1} G_{\Qu(\MM^0)}  + O(|\tau|^{-2} \log |\tau|), 
\end{equation}
where $G_{\Qu(\MM^0)} \in \IR^{k \times k}_{\geq 0}$ is the unique symmetric matrix corresponding to the projection onto the range of $\Qu(\MM^0)$.
The following ODIs hold:
\begin{equation}\label{e:Prop_diff_V_c}
\qquad\quad \Big\| \partial_\tau V^{+}(\tau) - L V^{+}(\tau) - 2 Q_2^{+} (\ov U^0(\tau),V^{+}(\tau))  \Big\|_{L^2_f} \leq C |\tau|^{-2} \Vert V^{+}(\tau) \Vert_{L^2_f} + C |\tau|^{-1} \VV^-(\tau).
\end{equation}
\begin{equation} \label{eq_dt_VVm_prop}
 \partial_\tau \VV^- (\tau) \leq - \tfrac1{2(n-k)} \VV^- (\tau) + C |\tau|^{-1} \Vert V^+ (\tau) \Vert_{L^2_f} 
\end{equation}
\item \label{Prop_diff_V_c}  If 
\[ \liminf_{\tau \to -\infty} e^{-2\tau} (\Vert V^+(\tau) \Vert_{L^2_f} + \VV^-(\tau)) < \infty, \]
then $\MM^0$ and $\MM^1$ agree for all times at which they are defined.
\item \label{Prop_diff_V_d} We have the following pointwise bounds on $\IB^{n-k}_{R(\tau)} \times \IS^{n-k}$
\begin{equation*} \label{eq_v_bound_exp}
 |v_\tau|, \ldots, |\nabla^{10} v_\tau| \leq C   e^r \big( \Vert V^{+}(\tau) \Vert + \VV(\tau)  \big)  . 
\end{equation*}
\end{enumerate}
\end{Proposition}
\medskip

We emphasize that the key property of \eqref{e:Prop_diff_V_c} and \eqref{eq_dt_VVm_prop} is that the right-hand sides is proportional to $\Vert V^+(\tau)\Vert + \VV^-(\tau)$, which is expected to decay exponentially.
This is true despite the fact that $V^+(\tau)$ arises from a difference $v_\tau$ of two functions that only decay to polynomial order.
In addition, the factors $C|\tau|^{-1}$ and $C|\tau|^{-2}$ decay fast enough to allow us to deduce precise asymptotic estimates for $V^+(\tau)$ by integrating both ODIs.

In order to make this asymptotic behavior precise, we need to decompose the space of \emph{unstable} modes $\sV_{> 0}$ into specific subspaces, depending on the null-space $N \subset \IR^k$ of $\Qu(\MM^0) = \Qu(\MM^1)$.
Recall from \cite[Lemma~\refx{Lem_mode_dec}]{Bamler_Lai_PDE_ODI} that
\[ \sV_{>0} = \sV_{\rot, 1} \oplus \sV_{\rot, \frac12} \oplus \sV_{\frac12,\Jac}, \]
where the first space is 1-dimensional and spanned by the constant, zeroth Hermite polynomial $\mathfrak p^{(0)}$, the second space is $k$-dimensional and spanned by the linear first Hermite polynomials $\mathfrak p^{(1)}_i$ and $\sV_{\frac12,\Jac}$ is $(n-k+1)$-dimensional and consists of $\Jac(\bY)$ for all constant Killing fields $\bY$ on $\IR^{n+1}$ that are perpendicular to the axis $\IR^k \times \bO^{n-k+1}$.
We now refine this decomposition as follows.

\begin{Definition}
Let $N \subset \IR^k$ be a (possibly trivial) linear subspace.
Then we define the splitting
\[ \sV_{\rot, \frac12} = \sV_{\rot, \frac12, N} \oplus \sV_{\rot, \frac12, N^\perp}, 
\]
where $\sV_{\rot, \frac12, N}$ (resp. $\sV_{\rot, \frac12, N^\perp}$) consists of all linear functions whose gradients are contained in $N$ (resp. $N^\perp$).
\end{Definition}

These spaces will have the following geometric interpretation for the mode $V^+(\tau)$, assuming that $N \subset \IR^k$ is the nullspace of $\Qu(\MM^0)$:
\begin{itemize}
\item $\sV_{\rot, \frac12, N}$ has the most geometric significance, as it can be used to distinguish different flying wing solitons.
\item $\sV_{\frac12, \Jac}$ is generated by translations perpendicular to the axis $\IR^k \times \bO^{n-k+1}$.
\item $\sV_{\rot, \frac12,N^\perp}$ is generated by translations parallel to $N^\perp$.
\item $\sV_{\rot, 1}$ is generated by time-shifts.
\end{itemize}
Our next result characterizes the asymptotic behavior of $V^{+}(\tau)$ as $\tau \to -\infty$.
Specifically, we show that $V^{+}_\tau$ exhibits one of three  dominant decay behaviors as $\tau \to -\infty$:
\begin{itemize}
\item 
$V^{+}(\tau) \approx e^{\tau/2} V^{+}_{-\infty}$, where $V^{+}_{\infty} \in \sV_{\rot, \frac12, N} \oplus \sV_{\frac12,\Jac}$.
\item 
$V^{+}(\tau) \approx |\tau|^{-1} e^{\tau/2} V^{+}_{-\infty}$, where $V^{+}_{-\infty} \in \sV_{\rot, \frac12,N^\perp}$.
\item 
$V^{+}(\tau) \approx e^{\tau} V^{+}_{-\infty}$, where $V^{+}_{-\infty} \in \sV_{\rot, 1}$.
\end{itemize}
If $V^{+}_{-\infty}$ vanishes in all cases, then we obtain an even faster decay of the form $V^{+}(\tau) \lesssim e^{2\tau}$, which implies $\MM^0 = \MM^1$.
\medskip

\begin{Proposition} \label{Prop_Vpp_asymp}
Consider the setting of Proposition~\ref{Prop_diff_V}, assume that $\Qu(\MM^0) = \Qu(\MM^1)$ and let $N \subset \IR^k$ be the nullspace of this matrix. 
Then there is a vector $V_{-\infty}^{+} \in \sV_{>0}$ that satisfies one of the following properties:
\begin{enumerate}[label=(\alph*)]
\item\label{p:V_1/2_null} $V^{+}_{-\infty} \in   \sV_{\rot, \frac12, N} \oplus \sV_{\frac12,\Jac}$ and $V^{+}_{-\infty} = \lim_{\tau \to -\infty} e^{-\tau/2} V^+(\tau) \neq 0$.
\item\label{p:V_1/2_non_null} $V^{+}_{-\infty} \in \sV_{\rot, \frac12,N^\perp}$ and $V^{+}_{-\infty} = \lim_{\tau \to -\infty} (-\tau) e^{-\tau/2} V^+(\tau) \neq 0$.
\item\label{p:V_1} $V^{+}_{-\infty} \in \sV_{\rot, 1}$ and $V^{+}_{-\infty} = \lim_{\tau \to -\infty} e^{-\tau} V^+(\tau)$.
If $V^{+}_{-\infty} = 0$, then  $\MM^0$ and $\MM^1$ agree at all times for which they are defined.
\end{enumerate}
Moreover for small $\tau$ we have a bound of the form 
\begin{equation} \label{eq_VmVptaum1}
 \VV^-(\tau) \leq C|\tau|^{-1} \Vert V^+(\tau) \Vert 
\end{equation}
and we even have $C^9_{\loc}$-convergence of $e^{-\tau/2} v_\tau$, $(-\tau) e^{-\tau/2} v_\tau$ or $e^{-\tau} v_\tau$ to $V_{-\infty}^+$, depending on the case.
So the vector $V^+_{-\infty}$ is uniquely determined and it does not depend on the precise choice of the cutoff functions $\omega_{R(\tau)}$.
\end{Proposition}

We can hence make the following definition.

\begin{Definition}
If $\MM^0$ and $\MM^1$ are two asymptotically $(n,k)$-cylindrical mean curvature flows with $\Qu(\MM^0) = \Qu(\MM^1)$, then their \textbf{difference at $-\infty$,}  
\[ V^{+}_{-\infty} (\MM_0, \MM_1) \in \sV_{>0}, \] 
is the vector from Proposition~\ref{Prop_Vpp_asymp}.
\end{Definition}

\begin{Remark}
It should be possible to merge Cases~\ref{p:V_1/2_null} and \ref{p:V_1/2_non_null} into a single case, since these differ only by a polynomial factor.
In other words, one may expect an expansion of the form
\[ V^{+}(\tau) =  e^{\tau/2} V^{+,1}_{-\infty} + (  - \tau)^{-1} e^{\tau/2} V^{+,2}_{-\infty} + O\big(  (  - \tau)^{-2} e^\tau \big), \]
where $V^{+,1}_{-\infty} \in  \sV_{\rot, \frac12, N^\perp} \oplus \sV_{\frac12,\Jac}$ and $V^{+,2}_{-\infty}  \in \sV_{\rot, \frac12,N^\perp}$.
Carrying out this analysis, however, would require studying higher order Taylor approximations of the nonlinear term $Q$ in the rescaled mean curvature flow equation, which we do not pursue here since it is unnecessary for our purposes.
\end{Remark}

We record:

\begin{Corollary}\label{cor:unique_V}
If $\Qu(\MM^0) = \Qu( \MM^1)$ and if both flows are defined on the same time-interval and if $\MM^0$ is convex and rotationally symmetric, then
\[ V^{+}_{-\infty} (\MM^0, \MM^1)=0 \qquad \Longleftrightarrow \qquad \MM^0 = \MM^1. \]
\end{Corollary}

We will see a posteriori that all asymptotically cylindrical flows are rotationally symmetric about an axis parallel to $\IR^k \times \bO^{n-k+1}$, so assumption of the corollary is always fulfilled modulo a translation.

The last result establishes useful identities for the difference at infinity.

\begin{Proposition}\label{prop:actions}
Let $\MM^0$ and $\MM^1$ be two asymptotically $(n,k)$-cylindrical mean curvature flows with $\Qu(\MM^0) = \Qu(\MM^1)$ and assume that $\MM^0$ is convex and rotationally symmetric.
Let $N \subset \IR^k$ be the nullspace of $\Qu(\MM^0) = \Qu(\MM^1)$.
Then the following is true:
\begin{enumerate}[label=(\alph*)]

\item\label{prop:actions_c} If $S \in O(n+1)$ is an orthogonal linear map with $S (M_{\cyl}) = M_{\cyl}$, then the rotated flows $\MM^{i,\prime} := S\MM^i$ satisfy
\[ V^{+}_{-\infty} (\MM^{0,\prime}, \MM^{1,\prime}) = S \big(V^{+}_{-\infty} (\MM^0, \MM^1) \big), \] 
where $S : \sV_{>0} \to \sV_{>0}$ acts in the natural way.
That is, it acts trivially on $\sV_1$, and on $\sV_{\rot, \frac12} \oplus \sV_{\frac12, \Jac} \cong \IR^k \times \IR^{n-k+1}$ via the standard action, which must preserve this splitting whenever $S (M_{\cyl}) = M_{\cyl}$.

\item\label{prop:actions_g} Consider a third asymptotically $(n,k)$-cylindrical mean curvature flow $\MM^2$ with $\Qu(\MM^0) = \Qu(\MM^1) = \Qu(\MM^2)$ and assume that $\MM^2$ is convex and rotationally symmetric.
Then
\[ \qquad\quad \PP_{\sV_{\rot, \frac12, N} \oplus \sV_{\frac12, \Jac}} \big( V^{+}_{-\infty} (\MM^0, \MM^2) + V^{+}_{-\infty} (\MM^2, \MM^1) - V^{+}_{-\infty} (\MM^0, \MM^1) \big) 
= 0.  \]

If $\PP_{\sV_{\rot, \frac12, N} \oplus \sV_{\frac12, \Jac}} ( V^{+}_{-\infty} (\MM^i, \MM^j) ) = 0$ for all $(i,j) \in \{ (0,2),(2,1),(0,1) \}$,  then
\[ \qquad\quad \PP_{\sV_{\rot, \frac12, N} } \big( V^{+}_{-\infty} (\MM^0, \MM^2) + V^{+}_{-\infty} (\MM^2, \MM^1) - V^{+}_{-\infty} (\MM^0, \MM^1) \big) 
= 0,  \]
and if in addition also $\PP_{\sV_{\rot, \frac12, N}} ( V^{+}_{-\infty} (\MM^i, \MM^j) ) = 0$ for all $(i,j) \in \{ (0,2),(2,1),(0,1) \}$, then
\[ \qquad\quad \PP_{\sV_{\rot, 1} } \big( V^{+}_{-\infty} (\MM^0, \MM^2) + V^{+}_{-\infty} (\MM^2, \MM^1) - V^{+}_{-\infty} (\MM^0, \MM^1) \big) 
= 0.  \]
\item\label{prop:actions_d} Let $\mathbf v \in \bO^k \times \IR^{n-k+1}$ be a vector orthogonal to the axis of rotation of $\MM^0$ and set $V' := \Jac (\mathbf v) \in \sV_{\Jac, \frac12}$, where we view $\mathbf v$ as a constant Killing field on $\IR^{n+1}$.
Let $\MM^{1,\prime} := \MM^1 + (\mathbf v,0)$ be the corresponding translated flow.
Then the following is true 
\[ \PP_{\sV_{\rot, \frac12, N} \oplus \sV_{\frac12, \Jac}} \big( V^{+}_{-\infty} (\MM^0, \MM^{1,\prime}) - V^{+}_{-\infty} (\MM^0, \MM^1)\big) = V'. \]
\item \label{prop:actions_e}
Let $\mathbf v \in \IR^{k} \times \bO^{n-k+1}$ and consider the translated flow $\MM^{1,\prime} := \MM_1 + (\mathbf v,0)$.
Then 
\begin{align*}
\PP_{\sV_{\rot, \frac12, N^\perp} \oplus \sV_{\frac12, \Jac}} \big( V^{+}_{-\infty} (\MM^0, \MM^{1,\prime}) \big) &= \PP_{\sV_{\rot, \frac12, N^\perp} \oplus \sV_{\frac12, \Jac}} \big(V^{+}_{-\infty} (\MM^0, \MM^{1}) \big).
\end{align*}
Moreover, if both projections vanish (so if  Case~\ref{p:V_1/2_null} of Proposition~\ref{Prop_Vpp_asymp} does not apply) and if $\proj_{N^\perp} (\mathbf v) = \sum_{i=1}^k v_i^\perp \mathbf e_i$, then
\[ \PP_{\sV_{\rot, \frac12, N^\perp}} \big( V^{+}_{-\infty} (\MM^0, \MM^{1,\prime}) - V^{+}_{-\infty} (\MM^0, \MM^1)\big) =  \frac1{\sqrt 2} \sum_{i=1}^k v_i^\perp \mathfrak p^{(1)}_i. \]
\item\label{f:time-translate} 
Let $\MM^{1,\prime} = \MM^1 + (0, \Delta T)$ for some time-shift  $\Delta T \in \IR$.
Then 
\[ \PP_{\sV_{\frac12}} \big( V^{+}_{-\infty} (\MM^0, \MM^{1,\prime}) \big) = \PP_{\sV_{\frac12}} \big( V^{+}_{-\infty} (\MM^0, \MM^1) \big) , \]
where $\sV_{\frac12}= \sV_{\rot, \frac12, N} \oplus \sV_{\frac12, \Jac} \oplus\sV_{\rot, \frac12, N^\perp}$. 
Moreover, if both projections vanish (so if Case~\ref{p:V_1} of Proposition~\ref{Prop_Vpp_asymp} applies), then 
\[V^{+}_{-\infty} (\MM^0, \MM^{1,\prime}) - V^{+}_{-\infty} (\MM^0, \MM^1) = \frac12 \Delta T \, \mathfrak p^{(0)}. \]

\item\label{prop:actions_gg} Let $\mathbf v = \sum_{i=1}^k v_i \mathbf e_i \in \IR^k \times \bO^{n+1}$ and suppose that $\MM^0 = \MM^0 + (\mathbf v, 0)$  is invariant under translation in the $\mathbf v$-direction.
Suppose also that $\MM^1$ is convex and rotationally symmetric and suppose that
\[ V^+_{-\infty} (\MM^0, \MM^1) =\sum_{i=1}^k b_i \mathfrak p^{(1)}_i \in \sV_{\rot, \frac12, N}. \]
Then for $\MM^{1,\prime} = \MM^1 + (\mathbf v, 0)$ we have
\begin{equation} \label{eq_VpM11pp0}
 V^{+}_{-\infty} (\MM^{1}, \MM^{1,\prime}) =-\frac{1}{\sqrt2} \Big(\sum_{i=1}^k v_i b_i \Big) \mathfrak p^{(0)} \in \sV_1.
\end{equation}
\end{enumerate}
\end{Proposition}
\bigskip

\subsection{Proof of Proposition~\ref{Prop_diff_V}}

\begin{proof}[Proof of Proposition~\ref{Prop_diff_V}.]
We follow the general strategy from the proof of \cite[Lemma~\refx{l:evolution_ODE}]{Bamler_Lai_PDE_ODI}, but the additional control on $u^0$ and $u^1$, together with the Harnack-type estimate for $v$ from Proposition~\ref{Prop_diff_properties}, makes the proof substantially more elementary.

We will denote by $C$ a generic constant and we will indicate dependencies in parentheses.
For convenience we will write $\omega_\tau = \omega_{R(\tau)}$ and we will frequently drop the ``$\tau$''-subscript or ``$(\tau)$''-argument in time-dependent quantities.
When the context is clear, we will also often omit the subscript ``$L^2_{f}$'' on norms and inner products. 

Assertion~\ref{Prop_diff_V_a} is a direct consequence of \cite[Proposition~\refx{Prop_dom_qu_asymp}]{Bamler_Lai_PDE_ODI} applied to $\MM^0$ and $\MM^1$ for $J = 100$, $m = 10$, so the quantities $V^+(\tau)$ and $\VV^-(\tau)$ are well defined for sufficiently small $\tau$.
Let $U^{0,+} (\tau) = U_1^0 (\tau)  + \ldots + U^0_{-100}(\tau) \in \sV_{\geq -100}$ be the function supplied by applying this theorem to $u^0$ and recall that for $\tau \ll 0$
\begin{align}
 \Vert u_\tau^0 \Vert_{C^{10}(\DD_\tau^0)} , \quad \Vert u_\tau^1 \Vert_{C^{10}(\DD_\tau^1)}
 &\leq C |\tau|^{-1},  \label{eq_u0u1tauinv} \\
 \Vert u_\tau^0 - U^{0,+}(\tau) \Vert_{C^{10}(\DD_\tau^0)} 
 &\leq C |\tau|^{-2},  \label{eq_u0U0p} \\
  \Vert U^{0,+}(\tau) - \ov U^{0} (\tau) \Vert_{L^2_f} 
&\leq \Vert U^{0,+}(\tau) - U^{0}_0 (\tau) \Vert_{L^2_f} + \Vert U^{0}_0 (\tau) - \ov U^{0} (\tau) \Vert_{L^2_f} 
\leq C |\tau|^{-2}, 
\end{align}
where $\ov U^0$ solves the ODE \cite[(\refx{eq_barU_ODE})]{Bamler_Lai_PDE_ODI} and has the asymptotics \eqref{eq_ovU_asymp}.

Assertion~\ref{Prop_diff_V_c} follows by combining Proposition~\ref{Prop_diff_properties}\ref{Prop_diff_properties_b} with Proposition~\ref{Prop_M0isM1}.
Note that if this decay condition holds, then $V^+(\tau) = \VV^-(\tau) \equiv 0$, so all other assertions become trivial.
Therefore, we may assume in the following that for small enough $\tau$
\[ \Vert V^+ (\tau) \Vert + \VV^-(\tau) \geq e^{2\tau} \]
and Assertion~\ref{Prop_diff_V_d} follows from Proposition~\ref{Prop_diff_properties}\ref{Prop_diff_properties_a}.

It remains to show Assertion~\ref{Prop_diff_V_b}.
By direct computation, we obtain the following evolution equation for $v\omega$ (see \cite[Lemma~\refx{Lem_structure_MCF_graph_equation}]{Bamler_Lai_PDE_ODI}): 
\begin{equation} \label{eq_evol_eq_withE}
\partial_\tau (v\omega)=L (v\omega) +  Q[u^1\omega] - Q[u^0 \omega]  + E, 
\end{equation}
where
\begin{equation}\label{eq:E_V}
   E= \partial_\tau  \omega \cdot  v - 2\nabla \omega \cdot \nabla v -\Delta_{f} \omega \cdot v  + \big( Q[u^1] - Q[u^0 ] \big) \omega -  \big( Q[u^1\omega] - Q[u^0 \omega] \big) .
\end{equation}
Projecting \eqref{eq_evol_eq_withE} to $\sV_{\geq 0},\sV_{< 0}$, and noting that $L$ commutes with these projections, yields for $\td u^i := u^i \omega$ and $\td v := v \omega$
\begin{align}
 \partial_\tau V^+ &= L V^+ +  \PP_{\sV_{\geq 0}} \big(   Q[\td u^1] - Q[\td u^0 ]  \big)  
+ \PP_{\sV_{\geq 0}}  E ,  \label{eq_dtVp_L_etc} \\
 \partial_\tau V^- &= L V^-   + \PP_{\sV_{< 0}} \big(   Q[\td u^1] - Q[\td u^0 ] \big)
 + \PP_{\sV_{<0}}  E. \label{eq_dtVm_L_etc} 
\end{align}

Let us first bound the term $E$.
Recall that $Q[u] = Q(u,\nabla u, \nabla^2 u)$ is a smooth  function, which depends only on $u$ and its first and second derivative and not on the space or time-parameters \cite[Lemma~\refx{Lem_structure_MCF_graph_equation}]{Bamler_Lai_PDE_ODI}.
As in the proof of \cite[Lemma~\refx{lem:Q_initial}]{Bamler_Lai_PDE_ODI}, we can write
\begin{align} 
 Q[u^1]-Q[u^0] &= Q^*_0 * v + Q^*_1 * \nabla v + Q^*_2 * \nabla^2 v ,  \notag \\
   Q[\td u^1 ]-Q[\td u^0 ] &= \td Q^*_0 * \td v + \td Q^*_1 * \nabla \td v + \td Q^*_2 * \nabla^2 \td v , \label{eq_Qu1Qu0}
\end{align}
where 
\[ Q^*_j = Q^*_j[u^0,u^1] = \int_0^1 \partial_{\nabla^j u} Q\big(su_1+(1-s)u_0, s \nabla u_1+(1-s) \nabla u_0, s \nabla^2 u_1+(1-s) \nabla^2 u_0 \big) \, ds \]
and similar for $\td Q^*_j$.
Since $Q$ vanishes to second order, we get by \eqref{eq_u0u1tauinv} for $m = 0,1,2$
\begin{equation} \label{eq_Qstar_bound}
 |\nabla^m Q^*_j|, \; |\nabla^m \td Q^*_j| \leq C \sum_{i=0}^1 \sum_{j=0}^4 |\nabla^j u_i| \leq C |\tau|^{-1}. 
\end{equation}
Therefore, using Assertion~\ref{eq_v_bound_exp} (note that $|\triangle_f \omega| \leq C r$)
\[ |E| \leq C ( r |v| + |\nabla v| + |\nabla^2 v| )
\leq C r e^r \big( \Vert V^{+} \Vert + \VV^- \big). \]
In addition $E_\tau$ is supported on $(\IB^k_{R(\tau) } \setminus \IB^k_{ R(\tau) -1}) \times \IS^{n-k}$, so we obtain for $R = R(\tau)$
\begin{multline}
\Vert E \Vert_{L^2_f} 
  \leq \Big(  (C R e^R)^2   R^{k-1}  e^{-\frac{( R-1)^2}{4}} \Big)^{1/2}  |\tau|^{-1} (\Vert V^+ \Vert + \VV^-)
   \leq  C e^{-\frac{R^2}{10}} |\tau|^{-1}  (\Vert V^+ \Vert + \VV^-) \\
   \leq C |\tau|^{-2} (\Vert V^+ \Vert + \VV^-). \label{eq_I4}
\end{multline}

Next, we analyze the projections of the term $ Q[\td u^1] - Q[\td u^0 ]$.
Combining~\eqref{eq_Qu1Qu0}, \eqref{eq_Qstar_bound} and
\eqref{eq_v_bound_exp} implies
\[ \big\Vert Q[\td u^1 ] - Q[\td u^0 ] \big\Vert_{L^2_f} \leq \sum_{j=0}^2 \Vert \td Q^*_j \Vert_{L^\infty} \Vert \nabla^j \td v  \Vert_{L^2_f}
\leq C |\tau|^{-1}  \big( \Vert V^+ \Vert + \VV^-). \]
Using this bound, \eqref{eq_dtVm_L_etc}, \eqref{eq_I4} and the fact that $\td v = V^+ + V^-$, we obtain the second evolution inequality \eqref{eq_dt_VVm_prop} for sufficiently small $\tau$ (see \cite[Lemma~\refx{Lem_mode_dec}]{Bamler_Lai_PDE_ODI}):
\begin{align*}
\| V^- \|_{L^2_f} &\cdot \partial_\tau \| V^- \|_{L^2_f} =
 \tfrac12 \partial_\tau \| V^- \|^2_{L^2_f}  \\
&= \big\langle V^-, L V^- \big\rangle 
+ \big\langle V^-, Q[\td u^1 ] - Q[\td u^0 ] \big\rangle + \big\langle V^-, E \big\rangle \\
&\leq -\tfrac1{n-k} \Vert V^- \Vert_{L^2_f}^2 
+ \Vert V^- \Vert_{L^2_f} \cdot \Vert Q[\td u^1 ] - Q[\td u^0 ]\Vert_{L^2_f}  
+ \Vert V^- \Vert_{L^2_f} \cdot \Vert E \Vert_{L^2_f} \\
&\leq -\tfrac1{n-k} \Vert V^- \Vert_{L^2_f}^2 + C|\tau|^{-1} \Vert V^- \Vert_{L^2_f} \big( \Vert V^+ \Vert + \VV^-) \\
&\leq -\tfrac1{2(n-k)} \Vert V^- \Vert_{L^2_f}^2 + C|\tau|^{-1} \Vert V^- \Vert_{L^2_f} \Vert V^+ \Vert .
\end{align*}

To prove the first evolution inequality~\eqref{e:Prop_diff_V_c}, we use \eqref{eq_dtVp_L_etc} and find
\begin{equation*} 
 \big\| \partial_\tau V^{+} - L V^{+} - 2 Q_2^{+} (U^0_{0},V^{+})  \big\|_{L^2_f}
\leq  \big\Vert \PP_{\sV_{\geq 0}}   (Q[\td u^1] - Q[\td u^0 ]) - 2 Q_2^+(\ov U^0, V^+) \big\Vert_{L^2_f} + \Vert E \Vert_{L^2_f}.
\end{equation*}
Since the last term can again be bounded by \eqref{eq_I4}, it remains to establish a bound of the form
\[ \big\Vert \PP_{\sV_{\geq 0}} (Q[\td u^1] - Q[\td u^0 ])  - 2 Q_2^+(\ov U^0, V^+) \big\Vert_{L^2_f}  \leq C |\tau|^{-2} \Vert V^+ \Vert + C |\tau|^{-1} \VV^-. \]
To achieve this, we fix some arbitrary $V' \in \sV_{\geq 0}$ with $\Vert V' \Vert = 1$ and we aim to establish a bound of the form
\begin{equation} \label{eq_Vp_desired_bound}
 \big\langle V',   Q[\td u^1] - Q[\td u^0 ]  - 2 Q_2^+(\ov U^0, V^+) \big\rangle \leq C |\tau|^{-2} \Vert V^+ \Vert + C |\tau|^{-1} \VV^-. 
\end{equation}

We first use integration by parts, Cauchy-Schwarz, \eqref{eq_Qstar_bound} and the fact that $V'$ has polynomial growth to deduce (see again the proof of \cite[Lemma~\refx{lem:Q_initial}]{Bamler_Lai_PDE_ODI})
\begin{align*}
 \langle V', \td Q^*_0 *  V^-  \rangle 
 &\leq \Vert V' * \td Q^*_0  \Vert \cdot \Vert V^- \Vert 
\leq C |\tau|^{-1} \Vert V^- \Vert  \\
 \langle V', \td Q^*_1 *  \nabla V^-  \rangle 
&= \langle V' * \td Q^*_1 ,  \nabla V^-  \rangle 
=  \langle \nabla (V' * \td Q^*_1 ) + \nabla f * (V' * \td Q^*_1 ),   V^-  \rangle \\
&\leq C |\tau|^{-1} \Vert V^- \Vert \\
 \langle V', \td Q^*_2 *  \nabla^2 V^-  \rangle 
&= \langle V' * \td Q^*_2 ,  \nabla^2 V^-  \rangle  \\
&=  \langle \nabla^2 (V' * \td Q^*_2 ) + \nabla^2 f * (V' * \td Q^*_2 ) +  \nabla f * \nabla (V' * \td Q^*_2 ),   V^-  \rangle \\
&\leq C |\tau|^{-1} \Vert V^- \Vert 
\end{align*}
Therefore by \eqref{eq_Qu1Qu0} and since $\td v = V^+ + V^-$
\begin{equation} \label{eq_Vp_1}
 \Big| \langle V', Q[\td u^1]-Q[\td u^0] \rangle 
- \sum_{j=0}^2 \langle V', \td Q^*_j * \nabla^j V^+  \rangle \Big|
\leq C |\tau|^{-1} \Vert V^- \Vert. 
\end{equation}
Next, consider the second Taylor polynomial $Q'_2 [u] =  Q'_2(u, \nabla u, \nabla^2 u)$ of $Q[u] = Q(u,\nabla u, \nabla^2)$ at $(0,0,0)$.
Note that $Q'_2$ is homogeneous quadratic, so we can write $ Q'_2[u] =  Q'_2[u,u]$, where $ Q'_2[u'_0,u'_1] =  Q'_2 ((u'_0, \nabla u'_0, \nabla^2 u'_0), (u'_1, \nabla u'_1, \nabla^2 u'_1))$ is the associated symmetric bilinear form. 
Then for $U'_0, U'_1 \in \sV_{\geq 0}$ we have 
\[ Q_2(U'_0, U'_1) =  Q'_2[U'_0, U'_1] \]
and by definition of $\td Q^*_j$ we have by \eqref{eq_u0u1tauinv}
\begin{equation*}
 \bigg| \sum_{j=0}^2 \td Q^*_j  * \nabla^j V^+ -  2Q'_2[\td u^0, V^+] \bigg|
\leq C \bigg( \sum_{i=0}^1 \sum_{j=0}^2 |\nabla^j \td u^i| \bigg)^2 \bigg( \sum_{j=0}^2 |\nabla^j V^+|  \bigg) 
\leq C |\tau|^{-2} \bigg( \sum_{j=0}^2 |\nabla^j V^+|  \bigg), 
\end{equation*}
hence, again due to the polynomial growth of $V'$,
\begin{equation} \label{eq_Vp_2}
  \bigg| \bigg\langle V',  \sum_{j=0}^2 \td Q^*_j  * \nabla^j V^+   \bigg\rangle -  \langle V',  2Q'_2[\td u^0, V^+]  \rangle \bigg| \leq C |\tau|^{-2} \Vert V^+ \Vert. 
\end{equation}
Next, we obtain using \eqref{eq_u0U0p} that
\begin{equation} \label{eq_Vp_3}
 \big| \langle V',  Q'_2[\td u^0, V^+]  \rangle  - \langle V',  Q'_2[U^{0,+} \omega, V^+]  \rangle \big| 
= \big| \langle V',  Q'_2[ (u^0  -U^{0,+}) \omega, V^+]  \rangle  \big|
\leq C |\tau|^{-2} \Vert V^+ \Vert .
\end{equation}
Since $U^{0,+},V^+ \in \sV_{\geq -100}$ have bounded polynomial growth (see \cite[Lemma~\refx{Lem_polynomial_bounds}]{Bamler_Lai_PDE_ODI}), we can estimate
\[ |  Q'_2[U^{0,+} \omega, V^+]  -    Q'_2[U^{0,+} , V^+] | \leq C R^C(\tau) \Vert U^{0,+} \Vert \cdot \Vert V^+ \Vert 
\leq C R^C(\tau) |\tau|^{-1} \Vert V^+ \Vert, \]
and the left-hand side is supported on $(\IR^k \setminus \IB^k_{ R(\tau) -1}) \times \IS^{n-k}$.
So we obtain as in \eqref{eq_I4} that
\begin{equation} \label{eq_Vp_4}
 \big| \langle V', Q'_2[U^{0,+} \omega, V^+] \rangle - \langle V', Q'_2[U^{0,+} , V^+] \rangle \big| \leq C|\tau|^{-2} \Vert V^+ \Vert. 
\end{equation}
Lastly, by \eqref{eq_u0U0p} and since $V', U^{0,+},V^+ \in \sV_{\geq -100}$ 
\begin{equation} \label{eq_Vp_5}
 \big| \langle V', Q'_2[U^{0,+} , V^+]  \rangle  - \langle V',  Q'_2[\ov U^0 , V^+]  \rangle \big| \leq C |\tau|^{-2} \Vert V^+ \Vert . 
\end{equation}
Combining \eqref{eq_Vp_1}, \eqref{eq_Vp_2}, \eqref{eq_Vp_3}, \eqref{eq_Vp_4} and \eqref{eq_Vp_5} implies \eqref{eq_Vp_desired_bound}, which shows the evolution inequality \eqref{e:Prop_diff_V_c} and finishes the proof.
\end{proof}
\bigskip

\subsection{Proof of Proposition~\ref{Prop_Vpp_asymp}}
We need the following lemma characterizing the symmetric form $Q_2^+ : \sV_{\geq 0} \times \sV_{\geq 0} \to \sV_{\geq 0}$:

\begin{Lemma}\label{l:expansion_Q_2}
If $U_0 = \sum_{i,j=1}^k c_{ij} \mathfrak{p}^{(2)}_{ij} \in \sV_{\rot, 0}$ and $V \in \sV_{\geq 0}$ with $\PP_{\sV_{\rot}} V = a \mathfrak{p}^{(0)} + \sum_{i=1}^k b_i \mathfrak{p}^{(1)}_i$, then
\begin{equation} \label{eq_PVpQpU0V}
\PP_{\sV_{>0}} \big( Q^{+}_2 (U_0, V) \big) = - \sqrt 2 \sum_{i,l=1}^k  c_{il} b_l \mathfrak{p}^{(1)}_i .
\end{equation}

\end{Lemma}

\begin{proof}
If $V \in \sV_{\rot, \geq 0}$, then the lemma follows from \cite[Lemma~\refx{Lem_Q2}]{Bamler_Lai_PDE_ODI}.

So suppose now that $V \in \sV_{\osc, \geq 0} = \sV_{\Jac, \geq 0}$.
Since this space generated by products of affine linear functions on $\IR^k$ and first spherical harmonics on $\IS^{n-k}$ (see \cite[Lemma~\refx{Lem_mode_dec}]{Bamler_Lai_PDE_ODI}), we may assume by linearity that $V(\bx,\by) = V'(\bx) V''(\by)$ for an affine linear $V'$ and a first spherical harmonic $V''$.
The second Taylor polynomial of the non-linear part of the evolution equation from \cite[Lemma~\refx{Lem_MCF_u_explicit}]{Bamler_Lai_PDE_ODI} is 
\[ -\tfrac12 u^2 -  u \cdot \triangle_{\IS^{n-k}} u - |\nabla_{\by} u|^2, \]
where $\triangle_{\IS^{n-k}} u$ is the spherical Laplacian of $u$ and $\nabla_{\by} u$ is the projection of $\nabla u$ to the spherical factor.
Since these operators applied to $U_0$ vanish, the left-hand side of \eqref{eq_PVpQpU0V} equals the projection of $-\frac12 U_0 V' V'' + c_{n-k} U_0 V' V''$ to $\sV_{>0}$, for some dimensional $c_{n-k} \in \IR$.
Let $W \in \sV_{\osc, >0} = \sV_{\Jac, \frac12}$, so $W(\bx,\by) = W''(\by)$ must be a spherical harmonic on $\IS^{n-k}$.
Then the $L^2_f$-inner product of $Q^+_2(U_0,V)$ with $W$ is proportional to
\begin{multline*}
  \int_{\IR^k \times \IS^{n-k}} U_0(\bx) V'(\bx) V''(\by) W''(\by) e^{-f(\bx)} d\bx \, d\by \\
=  \int_{\IR^k } U_0(\bx) V'(\bx) e^{-f(\bx)}  d\by \cdot \int_{ \IS^{n-k}}  V''(\by) W''(\by)   d\by.
\end{multline*}
Since $V'$ is a linear combination of zeroth and first Hermite polynomials and $U_0$ is a linear combination of second Hermite polynomials, the first integral on the right-hand side.
So $Q^+_2(U_0,V)$ must be perpendicular to $\sV_{\osc, > 0}$.
On the other hand, if $W \in \sV_{\rot, > 0}$, then it is of the form $W(\bx,\by) = W'(\bx)$ and a similar calculation shows that its inner product with $-\frac12 U_0 V' V'' + c_{n-k} U_0 V' V''$ must vanish.
\end{proof}
\bigskip

\begin{proof}[Proof of Proposition~\ref{Prop_Vpp_asymp}.]
In the following $C$ will denote a generic constant and $O(X)$ will denote a term bounded by $C X$.
After applying a rotation, we may assume without loss of generality that $N = \spann\{ \mathbf e_1, \ldots, \mathbf e_l \}$ for some $l \in \{ 0, \ldots, k \}$.
So 
\[ \ov U^0(\tau) := - \frac1{\sqrt{2}} \sum_{i=l+1}^l (- \tau)^{-1} \mathfrak{p}^{(2)}_{ii} + O(|\tau|^{-2} \log|\tau|). \] 

\begin{Claim} \label{Cl_VpCtaum1Vp}
For small $\tau$ we have
\[ \VV^-(\tau) \leq C |\tau|^{-1} \Vert V^+(\tau) \Vert. \]
\end{Claim}

\begin{proof}
The bound \eqref{e:Prop_diff_V_c} implies for small $\tau$
\[ \partial_\tau \Vert V^+(\tau) \Vert \geq -C|\tau|^{-1} \Vert V^+(\tau) \Vert - C |\tau|^{-1} \VV^-(\tau). \]
So if for some small enough time we have $\VV(\tau) = \Vert V^+(\tau) \Vert > 0$, then 
\[ \partial_\tau \VV^-(\tau) \leq - \frac1{2(n-k)} \VV^-(\tau) + C|\tau|^{-1} \VV^-(\tau) <  -C|\tau|^{-1} \VV^-(\tau) - C |\tau|^{-1} \VV^-(\tau) 
\leq \partial_\tau \Vert V^+(\tau) \Vert, \]
so $\VV^-(\tau') < \Vert V^+(\tau') \Vert$ for $\tau' > \tau$ close to $\tau$.
If follows that for $T \ll 0$ the set $\{ \VV^- \geq \Vert V^+ \Vert \} \cap (-\infty,T]$ must be a union of closed intervals without left endpoints, so it must be either empty or of the form $(-\infty, \ov\tau]$.
In the second case we would have $\VV^-(\tau) \geq \Vert V^+(\tau) \Vert$ for small enough $\tau$ \eqref{eq_dt_VVm_prop} implies that $\partial_\tau \VV^- \leq - \frac{1}{4(n-k)} \VV^-$ for small $\tau$, which contradicts the fact that $\limsup_{\tau \to -\infty} \VV^-(\tau) < \infty$.
Thus the first case holds, meaning that for small $\tau$ we must have
\[ \VV^-(\tau) \leq \Vert V^+(\tau) \Vert. \]
Setting $h(\tau) := \frac{\VV^-(\tau)}{\Vert V^+(\tau) \Vert} \leq 1$, we get for small $\tau$
\[ \partial_\tau h \leq \frac{-\frac1{2(n-k)} \VV^- + C|\tau|^{-1} \Vert V^+ \Vert}{\Vert V^+ \Vert} - h  \cdot \frac{-C|\tau|^{-1} \Vert V^+ \Vert - C |\tau|^{-1} \VV^-}{\Vert V^+ \Vert} \leq -\tfrac1{4(n-k)} h +C|\tau|^{-1}. \]
So for $\tau^* \leq \tau \ll 0$ we have since $h(\tau^*) \leq 1$
\begin{equation*}
 h(\tau) \leq \exp \big({ -\tfrac1{4(n-k)} (\tau-\tau^*) }\big) h(\tau^*) + C\int_{\tau^*}^{\tau} |\tau'|^{-1} \exp \big({ -\tfrac1{4(n-k)} (\tau-\tau') }\big) d\tau' 
\end{equation*}
Taking $\tau^* \to -\infty$ and using $h(\tau^*) \leq 1$, we get
\begin{equation*}
 h(\tau) \leq C\int_{-\infty}^{\tau} |\tau'|^{-1} \exp \big({ -\tfrac1{4(n-k)} (\tau-\tau') }\big) d\tau'   \leq C\int_{-\infty}^{\tau} |\tau|^{-1} \exp \big({ -\tfrac1{4(n-k)} (\tau-\tau') }\big) d\tau' \leq C|\tau|^{-1}.
\end{equation*}
which finishes the proof of the claim.
\end{proof}

Combining the claim with the bound \eqref{e:Prop_diff_V_c}, we find
\begin{equation} \label{eq_better_Vp_bound_evol}
 \Big\| \partial_\tau V^{+} - L V^{+} - 2 Q_2^{+} (\ov U^0,V^{+})  \Big\| \leq C |\tau|^{-2} \Vert V^{+}(\tau) \Vert_{L^2_f} . 
\end{equation}
Write
\[ V^+(\tau) = V_0(\tau) + V_1(\tau) + V_2(\tau) + V_3 (\tau) \in \sV_{0} \oplus \big( \sV_{\rot, \frac12, N} \oplus \sV_{\frac12, \Jac} \big) \oplus  \sV_{\rot, \frac12, N^\perp} \oplus \sV_{\rot, 1} \]
and define $I_j \subset (-\infty, \ov\tau]$, $j = 0,1,2,3$, to be the set of times $\tau$ for which $\Vert V_j (\tau) \Vert$ is maximal among $\Vert V_0 (\tau) \Vert, \ldots, \Vert V_4 (\tau) \Vert$.
Then the evolution inequalitiy \eqref{eq_better_Vp_bound_evol} implies, using Lemma~\ref{l:expansion_Q_2}, that
\begin{alignat}{3}
\partial_\tau V_0 &= && O\big( |\tau|^{-1}  \Vert V_0 \Vert  \big) \qquad\qquad &&\text{on $I_0$} \label{eq_evol_V0} \\
\partial_\tau V_1 &= \tfrac12 V_1 &+& O\big( |\tau|^{-2} \log|\tau|  \, \Vert V_1 \Vert  \big) \qquad\qquad&&\text{on $I_1$}  \label{eq_evol_V1} \\
\partial_\tau V_2 &= \tfrac12 V_2 + (-\tau)^{-1} V_2 &+& O\big( |\tau|^{-2} \log|\tau|  \, \Vert V_2 \Vert  \big) \qquad\qquad&&\text{on $I_2$} \\
\partial_\tau V_3 &=  V_3 &+& O\big( |\tau|^{-2} \log|\tau|  \, \Vert V_3 \Vert \big) \qquad\qquad&&\text{on $I_3$} \label{eq_evol_V3}
\end{alignat}

\begin{Claim}
There is a $j \in \{ -1, 0, \ldots, 3 \}$ such that $(-\infty,\ov\tau_j ] \subset I_j$ for some $\ov\tau_j \in \IR$.
\end{Claim}

\begin{proof}
If $\tau$ is sufficiently small and $\tau \in I_i \cap I_j$ for $i < j$, then \eqref{eq_evol_V0}--\eqref{eq_evol_V3} imply that we must have $\partial_\tau \Vert V_j (\tau) \Vert < \partial_\tau \Vert V_i (\tau) \Vert$.
So $\tau' \not\in I_i$ for $\tau' > \tau$ close to $\tau$ and $\tau' \not\in I_j$ for $\tau' < \tau$ close to $\tau$.
It follows that $j(\tau) := \max \{ j \;\; : \;\; \tau \in I_j \}$ is non-decreasing, so it must be constant for small enough $\tau$.
\end{proof}

\begin{Claim}
$j \neq 0$
\end{Claim}

\begin{proof}
If $j = 0$, then for small $\tau$ the identity \eqref{eq_evol_V0} implies
\[ \partial_\tau \log ( |\tau|^C \Vert V_0(\tau) \Vert  ) \leq 0  \qquad \Rightarrow \qquad \Vert V_0(\tau) \Vert \geq c |\tau|^{-C}. \]
However, if $\Qu(\MM^0) = \Qu(\MM^1)$, then this contradicts \cite[Proposition~\refx{Prop_same_Q_close}]{Bamler_Lai_PDE_ODI}.
\end{proof}

So the asymptotics of $V$ are governed by one of the other three equations \eqref{eq_evol_V1}--\eqref{eq_evol_V3}.
Setting
\[ V'_1(\tau) := e^{-\tau/2} V_1(\tau), \qquad V'_2 (\tau) := (-\tau) e^{-\tau/2} V_2(\tau), \qquad V'_3(\tau) := e^{-\tau} V_3(\tau), \]
these equations are equivalent to
\[  \Vert \partial_\tau V'_j \Vert \leq C |\tau|^{-2} \log|\tau|  \, \Vert V'_j \Vert. \]
It follows that $\Vert V'_j(\tau) \Vert$ is uniformly bounded as $\tau \to -\infty$, implying 
\[  \Vert \partial_\tau V'_j \Vert \leq C |\tau|^{-2} \log|\tau| . \]
This implies that $\lim_{\tau \to -\infty} V'_j(\tau)$ exists.
Moreover, if the limit vanishes, then $V'_j (\tau) = 0$ and hence $V(\tau) = 0$ for $\tau \ll 0$, which implies that $\MM^0$ and $\MM^1$ must agree by Proposition~\ref{Prop_diff_V}\ref{Prop_diff_V_c}.

Claim~\ref{Cl_VpCtaum1Vp} combined with the convergence of $V'_j$ implies that we have local $L^2$-convergence of $e^{-\tau/2} v_\tau$, $(-\tau) e^{-\tau/2} v_\tau$ or $e^{-\tau} v_\tau$.
Together with the local $C^{10}$-bounds of Proposition~\ref{Prop_diff_V}\ref{Prop_diff_V_d}, this implies local $C^9$-convergence.
\end{proof}

\begin{Lemma}\label{l:implicit}
Let $u_0$ be a smooth function over an open product domain $D=\IB^n_{R}(\bO)\times\IS^{n-k}\subset M_{\cyl}$ for some $R>0$, and assume $\|u_0\|_{C^{1}(D)}\le\eta<0.1$. Let $\mathbf v\in \IR^{n+k}$ be a vector orthogonal to the axis of $M_{\cyl}$ with Jacobi field $V=\Jac(\mathbf v)$, where we view $\mathbf v$ as a constant vector field on $\IR^{n+1}$.
Assume $M_i = \Gamma_{\cyl}(u_i)$, $i =0,1$, is the normal graph of $u_i$ over $M_{\cyl}$ such that $M_1=M_0+\mathbf v$.
Then
\[\|u_1-u_0-V \|_{C^0( D)}\le C|\mathbf v|^2+C\eta|\mathbf v|.\] 
\end{Lemma}

\begin{proof}
Since $M_1 = M_0 + \mathbf v$, for any $(\bx,\by) \in D$ there is a $\by' \in \IS^{n-k}$ such that
\[ (u_1(\bx,\by) + 1) \by = (u_0(\bx,\by') + 1) \by' + \mathbf v. \]
It follows that $|\by - \by'| \leq C |\mathbf v |$ and therefore
\begin{equation} \label{eq_u0u0_diff}
 |u_0(\bx,\by) - u_0(\bx,\by') | \leq C \eta |\mathbf v |. 
\end{equation}
Taking the scalar product with $\by$ and noting that $|\by| = |\by'|$ implies
\begin{equation} \label{eq_u1u0y}
 (u_1(\bx,\by) - u_0(\bx,\by')) |\by|^2 = (\by' - \by) \cdot \by + \mathbf v \cdot \by. 
\end{equation}
Note that $\mathbf v \cdot \by = (\Jac(\mathbf v) )(\bx,\by) |\by|^2$ and the third term can be bounded as follows
\[ \big| (\by'-\by) \cdot \by \big| = \big| \by' \cdot \by - |\by|^2 \big| = \big| \by' \cdot \by - \tfrac12 |\by|^2 - \tfrac12 |\by'|^2 \big| = \tfrac12 |\by-\by'|^2 \leq C|\mathbf v|^2. \]
Combining this with \eqref{eq_u1u0y} and \eqref{eq_u0u0_diff} implies the desired bound.
\end{proof}
\bigskip

\begin{proof}[Proof of Proposition \ref{prop:actions}]
Assertion~\ref{prop:actions_c} is clear.
Assertion~\ref{prop:actions_g} follows by additivity of $V^+(\tau)$ and using Proposition~\ref{Prop_Vpp_asymp}.
Specifically, if $V^+_{i,j}(\tau)$ denotes the unstable mode for the pair $(\MM^i, \MM^j)$, then $V^+_{0,2} (\tau) + V^+_{2,1} (\tau) + V^+_{1,0} (\tau) = 0$.
So by Proposition~\ref{Prop_Vpp_asymp}\ref{p:V_1/2_null} the limits $\lim_{\tau \to -\infty} e^{-\tau/2} V^+_{i,j}(\tau)$ must exist for all three functions and must satisfy the same additivity relation.
If all these limits vanish, then we can apply Part~\ref{p:V_1/2_non_null} of the same proposition and obtain the same statement for the limits $\lim_{\tau \to -\infty} (-\tau) e^{-\tau/2} V^+_{i,j}(\tau)$.
If these limits vanish as well, then repeating our argument using Part~\ref{p:V_1} implies the same statements for the limits $\lim_{\tau \to -\infty} e^{-\tau} V^+_{i,j}(\tau)$.

To show Assertions~\ref{prop:actions_d}, \ref{prop:actions_e}, \ref{f:time-translate} and \ref{prop:actions_gg} suppose without loss of generality that $N = \IR^l \times \bO^{k-l}$.
Consider the functions $u_\tau^0, u_\tau^1$ and $u_\tau^{1,\prime}$ from Proposition~\ref{Prop_diff_V}, which express larger and larger parts of the rescaled flows $\td\MM^{0,\reg}_\tau$, $\td\MM^{1,\reg}_\tau$ and $\td\MM^{1,\prime,\reg}_\tau$ as graphs over the cylinder.
By \cite[Proposition~\refx{Prop_dom_qu_asymp}]{Bamler_Lai_PDE_ODI} we know that
\begin{equation} \label{eq_u1toU0}
 (-\tau) u^{1}_\tau \; \xrightarrow[\quad \tau \to -\infty \quad ]{C^{10}_{\loc}}  \;  -\frac1{\sqrt{2}} \sum_{i=1}^l \mathfrak p^{(2)}_{ii} =: U_0. 
\end{equation}

Set $h_1(\tau) := e^{-\tau/2}$, $h_2(\tau) := (-\tau) e^{-\tau/2}$ and $h_3 (\tau) :=  e^{-\tau}$.
By Proposition~\ref{Prop_Vpp_asymp} we know that for some $i, i' \in \{ 1,2,3 \}$,
\begin{align}
 h_i (\tau) \big( u^1_\tau - u^0_\tau \big) \; \xrightarrow[\quad \tau \to -\infty \quad ]{C^{9}_{\loc}}  \; V^+_{-\infty} (\MM^0, \MM^1),  \label{eq_hv_V1} \\
 h_{i'} (\tau) \big( u^{1,\prime}_\tau - u^0_\tau \big) \; \xrightarrow[\quad \tau \to -\infty \quad ]{C^{9}_{\loc}}  \; V^+_{-\infty} (\MM^0, \MM^{1,\prime}). \label{eq_hv_V1p} 
\end{align}
Recall that $i$ and $i'$ depend on whether the right-hand side is contained in $ \sV_{\rot, \frac12, N} \oplus \sV_{\frac12,\Jac} \setminus \{ 0 \}$, $\sV_{\rot, \frac12,N^\perp} \setminus \{ 0 \}$ or $\sV_{\rot, 1}$.

In the setting of Assertion~\ref{prop:actions_d}, we can use Lemma~\ref{l:implicit} to deduce that
\[ h_1 (\tau) \big( u^{1,\prime}_\tau - u^1_\tau \big) \; \xrightarrow[\quad \tau \to -\infty \quad ]{C^{0}_{\loc}}  \; V'. \]
In the setting of Assertion~\ref{prop:actions_e} we have $u^{1,\prime}_\tau (\bx, \by) = u^1(\bx - e^{\tau/2} \mathbf v, \by)$, so due to \eqref{eq_u1toU0} we have
\[ h_2 (\tau) \big( u^{1,\prime}_\tau - u^1_\tau \big) \; \xrightarrow[\quad \tau \to -\infty \quad ]{C^{1}_{\loc}}  \; -\sum_{m=1}^k \frac{\partial U_0}{\partial x_m} v_m = \frac1{\sqrt 2} \sum_{m=1}^l \mathfrak p^{(1)}_m v_m. \]
In the setting of Assertion~\ref{f:time-translate} we have for $ - e^{-\tau'} = \Delta T-e^{-\tau}$
\[ e^{-\tau'/2} \big(1 + u^{1,\prime}_{\tau'}(\bx, \by) \big) = e^{-\tau/2} \big(1 + u^{1}_{\tau}(e^{(\tau-\tau')/2}\bx, \by) \big). \]
View $\tau'(\tau)$ as a function in $\tau$ and note that we have the following asymptotics as $\tau \to -\infty$
\[ \tau'(\tau) = -\log(e^{-\tau} - \Delta T ) = \tau - \log ( 1 - (\Delta T) e^{\tau} ) = \tau + (\Delta T) e^{\tau} + o(e^{\tau}). \]
So using \eqref{eq_u1toU0}, we obtain that
\[ u^{1}_{\tau}(e^{(\tau-\tau')/2}\bx, \by) - u^{1}_{\tau}(\bx, \by) = o(e^{\tau}). \]
Similarly, since $u^{1,\prime}_{\tau}(\bx,\by)$ is uniformly Lipschitz in time for fixed $(\bx,\by)$, we find
\[ u^{1,\prime}_{\tau'(\tau)}(\bx, \by) - u^{1,\prime}_{\tau}(\bx, \by) = o(e^\tau). \]
It follows that
\begin{equation*} \label{eq_h3upu}
 h_3 (\tau) \big( u^{1,\prime}_{\tau} - u^1_\tau \big) \; \xrightarrow[\quad \tau \to -\infty \quad ]{C^{1}_{\loc}}  \; \frac12 \Delta T \, \mathfrak p^{(0)}. 
\end{equation*}
Combining these identities with \eqref{eq_hv_V1} and \eqref{eq_hv_V1p} implies the desired identities for $V^+_{-\infty} (\MM^0, \MM^1)$ and $V^+_{-\infty} (\MM^0, \MM^{1,\prime})$ from Assertions~\ref{prop:actions_d}, \ref{prop:actions_e}, \ref{f:time-translate} and \ref{prop:actions_gg}.

In the setting of Assertion~\ref{prop:actions_gg} we have $i = 1$ and we also get from Proposition~\ref{Prop_Vpp_asymp} that for some $j \in \{ 1,2,3 \}$
\[   h_{j} (\tau) \big( u^{1,\prime}_\tau - u^1_\tau \big) \; \xrightarrow[\quad \tau \to -\infty \quad ]{C^{9}_{\loc}}  \; V^+_{-\infty} ( \MM^{1,\prime}, \MM^1) \]
Since $u^0(\bx - e^{\tau/2} \mathbf v, \by ) = u^0(\bx , \by )$,  the difference on the left-hand side is equal to
\[  h_j(\tau) \big( \big( u^{1}_\tau - u^0_\tau \big)(\bx - e^{\tau/2} \mathbf v, \by ) - \big(u^{1}_\tau - u^0_\tau \big) (\bx , \by ) \big). \]
Due to \eqref{eq_hv_V1}, this converges in $C^0_{\loc}$ to $0$ if $j \in \{ 1,2 \}$ and to the right-hand side of \eqref{eq_VpM11pp0} if $j = 3$.
\end{proof}
\bigskip

\section{Proofs of the main results I} \label{sec_proofs_I}
In this section we prove the main results from Subsection~\ref{subsec_main_results_I}.
We will frequently use the quantity $V^+_{-\infty}(\MM_0, \MM_1) \in \sV_{> 0}$ from Section~\ref{sec_diff}.
Recall that this quantity is only defined if $\MM^0$ is convex and rotationally symmetric, a property which we will establish a posteriori for all flows (modulo a translation in space).
Fix $0 \leq k < n$ and recall the space
\[ \MCF^{n,k}_0 := \MCF^{n,k}_{\Oval} \cup \MCF^{n,k}_{\soliton} \]
and the map 
\begin{equation} \label{eq_Qb_map_repeat}
  (\Qu, \mathbf b) : \MCF^{n,k}_0 \lto \{ (\Qu', \mathbf b') \in \IR^{k \times k}_{\geq 0} \times \IR^k \;\; : \;\; \mathbf b' \in \nullspace(\Qu') \}  
\end{equation}
from Subsection~\ref{subsec_main_results_I}.
In the case $k = 0$ all asymptotically $(n,0)$-cylindrical mean curvature flows must be round shrinking spheres by \cite{Huisken_1984} (though, strictly speaking, our methods apply to this case as well) and the results from Subsection~\ref{subsec_main_results_I} hold trivially.
So we may assume henceforth that $k \geq 1$.

\subsection{Classification in the convex, rotationally symmetric case}
We begin by classifying asymptotically $(n,k)$-cylindrical flows in the convex, rotationally symmetric case, as in this setting our methods apply more directly.
Combined with the existence result in Subsection~\ref{subsec_FW_construction}, these flows provide the necessary tools and comparison flows for the arguments in Subsection~\ref{subsec_classification}, which ultimately yields the full classification.

We begin with the following consequence of the previous section.

\begin{Lemma} \label{Lem_consequence_Vp_infty}
Suppose that $\MM^0, \MM^1 \in \MCF^{n,k}_0$ and that $\MM^{1, \prime}$ is an asymptotically $(n,k)$-cylindrical flow with
\[ \Qu(\MM^0) = \Qu(\MM^1) = \Qu(\MM^{1,\prime}), \qquad \PP_{\sV_{\rot, \frac12, N}} \big(V^{+}_{-\infty} (\MM^0, \MM^1) \big) = \PP_{\sV_{\rot, \frac12, N}} \big(V^{+}_{-\infty} (\MM^0, \MM^{1,\prime}) \big), \]
where $N$ is the nullspace of $\Qu(\MM^0)$.
Then there is a vector $(\mathbf v, \Delta T) \in \IR^{n+1} \times \IR$ such that $\MM^1$ and $\MM^{1,\prime} + (\mathbf v, \Delta T)$ agree at all times at which both flows are defined.
\end{Lemma}

\begin{proof}
By Proposition~\ref{prop:actions}\ref{prop:actions_g} we have
\[ \PP_{\sV_{\rot, \frac12, N}} \big(V^{+}_{-\infty} (\MM^1, \MM^{1,\prime}) \big) = 0. \]
By Proposition~\ref{prop:actions}\ref{prop:actions_d} we can find a vector $\mathbf v' \in \bO^k \times \IR^{n-k+1}$ such that
\[ \PP_{\sV_{\rot, \frac12, N} \oplus \sV_{\Jac, \frac12}} \big(V^{+}_{-\infty} (\MM^1, \MM^{1,\prime} + (\mathbf v', 0)) \big) = 0. \]
We can therefore apply the first and second part of Proposition~\ref{prop:actions}\ref{prop:actions_e} to $\MM^1$ and $\MM^{1,\prime} + (\mathbf v', 0)$ and obtain that there is a vector $\mathbf v'' \in N^\perp$ such that
\begin{align*}
 \PP_{\sV_{\rot, \frac12, N} \oplus \sV_{\Jac, \frac12}} \big(V^{+}_{-\infty} (\MM^1, \MM^{1,\prime} + (\mathbf v' + \mathbf v'', 0)) \big) &= 0, \\
 \PP_{\sV_{\rot, \frac12, N^\perp} } \big(V^{+}_{-\infty} (\MM^1, \MM^{1,\prime} + (\mathbf v' + \mathbf v'', 0)) \big) &= 0. 
\end{align*}
Likewise, we can apply Proposition~\ref{prop:actions}\ref{f:time-translate} to show that there is a $\Delta T \in \IR$ such that
\begin{align*}
 \PP_{\sV_{\rot, \frac12, N} \oplus \sV_{\Jac, \frac12}} \big(V^{+}_{-\infty} (\MM^1, \MM^{1,\prime} + (\mathbf v' + \mathbf v'', \Delta T)) \big) &= 0, \\
 \PP_{\sV_{\rot, \frac12, N^\perp} } \big(V^{+}_{-\infty} (\MM^1, \MM^{1,\prime} + (\mathbf v' + \mathbf v'', \Delta T)) \big) &= 0, \\
  \PP_{\sV_{1} } \big(V^{+}_{-\infty} (\MM^1, \MM^{1,\prime} + (\mathbf v' + \mathbf v'', \Delta T)) \big) &= 0. 
\end{align*}
So for $\mathbf v := \mathbf v' + \mathbf v''$ we have $V^{+}_{-\infty} (\MM^1, \MM^{1,\prime} + (\mathbf v, \Delta T)) = 0$ and the lemma follows from Corollary~\ref{cor:unique_V}.
\end{proof}

We can now show Theorem~\ref{Thm_classification_Q}.

\begin{Lemma} \label{Lem_compare_flows_Qsv}
Theorem~\ref{Thm_classification_Q} is true.
Moreover, if $\MM^0 \in \MCF^{n,k}_{\Oval}$ and $\MM^1$ is an asymptotically $(n,k)$-cylindrical mean curvature flow with 
\begin{equation} \label{eq_conditionQsvN}
 \Qu(\MM^0) = \Qu(\MM^1), \qquad \PP_{\sV_{\rot, \frac12, N}} \big(V^{+}_{-\infty} (\MM^0, \MM^1) \big) = 0, 
\end{equation}
where $N$ is the nullspace of $\Qu(\MM^0)$.
then there is a vector $(\mathbf v, \Delta T) \in \IR^{n+1} \times \IR$ such that $\MM^0$ and $\MM^{1} + (\mathbf v, \Delta T)$ agree at all times at which both flows are defined.
\end{Lemma}

\begin{proof}
The last statement is a direct consequence of Lemma~\ref{Lem_consequence_Vp_infty} if we set $\MM^0 \leftarrow \MM^0$, $\MM^{1} \leftarrow \MM^0$ and $\MM^{1,\prime} \leftarrow \MM^1$.
So it remains to prove Theorem~\ref{Thm_classification_Q}.

We have shown in \cite[Theorem~\refx{Thm_existence_oval}]{Bamler_Lai_PDE_ODI} that $\Qu |_{\MCF^{n,k}_{\Oval}}$ is surjective.
To see injectivity, consider two flows $\MM^0, \MM^1 \in \MCF^{n,k}_{\Oval}$ with $\Qu(\MM^0) = \Qu(\MM^1)$.
Since both flows are invariant under reflection about the origin, we must have 
$$\PP_{\sV_{\rot, \frac12, N}} \big(V^{+}_{-\infty} (\MM^0, \MM^1) \big) = 0.$$ 
So since both flows are defined on a maximal time-interval we obtain from the last statement that $\MM^0 = \MM^{1} + (\mathbf v, \Delta T)$ for some $(\mathbf v, \Delta T) \in \IR^{n+1} \times \IR$.
Since both flows go extinct at time $0$, we must have $\Delta T$.
If $\mathbf v \neq \bO$, then $\MM^0$ is invariant under reflections about both $\bO$ and $\mathbf v$, so by convexity, it must split off a line parallel to $\mathbf v$.
This implies that $\MM^0 = \MM^0 - (\mathbf v, 0) = \MM^1$.

Continuity of $\Qu$ was established in \cite[Proposition~\refx{Prop_Q_continuous}]{Bamler_Lai_PDE_ODI}.
Now suppose that $\MM^i \in \MCF^{n,k}_{\Oval}$, $i \leq \infty$, with $\Qu(\MM^i) \to \Qu(\MM^\infty)$.
We need to show that every subsequence of $\MM^i$ subsequentially converges to $\MM^\infty$.
For any such subsequence, we can extract another subsequence, using \cite[Proposition~\refx{Prop_Q_continuous}]{Bamler_Lai_PDE_ODI}, such that $\MM^i \to \MM^{\infty, \prime}$, where the limit is either asymptotically $(n,k)$-cylindrical, an affine plane or empty.
Since all flows $\MM^i$ go extinct at time $0$, the last two cases cannot occur and it is clear that $\MM^{\infty, \prime} \in \MCF^{n,k}_{\Oval}$ with $\Qu(\MM^{\infty, \prime}) = \Qu(\MM^\infty)$.
So $\MM^{\infty, \prime} = \MM^\infty$ due to injectivity.
This concludes the proof of Assertion~\ref{Thm_classification_Q_a}.

Assertion~\ref{Thm_classification_Q_b} is a direct consequence of injectivity and equivariance of $\Qu$.
Assertion~\ref{Thm_classification_Q_c} follows from injectivity and the fact that we can construct flows isometric to $\IR^l \times \MM'' \in \MCF^{n,k}_{\Oval}$ with for the desired value of $\Qu$.
Assertion~\ref{Thm_classification_Q_d} follows from Assertion~\ref{Thm_classification_Q_c}. 
\end{proof}
\medskip

Next, we consider the more general case and relate the vector $\mathbf b(\MM^1)$ with a component of $V_{-\infty}^+(\MM^0, \MM^1)$.
The following theorem is the precise version of Theorem~\ref{Thm_def_b_vague} from the introduction (stated in a slightly more general form).

\begin{Theorem} \label{Thm_Vp_mode_precise}
Suppose that $\MM^0 \in \MCF^{n,k}_{\Oval}$ and $\MM^1$ is an asymptotically $(n,k)$-cylindrical flow that is convex and rotationally symmetric  and assume
\[ \Qu(\MM^0) = \Qu(\MM^1), \qquad \PP_{\sV_{\rot, \frac12, N}} \big(V^{+}_{-\infty} (\MM^0, \MM^1) \big) = \frac1{\sqrt{2}} \sum_{i=1}^k b'_i \mathfrak p^{(1)}_i , \]
where $N$ is the nullspace of $\Qu(\MM^0)$.
If $\mathbf b' := \sum_{i=1}^k b'_i \mathbf e_i \neq \bO$, then $\MM^1$ is a translating soliton with velocity vector $\mathbf v :=  |\mathbf b'|^{-2} \mathbf b'$. 
Moreover, if $\MM^1 \in \MCF^{n,k}_{0}$, then $\mathbf b(\MM^1) = \mathbf b' \in N$.
So the map \eqref{eq_Qb_map_repeat} is well-defined.
\end{Theorem}

\begin{proof}
If $\mathbf b' = \bO$ and $\MM^1 \in \MCF^{n,k}_0$, then Lemma~\ref{Lem_compare_flows_Qsv} implies that $\MM^1$ is the restriction of a translation of $\MM^0$, so $\MM^1$ cannot be a translating soliton and thus $\mathbf b (\MM^1) = \bO$.
So assume for the remainder of the argument that $\mathbf b' \neq \bO$.

By symmetry we have
\[ \PP_{ \sV_{\Jac, \frac12}} \big(V^{+}_{-\infty} (\MM^0, \MM^{1} ) \big) = 0, \]
so $V^{+}_{-\infty} (\MM^0, \MM^{1})  \in \sV_{\rot, \frac12, N}$.
Let $s \in \IR$ be arbitrary and set $\mathbf v' := s \mathbf v \in N$.
By Theorem~\ref{Thm_classification_Q} we have $\MM^0 = \MM^0 + (\mathbf v, 0)$.
So by Proposition~\ref{prop:actions}\ref{prop:actions_gg}
\[ V_{-\infty}^+ \big(\MM^1 , \MM^1 + (\mathbf v', 0) \big) = \tfrac1{\sqrt 2} (\mathbf v' \cdot \mathbf b') \mathfrak p^{(0)} \]
and by Proposition~\ref{prop:actions}\ref{f:time-translate}
\[ V_{-\infty}^+ \big(\MM^1 + (\mathbf v', 0 ), \MM^1 + (\mathbf v', 0 ) + (\bO,  \mathbf v' \cdot \mathbf b) \big) = - \tfrac1{\sqrt 2}(\mathbf v' \cdot \mathbf b) \mathfrak p^{(0)}. \]
So by Proposition~\ref{prop:actions}\ref{prop:actions_g}
\[ V_{-\infty}^+ \big(\MM^1 , \MM^1 + (\mathbf v', 0 ) + (\bO, \mathbf v' \cdot \mathbf b) \big) = 0 \]
Hence by Corollary~\ref{cor:unique_V} the flows $\MM^1$ and $\MM^1 + (\mathbf v',  \mathbf v' \cdot \mathbf b) = \MM^1 + s(\mathbf v, 1)$   agree for all times at which they are defined.
This shows that $\MM^1$ is a translating soliton with velocity $\mathbf v$.

For the last statement note that after applying a rotation we can write $\MM^1 = \IR^l \times \MM'$ for some maximal $l \in \{ 0, \ldots, k \}$.
By definition $\MM^1$ also has velocity vector $|\mathbf b(\MM^1)|^{-2} \mathbf b(\MM^1)$.
So if $\mathbf b(\MM^1) \neq \mathbf b'$, then $\MM^1$ has two distinct velocity vectors and hence must split off a line parallel to the difference of both vectors.
However, both vectors $\mathbf b(\MM^1)$ and $\mathbf b'$ are orthogonal to the $\IR^l$-factor, in contradiction to the maximal choice of $l$.
\end{proof}
\bigskip

\subsection{Construction of flying wing solitons} \label{subsec_FW_construction}
In this subsection we present an alternative argument to the construction in Hoffman-Ilmanen-Mart{\'\i}n-White \cite{HoffmanIlmanenMartinWhite2019}.

\begin{Proposition}\label{l:FW_construction}
Let $\Qu' \in \IR^{k \times k}_{\geq 0}$ be a diagonal matrix whose first diagonal entry is zero and let $H' \neq 0$.
Then there is an asymptotically $(n,k)$-cylindrical mean curvature flow $\MM$ that is a translating soliton such that
\[ \Qu(\MM) = \Qu', \qquad (\bO,0) \in \spt \MM, \qquad \mathbf H(\bO,0) = H' \mathbf e_1. \]
Moreover, $\MM$ is smooth, non-collapsed, convex, rotationally symmetric, invariant under reflections perpendicular to all coordinate axes except for the first and has uniformly bounded second fundamental form.
Hence, the map \eqref{eq_Qb_map_repeat} is surjective.
\end{Proposition}

We will use the following well-known fact:

\begin{Lemma} \label{Lem_convex_compactness}
Let $\MM^i$ be a sequence of non-collapsed, convex asymptotically $(n,k)$-cylindrical flows in $\IR^{n+1} \times (-\infty, T_i)$, for $T_i \to \infty$.
Suppose that $(\bO,0) \in \MM^{i, \reg}$ and that $|\mathbf H^{\MM^i}|(\bO,0)$ is uniformly bounded.
Then, after passing to a subsequence, we have convergence $\MM^i \to \MM^\infty$ in the Brakke flow, where $\MM^\infty$ is a non-collapsed and convex mean curvature flow that does not go extinct on or before time $0$.
Moreover, if $\Vert \Qu(\MM^i) \Vert$ is uniformly bounded and $|\mathbf H^{\MM^i}|(\bO,0) > c$ for some uniform $c > 0$, then $\MM^\infty$ is also asymptotically $(n,k)$-cylindrical.
\end{Lemma}

\begin{proof}
The fact that the limit does not go extinct on or before time $0$ follows from the non-collapsedness and the uniform bound on the mean curvature via \cite{Haslhofer_Kleiner}.
Alternatively, we may argue as follows.
For each $i$ let $r_i$ be the supremum over all radii such that the two spheres of radius $r$ tangent to $(\spt \MM^i)_0$ at $(\bO,0)$ only intersect $(\spt \MM^i)_0$ in the origin.
If $r_i > c' > 0$ for a subsequence, then the non-extinction follows from \cite{Sheng_Wang_09}.
Now suppose by contradiction that $r_i \to 0$.
Then again by \cite{Sheng_Wang_09} the parabolically rescaled flows $r_i^{-1} \MM^i$ subsequentially converge to a convex limit $\MM^{\infty,\prime}$ that does not go extinct at or before time $0$.
Since the convergence is smooth at time $0$, its mean curvature at the origin must vanish, so $\MM^{\infty,\prime}$ must contain a constant plane passing through the origin.
However, by the choice of $r_i$, the limit must contain another component.
So since it is convex, it must be a union of two parallel affine planes, which contradicts the fact that $\Theta^{\MM_i} < \Theta_{\IR^k \times \IS^{n-k}} < 2$.

The last statement of the lemma is a direct consequence of \cite[Proposition~\refx{Prop_Q_continuous}]{Bamler_Lai_PDE_ODI}.
\end{proof}
\medskip

\begin{proof}[Proof of Proposition~\ref{l:FW_construction}.]
Without loss of generality we may assume that $H' > 0$.
Fix a sequence of positive definite diagonal matrices $\Qu'_i \in \IR^{k \times k}_{\geq 0}$ with $\Qu'_i \to \Qu'$.
By Theorem~\ref{Thm_classification_Q} there are $\MM^i \in \MCF^{n,k}_{\Oval}$ such that $\Qu(\MM^i) = \Qu'_i$ and each $\MM^i$ is invariant under reflections across the coordinate hyperplanes.
For each $i$ and time $t < 0$ choose the unique point $\bp_{i,t} = p_{i,t} \mathbf e_1 \in (\spt \MM^{i})_t$ with $p_{i,t} < 0$.
Note that by symmetry the mean curvature at this point must be of the form $\mathbf H^{\MM_i}(\bp_{i,t}, t) = H_{i,t} \mathbf e_1$ for $H_{i,t} > 0$.

\begin{Claim}
For each $i$ we have $\lim_{t \to -\infty} H_{i,t} = 0$ and $\lim_{t \to 0} H_{i,t} = \infty$.
\end{Claim}

\begin{proof}
The second limit is clear, because the flow $\MM^i$ develops a spherical singularity at $(\bO,0)$.
To see the first limit, assume by contradiction that $H_{i,t_j} > c > 0$ for some $t_j \to -\infty$.
Consider the flows $\MM^{i,j} := \MM^i - (\bp_{i,t_j}, t_j)$, whose supports contain $(\bO,0)$.
Lemma~\ref{Lem_convex_compactness} implies subsequential convergence $\MM^{i,j} \to \MM^{i,\infty}$.
The limit cannot be empty or an affine plane, so it must be asymptotically $(n,k)$-cylindrical with $\Qu(\MM^{i,\infty}) = \Qu'_i$ by \cite[Proposition~\refx{Prop_Q_continuous}]{Bamler_Lai_PDE_ODI}.
We can apply Lemma~\ref{Lem_compare_flows_Qsv} for $\MM^i$ and $\MM^{i,\infty}$; the second condition in \eqref{eq_conditionQsvN} is vacuously true as $\Qu'_i$ is positive definite.
We obtain that $\MM^{i,\infty}$ must be a translation of $\MM^i$.
However, this is impossible, because by construction, its time-slices must be non-compact.
\end{proof}

By continuity, we can pick a time $t_i < 0$ for each $i$ such that $H_{i,t_i} = H'$.
By Lemma~\ref{Lem_convex_compactness} the flows $\MM^i - (\bp_{i,t_j}, t_i)$ must subsequentially converge in the Brakke sense to a flow $\MM^\infty$ that is smooth at time $0$ and by \cite[Proposition~\refx{Prop_Q_continuous}]{Bamler_Lai_PDE_ODI} we must have $\Qu(\MM^\infty) = \Qu'$ and $\mathbf H^{\MM^\infty} (\bO,0) = H' \mathbf e_1$.
By Lemma~\ref{Lem_compare_flows_Qsv} and Theorem~\ref{Thm_Vp_mode_precise} applied to a flow of the form $\IR \times \MM'$ in $\MCF^{n,k}_{\Oval}$ and $\MM^\infty$, the flow $\MM^\infty$ must be either a translating soliton or a translation of $\IR \times \MM'$, where the $\IR$-factor is in the first coordinate direction.
The condition on the mean curvature vector at $(\bO,0)$ rules out the second possibility.

It remains to show that $\MM^\infty$ has uniformly bounded second fundamental form.
Suppose by contradiction that $|\mathbf H|(\bp_i, 0) \to \infty$ for some sequence $\bp_i \in (\spt \MM^\infty)_0$ and use again \cite[Proposition~\refx{Prop_Q_continuous}]{Bamler_Lai_PDE_ODI} to pass to a subsequence such that we have convergence $\MM^\infty - (\bp_i,0) \to \MM^{\infty,\prime}$ in the Brakke sense.
The limit must be asymptotically $(n,k)$-cylindrical and it must be singular at time $0$.
However, since $\MM^\infty$ is a translating soliton, the limit $\MM^{\infty,\prime}$ must also be a translating soliton, which is impossible.
\end{proof}
\bigskip

\subsection{Proof of the main classification result} \label{subsec_classification}
It suffices to prove Theorem~\ref{Thm_classification_precise} as it implies Theorem~\ref{Thm_classification_simple}.

\begin{proof}[Proof of Theorem~\ref{Thm_classification_precise}.]
Let $\MM$ be an asymptotically $(n,k)$-cylindrical mean curvature flow and choose $\MM^0 \in \MCF^{n,k}_{\Oval}$ such that $\Qu(\MM^0) = \Qu(\MM)$.
Let $N$ be the nullspace of $\Qu(\MM^0)$ and write
\[ \PP_{\sV_{\rot, \frac12, N}} \big(V^{+}_{-\infty} (\MM^0, \MM) \big) = \sum_{i=1}^{k'} b'_i \mathfrak p^{(1)}_i, \qquad \mathbf b' := \frac1{\sqrt 2}\sum_{i=1}^k b'_i \mathbf e_i \in N. \]
By Proposition~\ref{l:FW_construction}, we can find an $\MM^1 \in \MCF^{n,k}_0$ such that
\[ \Qu(\MM^1) = \Qu(\MM^0) = \Qu(\MM), \qquad \mathbf b(\MM^1) = \mathbf b', \]
which implies by Theorem~\ref{Thm_Vp_mode_precise} that
\[ \PP_{\sV_{\rot, \frac12, N}} \big(V^{+}_{-\infty} (\MM^0, \MM^1) \big) = \frac1{\sqrt 2} \sum_{i=1}^{k'} b'_i \mathfrak p^{(1)}_i = \PP_{\sV_{\rot, \frac12, N}} \big(V^{+}_{-\infty} (\MM^0, \MM) \big). \]
So by Lemma~\ref{Lem_consequence_Vp_infty} there is a vector $(\mathbf v, \Delta T) \in \IR^{n+1} \times \IR$ such that $\MM + (\mathbf v, \Delta T)$ is a restriction of $\MM^1$ to a possibly smaller time-interval.
\end{proof}
\bigskip

\subsection{Proof of the remaining results}

\begin{proof}[Proof of Theorem~\ref{Thm_classification_Qb}.]
To see injectivity of the map \eqref{eq_Qb_map}, consider two flows $\MM^1, \MM^{1,\prime} \in \MCF^{n,k}_0$ with $\Qu(\MM^1) = \Qu(\MM^{1,\prime})$ and $\mathbf b(\MM^1) = \mathbf b(\MM^{1,\prime})$.
Let $\MM^0 \in \MCF^{n,k}_{\Oval}$ be the unique flow with $\Qu(\MM^0) = \Qu(\MM^1) = \Qu(\MM^{1,\prime})$.
Theorem~\ref{Thm_Vp_mode_precise} applied to the pairs $(\MM^0,\MM^1)$ and $(\MM^0, \MM^{1,\prime})$ implies \eqref{eq_conditionQsvN}.
So, since both flows are defined on a maximal time-interval, Lemma~\ref{Lem_consequence_Vp_infty} yields that $\MM^1 = \MM^{1,\prime} + (\mathbf v, \Delta T)$ for some $(\mathbf v, \Delta T) \in \IR^{n+1} \times \IR$.
It follows that both flows must be either both be in $\MCF^{n,k}_{\Oval}$ or both be in $\MCF^{n,k}_{\soliton}$.
In the first case, we obtain $\MM^1 = \MM^{1,\prime}$ from Theorem~\ref{Thm_classification_Q}.
In the second case, we may assume that $\Delta T = 0$; so $\MM^1 = \MM^{1,\prime} + (\mathbf v, 0)$.
Suppose that $\mathbf v \neq \bO$, because otherwise we are done.
Since $\mathbf b(\MM^1) = \mathbf b(\MM^{1,\prime})$, the tangent spaces and mean curvature vectors of $\MM^{1,\reg}_0$ and $\MM^{1,\prime, \reg}_0$ at the origin must agree.
So by convexity, $\mathbf v$ must be contained in this tangent space and $\MM^1$ must split off a line parallel to $\mathbf v$.
Therefore, $\MM^1 = \MM^{1} - (\mathbf v, 0) = \MM^{1,\prime}$, as desired.

Continuity of $\Qu$ was established before and continuity of $\mathbf b$ is clear.
Now consider a sequence $\MM^i$, $i \leq \infty$, with $\Qu(\MM^i) \to \Qu(\MM^\infty)$ and $\mathbf b(\MM^i) \to \mathbf b(\MM^\infty)$.
We need to show again that every subsequence of $\MM^i$ subsequentially converges to $\MM^\infty$.
So pick a subsequence and use \cite[Proposition~\refx{Prop_Q_continuous}]{Bamler_Lai_PDE_ODI} to pass to another subsequence such that $\MM^i \to \MM^{\infty,\prime}$, where the limit is either asymptotically $(n,k)$-cylindrical, an affine plane or empty.
The last case cannot occur, because all flows contain $(\bO,0)$ in their support.
The second last case also cannot occur, because otherwise $|\mathbf H^{\MM^i}|(\bO,0) \to 0$, which would imply $|\mathbf b(\MM^i)| \to \infty$.
So $\MM^{\infty, \prime}$ is asymptotically $(n,k)$-cylindrical and by continuity $\Qu(\MM^{\infty, \prime}) = \Qu(\MM^\infty)$.

If $\mathbf b(\MM^\infty) \neq \bO$, then $|\mathbf b(\MM^i)| \not\to 0$, so $|\mathbf H^{\MM^i}|(\bO,0)$ remains bounded and we can use Lemma~\ref{Lem_convex_compactness} to show that we have smooth convergence at time $0$.
Since in this case $\MM^i$ are translating solitons with uniformly controlled speed, for $i$ sufficiently large, we get $\MM^{\infty, \prime} \in \MCF^{n,k}_{\soliton}$ with $\mathbf b(\MM^{\infty,\prime}) = \mathbf b(\MM^\infty)$, so $\MM^{\infty, \prime}=\MM^\infty$ by injectivity.

If $\mathbf b(\MM^\infty) = \bO$, then $|\mathbf H^{\MM^i}|(\bO,0) \to \infty$, so the limit $\MM^{\infty,\prime}$ is not smooth near $(\bO, 0)$.
It follows from Theorem~\ref{Thm_classification_precise} that $\MM^{\infty, \prime} = \MM'' + (\mathbf v , \Delta T)$ for some $\MM'' \in \MCF^{n,k}_0$ and $(\mathbf v , \Delta T) \in \IR^{n+1} \times \IR$.
Since $\MM^{\infty,\prime}$ is singular at $(\bO, 0)$ we must have $\MM'' \in \MCF^{n,k}_{\Oval}$ and $\Delta T = 0$ and $\mathbf v$ must be contained in the nullspace of $\Qu(\MM'')$.
This implies that $\MM'' = \MM'' -  (\mathbf v ,  0)= \MM^{\infty, \prime}$, as desired.
So \eqref{eq_Qb_map} is indeed a homeomorphism.

Assertion~\ref{Thm_classification_Qb_a} is clear and Assertion~\ref{Thm_classification_Qb_b} is a restatement of \cite[Proposition~\refx{Prop_Q_basic_properties}]{Bamler_Lai_PDE_ODI}.
Assertion~\ref{Thm_classification_Qb_c} follows again by injectivity since we can construct a flow of product form with the desired values of $\Qu$ and $\mathbf b$.
\end{proof}
\medskip

\begin{proof}[Proof of Theorem~\ref{Thm_limit_is_flying_wing}.]
This is a direct consequence of the proof of Proposition~\ref{l:FW_construction}.
\end{proof}
\medskip

The proof of Theorem~\ref{Thm_compactness_ancient} is standard.
It also follows from Lemma~\ref{Lem_blow_ups} in the next section.

\bigskip

\section{Proofs of the main results II} \label{sec_proofs_II}
We need the following lemma.

\begin{Lemma} \label{Lem_blow_ups}
Let $\MM_i$ be a sequence of $n$-dimensional, unit-regular integral Brakke flows in $\IR^{n+1} \times I_i$.
Suppose that for some sequence $r_i \geq 1$ we have convergence in the Brakke sense of the parabolic rescalings $r_i^{-1} \MM_i \to \MM_{\cyl}^{n,k}$, for some $k \in \{ 0, \ldots, n-1 \}$.
Then for a subsequence, we have convergence $\MM_i \to \MM_\infty$ in the Brakke sense, where the limit is one of the following:
\begin{itemize}
\item An asymptotically $(n,k')$-cylindrical flow with $k' \in \{ 0, \ldots, k \}$.
\item A constant, affine, multiplicity one plane.
\item An empty flow.
\end{itemize}
Moreover, if $(\bO,0) \in \MM_i^{\sing}$ for all $i$, then for large $i$ its tangent flow must be isometric to $\MM^{n,k'}_{\cyl}$ for some $k' \in \{ 0, \ldots, k \}$.
\end{Lemma}

\begin{proof}
Suppose by induction that the lemma is true for $k$ replaced with any number in $\{ 0, \ldots, k-1 \}$ (if $k =0$, then this assumption is vacuous).
Let $\delta > 0$ be a constant whose value we will determine later.

Choose $r'_i \in [0, r_i]$ minimal such that for all $r \in (r'_i, r_i]$, the origin is a center of an $(n,k, \delta)$-neck of $\MM_i$ at scale $r$ and at time $-r^2$.
By assumption we know that $r'_i / r_i \to 0$.
After passing to a subsequence, we may assume $r'_i \to r'_\infty \in [0,\infty]$.
Consider an arbitrary sequence $r''_i > c r'_i$ for a uniform $c > 0$.
Then for a subsequence we have $(r''_i)^{-1} \MM_i \to \MM'_\infty$ in the Brakke sense, where for all $r \geq 1$ the origin is a center of an $(n,k, 2\delta)$-neck of $\MM'_\infty$ at scale $r$ and at time $-r^2$.
If $\delta \leq \ov\delta$, then $\MM_\infty'$ is isometric to an asymptotically $(n,k)$-cylindrical flow (see \cite{colding_minicozzi_uniqueness_blowups} or \cite[Corollary~\refx{Cor_unique_tangent_infinity}]{Bamler_Lai_PDE_ODI}) and therefore falls under the classification of Theorem~\ref{Thm_classification_precise}.
So if $r'_\infty < \infty$, then we can choose $r''_i = 1$ and the first part of the lemma follows.
By the same reasoning, the second part of the lemma follows if $r'_i = 0$ for large $i$.
So after passing to a subsequence, it remains to consider the case $r'_i > 0$, take $r''_i := r'_i$ and consider the corresponding limit flow $\MM'_\infty$.
\medskip

\textit{Case 1: $(\bO,0) \not\in \spt \MM'_\infty$ or $(\bO, 0)$ is a regular point of $\MM'_\infty$. \quad}
In this case, the second part of the lemma is vacuous, so we may also assume that $r'_\infty = \infty$, because otherwise the first part is true.
In this case the convergence to $\MM'_\infty$ is locally smooth near $(\bO,0)$, so since $r'_i \to \infty$,  a subsequence of the original flows $\MM_i$ converge locally smoothly to an affine plane or empty flow.
\medskip

\textit{Case 2: $(\bO, 0)$ is a singular point of $\MM'_\infty$. \quad}
By the classification result from Theorem~\ref{Thm_classification_precise}, $\MM'_\infty$ must be isometric to $\MM_{\cyl}^{n,k}$ or the product of $\IR^{k''}$ times an $(n-k'')$-dimensional, compact ancient oval, for some $k'' \in \{ 0, \ldots, k-1 \}$.
The first case is impossible due to the choice of $r''_i$, so since the ancient oval goes extinct at a round $(n-k'')$-dimensional sphere, the tangent flow of $\MM'_\infty$ at $(\bO,0)$ must be isometric to $\MM_{\cyl}^{n,k''}$.
But this implies that {there is a sequence $r'''_i \in (0, r''_i]$} such that $(r'''_i)^{-1} \MM_i \to \MM''_\infty$ in the Brakke sense, where $\MM'''_\infty$ is isometric to $\MM_{\cyl}^{n,k''}$.
To show the first part of the lemma, we can again assume that $r'_i \to \infty$ and we can arrange that $r'''_i \to \infty$ as well.
The lemma now follows by induction.
\end{proof}
\bigskip

\begin{proof}[Proof of Theorem~\ref{Thm_MCN}.]
Fix $\eps > 0$.
Without loss of generality, we may assume that $(\bp_0, t_0)= (\bO,0)$ and $r_0 = 1$.

\begin{Claim}
For sufficiently small $\delta \leq \ov\delta(\eps)$, Assertion~\ref{Thm_MCN_b} holds and for every $\bp \in B(\bO, \eps^{-1})$ we have $(\bp,t) \in \spt \MM$ for at most one $t \in [0, \eps^{-1}]$.
\end{Claim}

\begin{proof}
Suppose that the claim was false.
Then we can find sequences of counterexamples $\MM_i$ for $\delta_i \to 0$ and points $(\bp_i ,t_i) \in (\spt \MM_i) \cap B(\bO, \eps^{-1}) \times [0,\eps^{-1}]$ such that one of the following is true for all $i$:
\begin{enumerate}
\item \label{en_case_1} $(\bp_i, t_i)$ is a singular point of $\MM_i$, but its tangent flow is not isometric to $\MM^{n,k'}_{\cyl}$ for any $k' \in \{ 0, \ldots, k \}$. 
\item \label{en_case_2} $(\bp_i, t_i)$ is a regular point of $\MM_i$, but does not have a strong $(\eps,k)$-canonical neighborhood.
In this case, choose $r_i > 0$ to be the supremum over all $r > 0$ such that for all $t \in [t_i - r^2, t_i] \cap [0, \eps^{-1}]$ the intersection $\MM^{\reg}_{i,t} \cap B(\bp_i, r)$ can be written as the local graph of a function over an affine plane with first derivatives bounded by $r^{-1}$ and second derivatives bounded by $r^{-2}$.
\item \label{en_case_3} There is a time $t'_i \in [0,t_i)$ such that $(\bp_i, t'_i) \in \spt \MM_i$.
In this case, choose $r_i$ as the maximum of the constant $r_i$ defined in Case~\ref{en_case_2} and $\sqrt{t_i - t'_i}$.
\end{enumerate}
After passing to a subsequence, we have $\MM_i \to \MM_\infty$ with initial condition isometric to $M^{n,k}_{\cyl}$.
By a standard uniqueness argument, we obtain that $\MM_\infty$ is isometric to $\MM^{n,k}_{\cyl}$.
So we have local smooth convergence $\MM_i \to \MM_\infty$ for all times, except the extinction time of $\MM_\infty$ and it is easy to see that $(\bp_i, t_i)$ must converge to the extinction locus of this flow and hence $r_i \to 0$ in Cases~\ref{en_case_1}, \ref{en_case_2}. 
Now Lemma~\ref{Lem_blow_ups} gives a contradiction to Case~\ref{en_case_1} for large $i$ and implies that for a subsequence we have convergence $\MM'_i:=r_i^{-1} ( \MM_i - (\bp_i, t_i)) \to \MM'_\infty$, where the limit is an asymptotically $(n,k')$-cylindrical flow for some $k' \leq k$, an affine plane or empty.
Since $(\bO,0) \in \spt \MM'_i$ in all cases, the limit cannot be empty and by the choice of $r_i$ we can also exclude the affine plane.
In Case~\ref{en_case_2}, we have smooth convergence near $(\bO,0)$ and the mean curvature in the limit does not vanish; so we obtain a contradiction for large $i$.
In Case~\ref{en_case_3}, we may assume that, after passing to a subsequence, $t''_i := r_i^{-2}  (t'_i - t_i) \to t''_\infty \in [-1, 0]$.
Then we must have $(\bO,0), (\bO, t''_\infty) \in \spt \MM'_\infty$.
If $t''_\infty \neq 0$, then this is impossible by the classification result, Theorem~\ref{Thm_classification_precise}.
If $t''_\infty = 0$, then for large $i$ the constant $r_i$ is defined according to the description from Case~\ref{en_case_2}.
In this case, we have smooth convergence of $\MM'_i \to \MM'_\infty$  near $(\bO,0)$, where the limit is smooth with positive mean curvature, and we have $(\bO, 0), (\bO, t''_i) \in \spt \MM'_i$, which is impossible for large $i$.
\end{proof}

The claim implies that we can express $\spt \MM \cap (B(\bO, \eps^{-1}) \times [0,\eps^{-1}])$ as the graph of some function $u : \UU' \to [0, \eps^{-1}]$ for some $\UU' \subset \UU$.
Since $\spt \MM$ is closed, the subset $\UU' \subset \UU$ must be closed and $u$ must be continuous.
We claim that $\UU' \subset \UU$ is also open.
Indeed, if $\bp \in \UU'$ and $(\bp, u(\bp) ) \in \MM^{\reg}$, then since $\mathbf H (\bp, u(\bp)) \neq 0$, we can represent in a neighborhood of $(\bp, u(\bp) )$ within $\MM^{\reg}$ as a graph of a smooth function over a neighborhood of $\bp$.
On the other hand, if  $(\bp, u(\bp) ) \in \MM^{\sing}$, then its tangent flow is a cylinder, so $\bp$ must be an interior point in $\UU'$.
So since $\UU'$ is non-empty, we have $\UU' = \UU$.

The last statement of Assertion~\ref{Thm_MCN_a} follows from the fact that $\MM^{\sing} \cap B(\bO, \eps^{-1}) \times [0, \eps^{-1}]$ has Hausdorff-dimension $\leq n-1$, which follows, for example via \cite[Theorem~1.4]{Cheeger_Haslhofer_Naber_13}.
See the proof of \cite[Theorem~1.9(a)]{Bamler_Kleiner_mult1} for further details.
\end{proof}
\bigskip

\begin{proof}[Proof of Theorem~\ref{Thm_MCN_basic}.]
This is a direct consequence of Theorem~\ref{Thm_MCN}.
\end{proof}
\bigskip

\begin{proof}[Proof of Theorem~\ref{Thm_blow_ups}.]
This is a direct consequence of Lemma~\ref{Lem_blow_ups}.
\end{proof}
\bigskip
\bigskip

\bibliography{references}	
\bibliographystyle{amsalpha}

\end{document}